\documentclass[smallextended]{svjour3}
\usepackage{latexsym}
\usepackage{amssymb,amsbsy,amsmath,amsfonts,amssymb,amscd}
\usepackage{mathrsfs}
\usepackage{epsfig, graphicx}
\usepackage{graphics,epsfig}
\usepackage[usenames,dvipsnames]{color}
\usepackage{enumerate}
\usepackage{epstopdf}
\usepackage[titletoc,toc,title]{appendix}
\usepackage[font=small,labelfont=bf]{caption}
\usepackage{subfigure}
\usepackage{hyperref}
\usepackage{float}
\usepackage{cite}
\smartqed

\usepackage{tikz}
\usetikzlibrary{arrows,positioning, arrows}

\usepackage[linesnumbered,ruled]{algorithm2e}
\usepackage[noend]{algpseudocode}

\newtheorem{innercustomproblem}{Problem}
\newenvironment{customproblem}[1]
  {\bigskip \renewcommand\theinnercustomproblem{#1}\innercustomproblem} \ \bigskip 
  {\endinnercustomproblem}

\newcommand{\eps}{\varepsilon}

\newcommand{\half}{\frac{1}{2}}
\newcommand{\energy}{\mathcal{E}}

\newcommand{\dW}{d_\mathcal{W}}
\newcommand{\Dt}{{\Delta t}}

\newcommand{\Dtime}{1}
\newcommand{\prox}{\text{Prox}}
\newcommand{\proj}{\text{Proj}}
\newcommand{\RR}{\mathbb{R}}

\newcommand{\Amat}{\mathsf{A}}
\newcommand{\Bmat}{\mathsf{B}}
\newcommand{\Dmat}{\mathsf{D}}
\newcommand{\Imat}{\mathsf{I}}
\newcommand{\Kron}{\otimes}
\newcommand{\indi}{\mathfrak{i}}
\newcommand{\Smat}{\mathsf{S}}
\newcommand{\rd}{\mathrm{d}}
\newcommand{\Rd}{{\mathbb{R}^d}}
\newcommand{\R}{\mathbb{R}}
\newcommand{\C}{\mathcal{C}}

\newcommand{\wsto}{\stackrel{*}{\rightharpoonup}}

\newcommand{\dx}{{{  (\Delta x)^d}}}
\newcommand{\brho}{\overline{\rho}}
\newcommand{\grad}{\nabla}
\newcommand{\F}{\mathcal{F}}
\newcommand{\G}{\mathcal{G}}

\def\bm{\overline{m}}
\def\P{{\mathcal P}}

\def\H{{\mathcal H}}

\DeclareMathOperator*{\argmin}{arg\,min}

 \renewcommand{\descriptionlabel}[1]{\hspace{\labelsep}\normalfont{#1}}

\newcommand{\average}[1]{\left\langle#1\right\rangle}

\graphicspath{{../}}

\begin{document}
\title{Primal dual  methods for  Wasserstein gradient flows}
\author{Jos\'e A. Carrillo \and  Katy Craig \and Li Wang \and  Chaozhen Wei }  
\institute{J.A. Carrillo \at Mathematical Institute, University of Oxford, Oxford OX2 6GG, UK. \\\email{carrillo@maths.ox.ac.uk.}  
\and  K. Craig \at Department of Mathematics, University of California, Santa Barbara, CA 93117, USA. \\\email{kcraig@math.ucsb.edu.} 
\and L. Wang \at School of Mathematics, University of Minnesota, Twin Cities, MN 55455, USA. \\
\email{wang8818@umn.edu.}
\and Chaozhen Wei \at Department of Mathematical Sciences, Worcester Polytechnic Institute, Worcester, MA 01609, USA. \\\email{cwei4@wpi.edu.}}
\date{\today }

\maketitle

\begin{abstract}
Combining the classical theory of optimal transport with modern operator splitting techniques, we develop a new numerical method for nonlinear, nonlocal partial differential equations, arising in models of porous media, materials science, and biological swarming. Our method proceeds as follows: First, we discretize in time, either via the classical JKO scheme or via a novel Crank-Nicolson type method we introduce. Next, we use the Benamou-Brenier  dynamical characterization of the Wasserstein distance to reduce computing the solution of the discrete time equations to solving fully discrete minimization problems, with strictly convex objective functions and linear constraints. Third, we compute the minimizers by applying a recently introduced, provably convergent primal dual splitting scheme for three operators \cite{Yan17}.

By leveraging the PDEs' underlying variational structure, our method  overcomes stability issues present in   previous numerical work built on explicit time discretizations, which suffer due to the equations' strong nonlinearities and degeneracies. Our method  is also naturally positivity and mass preserving and, in the case of the JKO scheme, energy decreasing.   We prove that  minimizers of the fully discrete problem converge to minimizers of the spatially continuous, discrete time problem as the spatial discretization is refined. We conclude with simulations of nonlinear PDEs and Wasserstein geodesics in one and two dimensions that illustrate the key properties of our approach, including higher order convergence  our novel Crank-Nicolson type method, when compared to the classical JKO method.
\end{abstract}



Communicated by Eitan Tadmor
 


\section{Introduction} 
Gradient flow methods are classical techniques for the analysis and numerical simulation of partial differential equations. Historically, such methods were exclusively based on gradient flows arising from a Hilbert space structure, particularly $L^2(\Rd)$, but since the work of Jordan, Kinderlehrer, and Otto in the late 90's \cite{Otto, O2, JKO98}, interest has emerged in a range of nonlinear, nonlocal partial differential equations that are gradient flows in the \emph{Wasserstein metric},
\begin{align} \label{maineqn}
\begin{cases}
\partial_t \rho = \grad \cdot (\rho \grad V) + \grad \cdot (\rho \grad W*\rho) + \alpha \Delta \rho^m , \quad x \in \Omega \subseteq \Rd, \quad V,W:\Omega \to \R ,
\\  \rho(x,0) = \rho_0(x)\,, \qquad \qquad \qquad \qquad  \qquad \qquad \hspace{.4cm}  m \geq 1, \quad \alpha \geq 0.
\end{cases}
\end{align}
When $\Omega \neq \Rd$, we consider no-flux boundary conditions.

Equations of this form arise in a number of physical and biological applications, including models in granular media \cite{BCP,Toscani00, CMV03, cmcv-06}, material science \cite{HP05}, and biological swarming \cite{CFTV10, KCBFL13,BCCD16}. Furthermore, many well-known equations may be written in this way: when $V = W = 0$ and $\alpha =1$, equation (\ref{maineqn}) reduces to the heat equation ($m=1$),   porous medium equation ($m>1$), and fast diffusion equation ($m<1$)  \cite{Vazquez07}. In the presence of a drift potential $V$, it becomes a  Fokker-Planck equation ($m=1$) or nonlinear Fokker-Planck equation ($m>1$), as used in models of tumor growth  \cite{Tang13, PQV14}. When the interaction potential $W$ is given by a repulsive-attractive Morse or power-law potential,
\begin{align} \label{repattW}
W(x) &= -C_a e^{- |x|/l_a} + C_r e^{-|x|/l_r} , \quad C_r/C_a<(l_r/l_a)^{-d}, 0 < l_r< l_a, 0< C_a<C_r , \\
W(x) &= \frac{|x|^a}{a} - \frac{|x|^b}{b}, \quad -d < b  < a , \nonumber
\end{align}
we recover a range of nonlocal interaction models, which are repulsive at short length scales and attractive at long length scales \cite{TBL06,BCLR13,BCLR13-2,CDM}. When $W = (\Delta)^{-1}$, the Newtonian potential, we have the Keller-Segel equation and its nonlinear diffusion variants  \cite{KellerSegel1971, BDP06,blanchet2012functional,CCH1,CCH2,CHMV,CCY}. Finally, as the diffusion exponent $m \to +\infty$, we recover congested aggregation and drift equations arising in models of pedestrian crowd dynamics and shape optimization problems \cite{CKY17, LTWZ18a, MRS, MRSV,BuChTo2016,FrLi2016}. 

In order to describe the gradient flow structure of equation (\ref{maineqn}), we begin by rewriting it as a continuity equation in $\rho(x,t)$ for a velocity field $v(x,t)$,
\begin{equation} \label{eqn:000}
\begin{cases}
\partial_t \rho = - \nabla \cdot (\rho v) := \nabla \cdot \left[ \rho \nabla \left(\alpha U_m'(\rho) + V + W* \rho \right)  \right],
\\  \rho(x,0) = \rho_0(x)\,,
\end{cases}
\quad  U_m(s) =\begin{cases} s \ln(s) &\text{ for } m = 1 , \\
\frac{s^m}{m-1} &\text{ for } m >1 . \end{cases}
\end{equation}
In this form, two key properties of the equation become evident: it is positivity preserving and conserves mass. In what follows, we will always consider nonnegative initial data, and we will typically renormalize so that the mass of the initial data equals one, i.e., $\rho_0 \in \P_{ac}(\Omega)$, where $ \P_{ac}(\Omega)$ is the set of probability measures on $\Omega$ that are absolutely continuous with respect to Lebesgue measure. Furthermore, as our objective is to develop a numerical method for these equations, we will exclusively consider the case when $\Omega$ is a bounded domain. Throughout, we commit a mild abuse of notation and identify all such probability measures with their densities, $d \rho(x) = \rho(x) \rd x$.

 As discovered by Otto \cite{O2}, given an energy $\energy: \P_{ac}(\Omega) \to \R \cup \{ +\infty\}$, we may formally define its gradient with respect to the Wasserstein metric $\dW$ using the formula
\[ \grad_{\dW} \energy(\rho) = - \grad \cdot \left( \rho \grad  \frac{\delta \energy}{\delta \rho} \right)  \,. \]
(See Section \ref{sec:dynamicJKO} for the definition of the Wasserstein metric $\dW$.)  In this way, gradient flows of $\energy$, $\partial_t \rho = - \grad_{\dW} \energy(\rho)$, correspond to solutions of the continuity equation with velocity $v = - \nabla \frac{\delta \energy}{ \delta \rho }$. In particular, equation (\ref{eqn:000}) is the gradient flow of the energy
\begin{equation} \label{eqn:energy}
\energy (\rho) = \int_{\Omega} \left[\alpha U(\rho(x)) + V(x) \rho(x) \right]\rd x + \half \int_{\Omega \times \Omega} W(x-y) \rho(x) \rho(y) \rd x \rd y\,.
\end{equation}
Differentiating the energy (\ref{eqn:energy}) along solutions of (\ref{eqn:000}), one formally obtains that the energy is decreasing along the gradient flow
\begin{equation} \label{energydissipation1}
\frac{d}{dt} \energy(\rho)(t) = - \int_{\RR^d} |v(t,x)|^2 \rho(t,x) \rd x\, ,
\end{equation}
which coincides with the theoretical interpretation of gradient flows as solutions that evolve in the direction of \emph{steepest descent} of an energy, where the notion of steepest descent is induced by the Wasserstein metric structure.

A key feature of equations of the form (\ref{eqn:000}) is the competition between repulsive and attractive effects. For repulsive-attractive interaction kernels $W$, as in equation (\ref{repattW}), these effects can arise purely through   nonlocal interactions, leading to rich structure of the steady states \cite{FHK11,BLL12,BCLR13-2,BKSUV15,CDM}. For purely attractive interaction kernels $W$, as in the Keller-Segel equation, the competition instead arises from the combination of nonlocal interaction with diffusion. In this case, different choices of interaction kernel $W$, diffusion exponent $m$, and initial data $\rho_0$ can lead to widely different behavior---from bounded solutions being globally well-posed to smooth solutions blowing up in finite time \cite{BDP06,blanchet2012functional,CCH1,CCH2,CHMV,CCY}.

\begin{figure}[!h]
\includegraphics[width=16.5cm,trim={1.5cm 0cm .5cm 0cm},clip ]{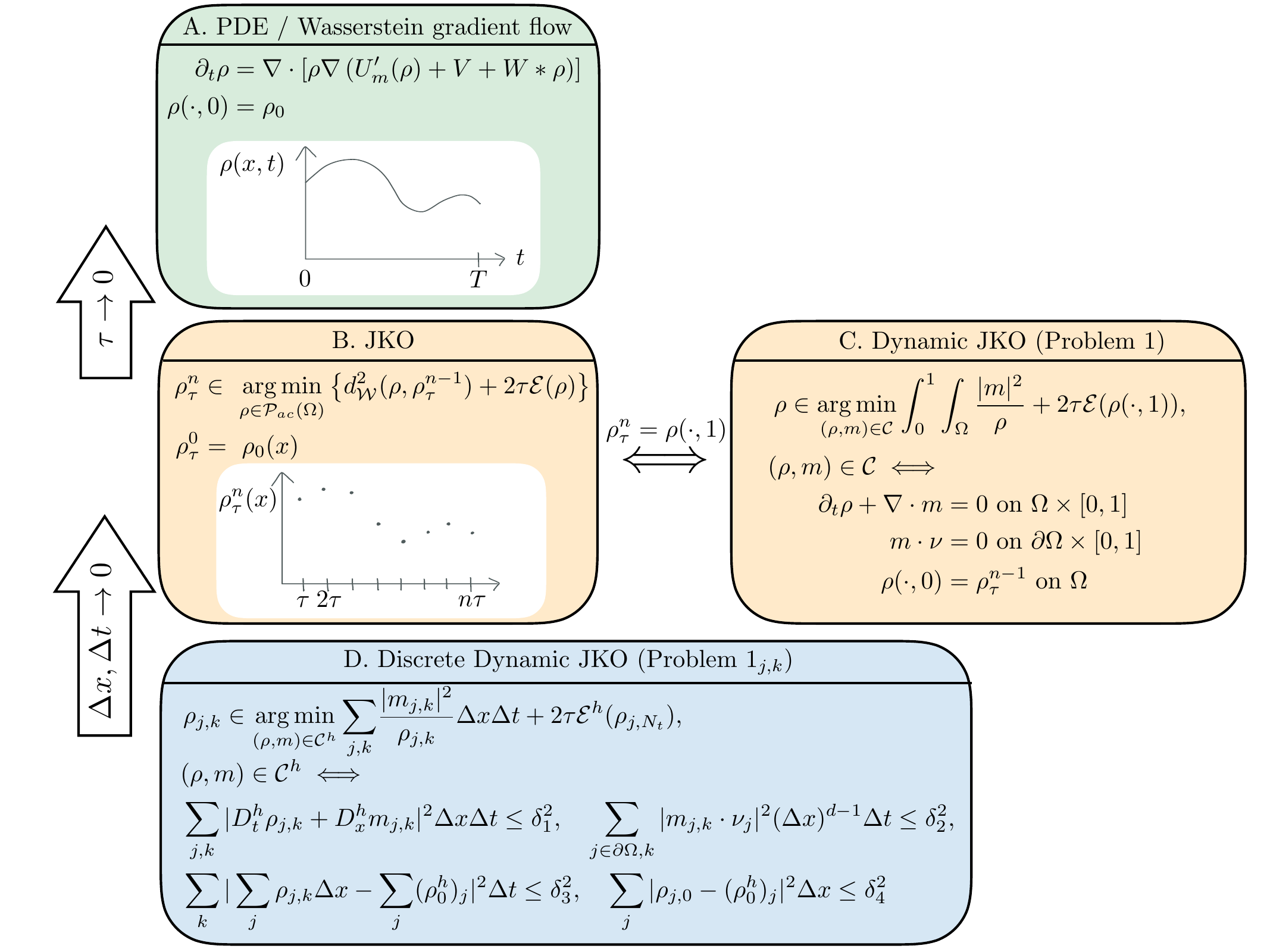}
\caption{Levels of discretization: $\tau$ is the outer JKO time step, $\Delta t$ is the inner time step, and $\Delta x$ is the spatial discretization. }
\label{fig:level}
\end{figure}

\subsection{Summary of numerical approach}
The goal of the present work is to develop new numerical approach for partial differential equations of the form (\ref{maineqn}) that combine gradient flow methods with modern operator splitting techniques. Our approach applies to equations of this form with any combination of diffusion $\alpha U'_m(\rho)$ $(\alpha \geq 0$), drift $V$, or interaction $W*\rho$ terms---in particular, it is not necessary for diffusion to be present in order for our scheme to converge. 

The main idea of our approach is to discretize the PDE/Wasserstein gradient flow at two levels. First, we consider  a  time discretization of the gradient flow with time step $\tau$ (see Figure \ref{fig:level}.B), either given by the classical JKO scheme (equation (\ref{eqn:JKO-000}) below) or a new Crank-Nicolson inspired variant (equation \eqref{2nd} below). This reduces computation of the gradient flow to solving a sequence of infinite dimensional minimization problems. Then, we  consider a dynamical reformulation of these minimization problems, stemming from  Benamou and Brenier's dynamic characterization of the Wasserstein metric, by which the problem becomes the minimization of a strictly convex integral functional subject to a linear PDE constraint (see Figure \ref{fig:level}.C). At this level, the problem remains continuous in space and time. We conclude by considering a further discretization of the problem, with inner time step $(\Delta t)$ and spatial discretization $(\Delta x)$, by taking piecewise constant approximations of the functions and using a finite difference approximation of the PDE constraint (see Figure \ref{fig:level}.D). In this final, fully discrete form, we then compute the minimizer using modern operator splitting techniques, applying Yan's recent extension of the classical primal dual algorithm for minimizing sums of three convex functions \cite{Yan17}.

Our paper is organized as follows. In section \ref{previousworksection}, we discuss the relationship between  our numerical approach  and previous work. In section \ref{contributionsection}, we summarize our contribution. In section \ref{sec:nummethod}, we describe   the details of our numerical method. Along with numerically simulating Wasserstein gradient flows, our method also provides, as a special case, a new method for computing Wasserstein geodesics and the Wasserstein distance between probability densities; see Remark \ref{W2georemark}. In section \ref{sec:convergence}, we prove that, provided  a smooth, positive solution of the continuum JKO scheme exists and the energy corresponding to the PDE is sufficiently regular, then minimizers of the fully discrete problem exist (Theorem \ref{minimizersexist}), the objective functions of the discrete problems $\Gamma$-converge to the objective function of the continuum problem (Theorem \ref{gammaconvergence}), and thus  solutions of the fully discrete scheme converge, up to a subsequence, to  a solution of the continuum scheme (Theorem \ref{convergence}). As a special case, we also recover convergence of a numerical method for computing Wasserstein geodesics, similar to that introduced by Papadakis, P\'eyre, and Oudet \cite{PPO14}. Finally, in section \ref{sec:examples}, we provide several numerical simulations illustrating our approach in both one and two dimensions, computing Wasserstein geodesics, nonlinear Fokker-Planck equations, aggregation diffusion equations, and other related equations.

\subsection{Details of   approach and comparison with previous work} \label{previousworksection}

\subsubsection{Classical numerical PDE methods}
We now compare our approach to existing numerical methods.
 Perhaps the most common numerical approach for equations of the form (\ref{maineqn}) is to consider the equation as an advection-diffusion equation  and apply classical  finite difference, finite volume, or Galerkin discretizations \cite{Filbet06, CK08, CCH15, LWZ17,BCH}. However, when such methods are based on explicit time discretizations, they    suffer from stability constraints  due either to the degeneracy of the diffusion  (when $m>1$)  or the non-locality from the  interaction potential $W$. (See for instance the mesa problem  \cite{LTWZ18b}.) Implicit time discretizations, on the other hand, are computationally intensive, due to the difficulty of matrix inversion, even when the implicit steps are solved by smart iterative methods to avoid the high computation cost of convolution \cite{BCH}.

Another common approach is to  leverage structural similarities between (\ref{eqn:000}) and equations from fluid dynamics  to develop particle methods  \cite{CB16, CCCC18,BLL12,CarrilloChoiHauray,CDFLS11,CHPW17,CPSW16,CraigTopalogluHeight,MO17,OM17}. Until recently, the key limitation of such methods has been developing approaches to incorporate diffusion. Following the analogy with the Navier-Stokes equations, stochastic particle methods have been proposed in the case of linear diffusion ($m=1$) \cite{Jabin,JW,Liu1,Liu2}. More recently the first two authors and Patacchini developed a deterministic blob method for linear and nonlinear diffusion ($m \geq 1$) \cite{CCP18}. On one hand, particle methods naturally conserve mass and positivity, and they can also be designed to respect the underlying gradient flow structure of the equation, including the energy dissipation property (\ref{energydissipation1}). On the other hand, a large number of particles is often required to resolve finer properties of solutions.

In constrast with such classical  methods, our method introduces an auxiliary momentum variable $m$ and an additional inner layer of time discretization, which enlarges the dimension of the problem. However, as later pointed out in \cite{LLW20}, the inner layer of time can be discretized with just one step without violating the overall first order accuracy, there completely eliminating the additional cost introduced by the inner layer. Another major advantage of our approach is that, by reforming the PDE problem into an optimization problem, we obtain unconditional stability (for the JKO discretization, see equation (\ref{eqn:JKO-000}) below) while avoiding the inversion of a full matrix in the general implicit setting, which is extremely expensive, especially in higher dimensions; see for instance \cite{BCH}. 
Finally, compared to other implicit methods, such as the backward Euler method, the sub optimization problems can be solved independently at each gridpoint, and therefore are massively parallelizable and suitable for high dimensional problems.

\subsubsection{Variational methods}

Compared to the classical numerical PDE approaches described in the previous section, a more modern class of numerical methods  leverages the gradient flow structure of (\ref{maineqn})  to approximate solutions of the PDE by solving a sequence of minimization problems. This is the approach we take in the present work. Originally introduced by Jordan, Kinderlehrer, and Otto as a technique for computing solutions of the Fokker-Planck equation (equation (\ref{maineqn}), $W=0$, $m=1$)  \cite{JKO98}, this scheme approximates the solution $\rho(x,t)$ at time $t$ by solving the following sequence of  $n$ minimization problems with time step $\tau = t/n$,
\begin{equation} \label{eqn:JKO-000}
\rho_\tau^{n} \in  \argmin_{\rho\in \P_{ ac}(\Omega)} \left\{ \dW^2 (\rho, \rho_\tau^{n-1}) + 2 \tau \energy(\rho) \right\} \,, \quad \rho_\tau^0 = \rho_0(x) .
\end{equation}
The JKO scheme is precisely the analogue of the implicit Euler method in the infinite dimensional Wasserstein space. The constraint  $\rho \in \P_{ ac}(\Omega)$ ensures that the method is positivity and mass preserving, and the fact that $\dW^2(\rho,\rho^n) \geq 0$ ensures the energy decreasing along the scheme, $\energy(\rho^{n+1}_\tau) \leq \energy(\rho^{n}_\tau)$.

Under sufficient assumptions on the underlying domain $\Omega$, drift potential $V$, interaction potential $W$, and initial data $\rho_0$ (see section \ref{sec:dynamicJKO}), the solution of the JKO scheme $\rho^n_\tau$ converges to the solution $\rho(x,t)$ of the partial differential equation (\ref{maineqn}), with a first order rate in terms of the time step $\tau = t/n$ \cite[Theorem 4.0.4]{AGS},
 \[ \dW(\rho_n(\cdot), \rho(\cdot, t)) \leq  C \tau . \]
 
 In our numerical simulations, we observe that this discretization error dominates other errors in our numerical method; see sections \ref{sec:PME} and \ref{sec:FP}. Consequently, we also introduce a new  time discretization, in analogy with the Crank-Nicolson method
  \begin{equation} \label{2nd}
\rho^{n+1} \in  \argmin_{\rho\in \P_{ ac}(\Omega)} \left\{  \dW^2 (\rho, \rho^n) + \tau\energy(\rho) + \tau  \int_\Omega \frac{\delta \energy}{\delta \rho} (\rho^n) \rho   \right\}\,.
\end{equation}
The connection between the above scheme and the classical Crank-Nicolson discretization can be seen by considering the optimality conditions for (\ref{2nd}):
 \[ \frac{1}{\tau} (\rho^{n+1}  - \rho^n) = \frac{1}{2} \nabla \cdot  \left(   \rho^{n+1} \nabla \left( \frac{\delta \energy(\rho^{n+1})}{\delta \rho} + \frac{\delta \energy (\rho^n)}{\delta \rho}\right)  \right) . \]
Like the JKO scheme, our Crank-Nicolson inspired method is also positivity and mass-preserving, though it is not energy decreasing. In Figures \ref{fig:convg}, \ref{fig:convgsmooth}, and \ref{fig:FK1D-conv} of our numerics section, we conduct a preliminary analysis of the rate of convergence of this method, which verifies that it is indeed higher order than the JKO scheme. As the goal of the present work is primarily the development of fully discrete numerical schemes, we leave a thorough analysis of the rate of convergence of our Crank-Nicolson inspired method as $\tau \to 0$ to  future work. 

On one hand, our Crank-Nicolson inspired method (\ref{2nd}) is not the first higher order method proposed for metric space gradient flows: Matthes and Plazotta developed a provably second order scheme for general metric space gradient flows by generalizing the backward differentiation formula \cite{matthes2017variational}.  The Matthes-Plazotta method, however, requires two evaluations of the Wasserstein distance at each outer time step and thus is less practical for our purpose of numerically computing gradient flows in higher dimensions. Another method was introduced by Legendre and Turinici \cite{LT17} based on the midpoint method. This method can be reformulated as the classical JKO step with half time step followed by an extrapolation. This extrapolation step could be implemented  by solving the corresponding continuity equation either explicitly or implicitly; however solving the equation explicitly could potentially violate conservation of positivity, while solving it implicitly would require an additional matrix inversion.    Another higher order variational method was also proposed in \cite{Lagu}, which resembles explicit Runge-Kutta methods and, again, require two or more evaluations of the Wasserstein distance at each outer time step.  

\subsubsection{Numerical methods for the Wasserstein distance}
To use either the classical JKO scheme (\ref{eqn:JKO-000}) or our new  Crank-Nicolson inspired scheme (\ref{2nd}) as a basis for numerical simulations, one must first develop a fully discrete approximation of the minimization problem at each step of the scheme. Here, the main numerical difficulty arises in approximating the Wassserstein distance, and there are several different approaches for dealing with this term. 
First, one can reformulate the Wasserstein distance in terms of a Monge-Amp\'ere equation with nonstandard boundary conditions \cite{GM96, BFO14}, though difficulties arise due to the lack of a comparison principle \cite{hamfeldt2018viscosity}. Second, one can reframe the problem as a classical $L^2(\Rd)$ gradient flow at the level of diffeomorphisms \cite{GT06,BCC08,CM09, CRW16,CDMM18}, but to pursue this approach, one has to overcome complications arising from the underlying geometry and the structure of the PDE system for the diffeomorphisms. Third, one can discretize the Wasserstein distance term as a finite dimensional linear program, overcoming the lack of strict convexity of the objective function by adding a small amount of entropic regularization \cite{cuturi13, Benamou15, CPSV16}. (For a detailed survey of  computational optimal transport, we  refer the reader to the recent book by P\'eyre and Cuturi for   \cite{PC18}.)

A fourth approach for computing the Wasserstein distance, and the one which we   develop in the present work, is to consider a dynamic formulation  due to Benamou and Brenier \cite{BB00}. This reframes the problem as a  strictly convex optimization problem with linear PDE constraints, which can  be discretized using Benamou and Brenier's original  augmented Lagrangian method ALG2 or, more generally, a range of modern proximal splitting methods, as shown by Papadakis, Peyre, and Oudet  \cite{PPO14}. (See also \cite{briceno2018proximal,briceno2019implementation} for related work on mean field games.) Adding an additional Fisher information term in this dynamic formulation (in analogy with entropic regularization) has also been explored in \cite{LYO17}. 

Only recently have these above approaches for computing the Wasserstein distance been integrated with the JKO scheme (\ref{eqn:JKO-000}) in order to simulate   partial differential equations of the form (\ref{maineqn}). The Monge-Amp\'ere approach extends naturally, though the presence of a diffusion term $\alpha U'_m(\rho)$ for $\alpha >0$ is required to enforce convexity constraints at the discrete level \cite{BCMO16}. Similarly, entropic regularization (or the addition of a Fisher information term) vastly accelerates the computation of gradient flows, but at the level of the partial differential equation, this corresponds to introducing numerical diffusion, which may disrupt the delicate balance between aggregation and diffusion inherent in PDEs of this type  \cite{CPSV16, CDPS17,LYO17}.  
Finally, Benamou and Brenier's dynamic reformulation of the Wasserstein distance has also been adopted in recent work to approximate gradient flows \cite{BCL16}. A key benefit of this latter approach when compared to  entropic regularization   is that it leads to an optimization problem in $N_x^d \times N_t$ variables, where $N_x$ and $N_t$ are the number of spatial and temporal gridpoints, whereas the latter leads to an optimization problem in $N_x^{2d}$ variables.

In the present work, we further  develop this last approach, using  Benamou and Brenier's dynamic reformulation of the Wasserstein distance to simulate Wasserstein gradient flows, via  both the classical JKO scheme (\ref{eqn:JKO-000}) and our new Crank-Nicolson inspired scheme (\ref{2nd}). This leads to a sequence of minimization problems (Figure \ref{fig:level}.C), which we discretize (Figure \ref{fig:level}.D) and then solve using  a modern primal dual three operator splitting scheme due to Yan \cite{Yan17}, instead of the classical ALG2 method. See section \ref{sec:nummethod} for a detailed description of our approach.

Due to the fact that we use operator splitting methods to compute the minimizer in Benamou and Brenier's dynamic formulation of the Wasserstein distance,  our work can be seen as an extension of previous work by Papadakis, Peyre, and Oudet \cite{PPO14}, which applied   similar two operator splitting schemes to simulate the Wasserstein distance.  However, there are a few key differences between our approach and previous work. First, we are able to implement the primal dual splitting scheme in a manner that does not require matrix inversion of the finite difference operator, which reduces the computational cost. Second, we succeed in obtaining the exact expression for the proximal operator, which allows our method to be truly positivity preserving, while other similar methods are only positivity preserving in the limit as $\Delta x, \Delta t \to 0$; see Remark \ref{comparisonremark}. Third, instead of imposing  the linear PDE constraint in Benamou and Brenier's dynamic reformulation \emph{exactly}, via a finite difference approximation,  we allow the linear PDE constraint to hold up to an error of order $\delta>0$, which can be tuned according to the spatial discretization $(\Delta x)$, the inner temporal discretization $(\Delta t)$, and the outer time step $\tau$ to respect the order of accuracy of the finite difference approximation; see Remark \ref{relaxpderemark}. Numerically, this allows our method to converge in fewer iterations, without any reduction in accuracy, as demonstrated in Figure \ref{fig:optimaldelta}. From a theoretical perspective, the fact that we only require the PDE constraint to hold up to an error of order $\delta> 0$ makes it \emph{possible} to prove convergence of minimizers of the fully discrete problem to minimizers of the JKO scheme (\ref{eqn:JKO-000}), since minimizers of the fully discrete problem always exist for $\delta  >0$, which is not the case when the PDE constraint is enforced exactly ($\delta = 0$); see Remark \ref{deltaremark} and Theorem \ref{minimizersexist}.

\subsection{Contribution} \label{contributionsection}
The main components of our numerical method for computing solutions to (\ref{maineqn}) are:
\begin{enumerate}[(a)]
\item an outer time discretization, of either  JKO (\ref{eqn:JKO-000}) or Crank-Nicolson type (\ref{2nd}) (Figure \ref{fig:level}.B)  
\item a dynamic interpretation of the Wasserstein distance (Figure \ref{fig:level}.C), which when discretized via finite difference approximations leads to a sequence of constrained optimization problems (Figure \ref{fig:level}.D)
\item an application of modern three operator splitting schemes for solving these optimization problems.
\end{enumerate}
Our main contributions are:
\begin{itemize}
\item Unlike classical explicit methods, our JKO type method is unconditionally stable. Unlike classical implicit methods, it achieves this stability without an expensive matrix inversion.
\item In practice, we observe that our Crank-Nicolson type  method performs even better than our JKO type method, in terms of rate of convergence with respect to the outer time step (see Figures \ref{fig:convg}, \ref{fig:convgsmooth}, and \ref{fig:FK1D-conv}). We leave a thorough analysis of the rate of this convergence of this method to future work.
\item By formulating our optimization problem with a linear \emph{inequality} constraint instead of a linear \emph{equality} constraint, our algorithm converges in fewer iterations when compared to related algorithms for Wasserstein geodesics; see Remark \ref{relaxpderemark} and Figure \ref{fig:optimaldelta}.
\item We prove convergence of our fully discrete method (Figure \ref{fig:level}.D) to the JKO scheme (Figure \ref{fig:level}.B and C) as the spatial discretization and inner time discretization go to zero. 
\end{itemize}

\section{Numerical Method} \label{sec:nummethod}
\subsection{Dynamic formulation of JKO scheme} \label{sec:dynamicJKO}

As described in the previous section, our numerical method for computing the JKO scheme is based on the following  dynamic reformulation of the Wasserstein distance due to Benamou and Brenier \cite{BB00}:
\begin{align} \label{BB1}
&\dW(\rho_0,\rho_1)   = \inf_{(\rho,v) \in \C_0}  \left\{ \int_0^1 \int_\Omega |v(x,t)|^2 \rd \rho(x,t)  \rd t \right\}^{1/2} ,
\end{align}
where $(\rho,v)\in AC(0,1;\P(\Omega) )\times  L^1(0,1; L^2(\rho))$ belongs to the constraint set $\C_0$ provided that
\begin{align} \label{BB2}
 \partial_t \rho + \nabla \cdot (\rho v) = 0 &\text{ on } \Omega \times[0,1] \\
  (\rho v) \cdot \nu = 0   &\text{ on } \partial \Omega \times [0,1] ,\label{BB3p5} \\
 \rho(\cdot,0) = \rho_0, \ \rho(\cdot, \Dtime) = \rho_{\Dtime} &\text{ on } \Omega, \label{BB3}
\end{align}
where $\nu$ is the outer unit normal on the boundary of the   domain $\Omega$.
A curve $\rho$ in $\P(\Omega)$ is \emph{absolutely continuous} in time, denoted $\rho \in AC(0,1;\P(\Omega) )$,  if there exists $w \in L^1(0,1)$ so that $\dW(\rho(\cdot, t_0), \rho(\cdot, t_1)) \leq \int_{t_0}^{t_1} w(s) \rd s$ for all $0 < t_0 \leq t_1 < 1$. The PDE constraint   (\ref{BB2}-\ref{BB3p5}) holds in the   duality with smooth test functions on $\Rd \times [0,1]$, i.e. for all $f \in C^\infty_c(\Rd \times [0,1])$,
\begin{align*}
\int_0^1  \int_\Omega \left[ \partial_t f(x,t) \rho(x,t) + \grad f(x,t) \cdot  v(x,t) \rho(x,t) \right] \rd x \rd t + \int_\Omega f(x,0)\rho_0(x) - f(x,1)\rho_1(x)\rd x = 0\,.
\end{align*}

This dynamic reformulation reduces the problem of finding the Wasserstein distance between any two measures to identifying  the curve  in $\P(\Omega)$ that connects them with minimal kinetic energy. However, the objective function   (\ref{BB1}) is not strictly convex, and the PDE constraint (\ref{BB2}) is nonlinear. For these reasons, in Benamou and Brenier's original work, they restrict their attention to the case $\rho(\cdot, t) \in \P_{ac}(\Omega)$ and introduce the momentum variables $m = v \rho$,  in order to rewrite (\ref{BB1}) as
\begin{eqnarray}\label{BB4}
\dW^2(\rho_0,\rho_1)  =  \min_{(\rho, m) \in \C_1}  \int_0^1 \int_\Omega \Phi(\rho(x,t), m(x,t))\rd x \rd t,
\end{eqnarray}
where
\begin{align} \label{Phidef}
\Phi(\rho, m) = \left\{  \begin{array}{cc}  \frac{\|m \|^2}{ \rho} & \textrm{ \hspace{-1.45cm} if } \rho>0\,, \\ 0 & \textrm{\hspace{0.2cm} if } (\rho, m)=(0,0) \,,
\\ +\infty & \textrm{ \hspace{-1cm} otherwise\,,}  \end{array} \right.
\end{align}
and $(\rho,m) \in AC(0,1; \P_{ac}(\Omega)) \times L^1(0,1;L^2(\rho^{-1}))$ belong to the constraint set $\C_1$ provided that
\begin{align*} 
 \partial_t \rho + \nabla \cdot m = 0 &\text{ on } \Omega \times[0,1]
 \\ m \cdot \nu = 0   &\text{ on } \partial \Omega \times [0,1] . 
 \\ \rho(\cdot,0) = \rho_0, \ \rho(\cdot, \Dtime) = \rho_{\Dtime} &\text{ on } \Omega . 
\end{align*}
After this reformulation, the integral functional  
\begin{align} \label{integralJ}
(\rho,m) \mapsto \int_0^1 \int_\Omega \Phi(\rho,m)
\end{align}
 is strictly convex along linear interpolations and lower semicontinuous with respect to weak-* convergence \cite[Example 2.36]{ambrosio2000functions}, and the PDE constraint is linear.
As an immediate consequence, one can conclude that minimizers are unique.
Furthermore, for any $\rho_0, \rho_1 \in \P_{ac}(\Omega)$, a direct computation shows that the minimizer $(\bar \rho, \bar m)$   is given by  the Wasserstein geodesic from $\rho_0$ to $\rho_1$,
\begin{align} \label{W2geoform}
&\bar \rho(x,t) =T_t \# \rho_0 , \  \bar v(x,t) = T \circ T_t^{-1}(x)  - T_t^{-1}(x) , \ \bar m  =  \bar v  \bar \rho , \ \text{ for }  T_t(x) := (1-t)x + t T(x),
\end{align}
where $T$ is the optimal transport map from $\rho_0$ to $\rho_1$. (See \cite{Villani03a, AGS,santambrogio2015optimal} for further background on optimal transport.)
Consequently, given any minimizer $(\bar \rho, \bar m)$ of (\ref{BB4}), we can recover the optimal transport plan $T$   via the following formula:
\begin{align}
T(x) =x + \bar v(x,0)= x + \bar m(x,0)/ \bar \rho(x,0) .
\end{align}

Building upon Benamou and Brenier's  dynamic reformulation of the Wasserstein distance, one can also consider a dynamic reformulation of the JKO scheme \eqref{eqn:JKO-000}.  In particular, substituting (\ref{BB4}) in \eqref{eqn:JKO-000} leads to the following dynamic JKO scheme:
\begin{customproblem}{1}[Dynamic JKO]  \label{dynamicJKO}
 Given $\tau >0$, $\energy$, and $\rho_0$, solve the constrained optimization problem,
\begin{align*} 
\inf_{(\rho,m) \in \C}  \int_0^1 \int_\Omega \Phi(\rho(x,t), m(x,t))\rd x \rd t + 2\tau \energy(\rho(\cdot,1)) ,
\end{align*}
where $(\rho,m) \in AC(0,1; \P_{ac}(\Omega)) \times L^1(0,1;L^2(\rho^{-1}))$ belong to the constraint set $\C$ provided that
\begin{align}
 \partial_t \rho + \nabla \cdot m = 0 &\text{ on } \Omega \times[0,1] ,   \quad m \cdot \nu = 0   \text{ on } \partial \Omega \times [0,1] , \quad \text{ and } \rho(\cdot,0) = \rho_0  \text{ on } \Omega . \label{BB8}
\end{align}
\end{customproblem}
\noindent We emphasize that the requirement $\rho(x,t) \in \P_{ac}(\Omega)$ for all $ t\in [0,1]$ ensures that $\rho(x,t) \geq 0$.

\begin{remark}[Wasserstein geodesics] \label{W2georemark}
Note that for any $\rho_1 \in \P_{ac}(\Omega)$, we may take
\begin{align} \label{Erho1}
\energy(\rho(\cdot, 1)) = \mathcal{G}_{\rho_1}(\rho(\cdot, 1)):= \begin{cases} 0 &\text{ if } \rho(\cdot, 1) = \rho_1 , \\ +\infty &\text{ otherwise,} \end{cases}
\end{align}
in which case Problem \ref{dynamicJKO} reduces to the Benamou-Brenier formulation of the Wasserstein distance (\ref{BB4}). Consequently, the numerical method we develop for Problem \ref{dynamicJKO} offers, as a particular case, a provably convergent numerical method for computing the Wasserstein geodesic and Wasserstein distance between $\rho_0$ and $\rho_1$.
On one hand, there are many alternative methods for computing Wasserstein geodesics in Euclidean space. Indeed the many algorithms described in the introduction for computing the Wasserstein distance also provide an optimal transport plan, which can be linearly interpolated to give the Wasserstein geodesic \cite{GM96, BFO14, cuturi13, Benamou15, CPSV16, PC18}. On the other hand, our method is distinguished because it could be more naturally extended to variants of the Wasserstein metric built on the Benamou-Brenier formulation \cite{maas2011gradient,erbar2016gradient,carrillo2020landau}, as well as to Wasserstein geodesics on non-Euclidean manifolds, where the geodesic equations on the underlying manifold may no longer be explicit, so that one cannot pass directly from the optimal transport plan to the Wasserstein geodesic.
\end{remark}
 
\begin{remark}[existence and uniqueness of minimizers] \label{uniquenessmincont}
 If the underling domain $\Omega$ is convex and the energy $\energy$ is proper, lower semicontinuous, coercive, and $\lambda$-convex along generalized geodesics, and also satisfies  $ \{ \mu : \energy(\mu) < +\infty \} \subseteq \P_{ac}(\Omega)$, then,  for  $\tau>0$ sufficiently small, there exists a unique solution  to Problem \ref{dynamicJKO} \cite[Theorem 4.0.4, Theorem 8.3.1]{AGS}.   In particular, these assumptions are satisfied by the energy $\G_{\rho_1}$ (\ref{Erho1}), as well as by the drift-diffusion-interaction energy from the introduction (\ref{eqn:energy}), for $U$ as in equation (\ref{eqn:000}), $V, W \in C^2(\Omega)$. (See, for example, \cite[Section 9.3]{AGS} or \cite{craig2017nonconvex} for more general conditions on $U$, $V$,   $W$.)
\end{remark}

Thus, if we denote by $(\bar{\rho},\bar{m})$ the minimizer of Problem \ref{dynamicJKO}, then for $\tau>0$ sufficiently small,  the \emph{proximal map},
\[ J_\tau(\rho_0) := \brho(\cdot, 1) \ ,  \]
is well defined for all   $\rho_0 \in D(\energy)$. Furthermore, the energy decreases under the proximal map,
\begin{align} \label{proxEnergyDecrease}
\energy(J_\tau(\rho_0)) \leq \energy(\rho_0) ,
\end{align}
which can be seen by comparing the value of the objective function at the minimizer $(\brho,\bm)$ to the value of the objective function at  $(\rho(x,0), 0) \in \mathcal{C}$ and using that $\Phi(\rho,m) \geq 0$.

Given $\rho_0 \in D(\energy)$, if we recursively define the \emph{discrete time gradient flow sequence}
\begin{align} \rho^n_\tau = J_\tau(\rho^{n-1}_\tau) , \ \text{ for all } n \in \mathbb{N} ,   \label{discretegradflow}
\end{align}
then, taking $\tau = t/n$,  $\rho^n_\tau$ converges to $\rho(x,t)$, the gradient flow of the energy $\energy$ with initial data $\rho_0$ at time $t$, and under mild regularity assumptions on $\rho_0$, we have  
 \begin{align}  \dW(\rho_n(\cdot), \rho(\cdot, t)) \leq C \tau.  \label{firstorderJKO}
 \end{align}
 In this way, the classical JKO scheme provides a \emph{first order} approximation of the gradient flow  \cite[Theorem 4.0.4]{AGS}. In our numerical simulations, we observe that this discretization error dominates other errors in our numerical method; see sections \ref{sec:PME} and \ref{sec:FP}. For this reason, we   introduce the following new scheme,   inspired by the Crank-Nicolson method.
\begin{customproblem}{2}[Crank-Nicolson Inspired Dynamic JKO]  \label{hodynamicJKO}
 Given $\tau >0$, $\energy$, and $\rho_0$, solve the constrained optimization problem,
\begin{align*} 
\inf_{(\rho,m) \in \C}  \int_0^1 \int_\Omega \Phi(\rho(x,t), m(x,t))\rd x \rd t + \tau \energy(\rho(x,1)) + \tau \int_\Omega \frac{\delta \energy}{\delta \rho}(\rho(x,0)) \rho(x,1) \rd x ,
\end{align*}
where $(\rho,m) \in AC(0,1; \P_{ac}(\Omega)) \times L^1(0,1;L^2(\rho^{-1}))$ belong to the constraint set $\C$ provided that
\begin{align*}
 \partial_t \rho + \nabla \cdot m = 0 &\text{ on } \Omega \times[0,1] ,   \quad m \cdot \nu = 0   \text{ on } \partial \Omega \times [0,1] , \quad \text{ and } \rho(\cdot,0) = \rho_0  \text{ on } \Omega . 
\end{align*}
\end{customproblem}
In section \ref{sec:FP}, we provide numerical examples comparing the above method to the classical JKO scheme from Problem \ref{dynamicJKO}, illustrating that it achieves a higher order rate of convergence in practice (see Figure \ref{fig:convg}, Figure \ref{fig:convgsmooth}, and Figure \ref{fig:FK1D-conv}), in spite of the fact that that it lacks the energy decay property of Problem 1. Under what conditions a higher order analogue of inequality (\ref{firstorderJKO}) holds for the new scheme is an interesting open question that we leave to future work, as the main goal of the present work is the development of fully discrete numerical methods for computing minimizers of Problem \ref{dynamicJKO} and \ref{hodynamicJKO}. By iterating either of these minimization problems, as in equation (\ref{discretegradflow}), we obtain a   numerical method for simulating Wasserstein gradient flows.

\subsection{Fully discrete JKO}
We now turn to the discretization of the dynamic  JKO scheme, Problem \ref{dynamicJKO}, and the Crank-Nicolson inspired scheme, Problem \ref{hodynamicJKO}. We begin by noting that the Crank-Nicolson inspired Problem \ref{hodynamicJKO} can be rewritten in the same form as Problem \ref{dynamicJKO} by considering the energy
\begin{align} \label{hoenergy}
\mathcal{H}_{\rho_0}(\rho) := \frac{1}{2} \energy(\rho(x,1)) + \frac{1}{2} \int_\Omega \frac{\delta \energy}{\delta \rho}(\rho(x,0)) \rho(x,1) \rd x .
\end{align}
Using this observation, we will now describe our discretization of both problems simultaneously.

\subsubsection{Discretization of functions and domain} \label{discretefundom}
Given an $n$ dimensional hyperrectangle $S = \Pi_{I=1}^{n} [a_I,b_I]   \subseteq \R^{n}$, we  discretize it as a union of cubes $Q_i$, $i \in \mathbb{N}^n$, where in the $l$th direction, we suppose there are $N_l$ intervals of spacing  $({ \bf \Delta z})_l = (b_l-a_l)/N_l$:
\begin{align*} 
S  = \bigcup_{i : Q_i \subseteq S} Q_i   , \quad Q_{i} := \{ (z_1, z_2, \dots, z_n)  \in \R^n  :  z_l\in [(i_l-1) ({\bf  \Delta z})_l, i_l ({ \bf \Delta z})_l]  \ \forall l = 1, \dots, n \}.
\end{align*}
 \emph{Piecewise constant functions}  with respect to this discretization are given by
\begin{align*} 
 f^h := \sum_{i : Q_i \subseteq S}   f_{i} 1_{Q_{i}}, \text{ for } f_{i} \in \R  \text{ and }1_{Q_{i}}(z) = \begin{cases} 1 \text{ if }z \in Q_{i} \\ 0 \text{ otherwise.} \end{cases}
 \end{align*}

To discretize Problem \ref{dynamicJKO},  we take $S = \overline{\Omega} \times [0,1] \subseteq \R^{d +1}$, where $\overline{\Omega} = \Pi_{i=1}^{n} [a_i,b_i]  $. For any $i \in   \mathbb{N}^{d+1}$,   write $i = (j,k)$, for the spatial index $j \in   \mathbb{N}^d$ and the temporal index $k \in  \mathbb{N}$. We let $ N_x  \in \mathbb{N} $  denote the number of intervals in each spatial direction and $N_t \in \mathbb{N}$ denote the number of intervals in the temporal direction.  Take  ${\bf \Delta z} = ({\bf\Delta x}, \Delta t)$ for $({\bf \Delta x})_l = (\Delta x)   >0$ for all $l = 1, \dots, d$ and   $\Delta t  >0$.

We consider piecewise constant approximations $(\rho^h,m^h)$ of the functions $(\rho,m)$, with coefficients denoted by $(\rho_{j,k}, m_{j,k})$. For any $ (\rho,m) \in C(\overline{\Omega} \times [0,1])$, one such approximation is the \emph{pointwise piecewise} approximation $(\hat{\rho}^h, \hat{m}^h)$, obtained by defining the coefficients $(\hat{\rho}_{j,k}, \hat{m}_{j,k})$ to be the value of $(\rho,m)$ on a regular grid of spacing $(\Delta x) \times (\Delta t)$:
\begin{align} \label{pointwisepiecewise}
 \hat{\rho}_{j,k} := \rho(x_{j}, t_k), \  \hat{m}_{j,k} := m(x_j, t_k) , \qquad    
  (x_{j}, t_k) &= (\mathbf{\hat{x}} + (j-\mathbf{1})(\Delta x), \hat{t} + (k-1)(\Delta t)) \nonumber \\
   \mathbf{{\hat{x}}}  &\in \Pi_{l = 1}^d [0,  \Delta x ] , \ \hat{t} \in [0, \Delta t ] .
 \end{align}
 where $\mathbf{1} = [1,1, \dots, 1]^t \in \mathbb{N}^d$.
 Note that, whenever $(\rho,m) \in C(\overline{\Omega} \times [0,1])$, we have that $(\hat{\rho}^h, \hat{m}^h)$ converges to $(\rho,m)$ uniformly. 
 
 \subsubsection{Discretization of energy functionals} \label{energydiscretizationsection}
Next, we approximate the energy functionals by discrete energies $\energy^h$, beginning with energies of the form  (\ref{eqn:energy}). Given a piecewise constant function $\rho^h$   with coefficients $\rho_j$,  
\begin{align} \label{Ehgradflow}
 \F^h(\rho_j) := \sum_{j} \left( U( \rho_j) + V_j \rho_j \right) \dx+ \frac12 \sum_{j,l}   W_{j, l} \rho_j \rho_{l}   (\Delta x)^{2d} ,
\end{align}
where $V^h(x) = \sum_{j} V_j 1_{Q_j}(x)$ is a piecewise constant approximation of $V(x)$   and $W^h(x,y) = \sum_{j,l} W_{j,l} 1_{Q_j}(x)1_{Q_j}(y) $ is a piecewise constant approximation of $W(x-y)$. Here $W_{j,l} = W(|x_j - x_l|)$ symmetric, i.e., $W_{j,l} = W_{l,j}$.

Likewise, for energies of the form (\ref{eqn:energy}), we consider the following discretization of the energy $\mathcal{H}_{\rho_0}$ from equation (\ref{hoenergy}) for the Crank-Nicolson inspired scheme, Problem \ref{hodynamicJKO}, 
\begin{align} \label{Hhgradflow}
 \mathcal{H}^h_{\rho_0}(\rho_j) := \frac12 \F^h(\rho_j) + \frac12 \sum_{j} \left( U'( (\rho_0)_j) + V_j  +\sum_l W_{j, l} (\rho_0)_l   (\Delta x)^{d}  \right) \rho_j (\Delta x)^{d} \,.
\end{align}

Finally, to compute Wasserstein geodesics between two measures $\rho_0, \rho_1 \in \P_{ac}(\Omega)$, we consider a discretization of the energy $\mathcal{G}_{\rho_1}$ from equation (\ref{Erho1}). Given a piecewise constant approximation $\rho_1^h$ of $\rho_1$ and $\delta \geq 0$,  define
\begin{align} \label{Ehrho1}
\G_{\rho_1}^h(\rho_j) := \begin{cases} 0 &\text{ if } \sum_j |\rho_j - (\rho^h_1)_j |^2 (\Delta x)^d \leq \delta^2  \\ +\infty \,, &\text{ otherwise.} \end{cases}
\end{align}

\subsubsection{Discretization of derivative operators} \label{discretederivativesection}
Let $D^h_t \rho^h $ and $D^h_x m^h$ denote the discrete  time derivative and spatial divergence on $\Omega \times [0,1]$ and let $\nu^h$ denote the discrete outer unit normal of $\Omega$. (See Hypothesis \ref{discreteophyp} for the precise requirements we impose on each of these discretizations). For example, in one dimension we may choose a centered difference in space and a forward Euler method in time,
\begin{align} \label{forwardEuler}
 D^h_t \rho_{j,k} = \frac{\rho_{j,k+1} - \rho_{j,k} }{\Delta t}, \quad  D^h_x m_{j,k}  =  \frac{m_{j+1,k} - m_{j-1,k}  }{2  \Delta x } .
 \end{align}
or a Crank-Nicolson method,
\begin{align} \label{crankspace}
D^h_t \rho_{j,k} = \frac{\rho_{j,k+1} - \rho_{j,k} }{\Delta t},  \quad  D^h_x m_{j,k}  =  \frac{m_{j+1,k} - m_{j-1,k} + m_{j+1,k+1} - m_{j-1,k+1} }{4 \Delta x } .
 \end{align}
We compute these discretizations of the derivatives at the boundary by extending $m_{j,k} $ to be zero in the direction of the outer unit normal vector.
As we can only expect  these approximations of the temporal and spatial derivatives to hold up to an error term, we relax the equality constraints from (\ref{BB8}) in the following discrete dynamic JKO scheme.  

\subsubsection{Discrete dynamic JKO}
The discretizations described in the previous sections lead to a  fully discrete dynamic JKO problem:
\begin{customproblem}{$1_{j,k}$}[Discrete Dynamic JKO]  \label{ddynamicJKO}
 Fix $\tau, \delta_1, \delta_2, \delta_3, \delta_4 >0$, $\energy^h$, and $\rho_0^h$. Solve the constrained optimization problem,
\begin{align} \label{ddynamicobj}
\inf_{(\rho_{j,k},m_{j,k}) \in \C^h}   \sum_{j} \sum_{k } \Phi( \rho_{j,k}, m_{j,k})  \dx \Delta t + 2 \tau \energy^h(\rho_{j, N_t}) ,
\end{align}
where $(\rho_{j,k},m_{j,k})$ belong to the constraint set $\C^h$ provided that for all $j,k$,
\begin{align}
\sum_{j,k} |D^h_t \rho_{j,k} + D^h_x  m_{j,k}|^2 {\dx (\Delta t)}  &\leq \delta_1^2 ,  & \  \sum_{j \in \partial \Omega , k} |m_{j,k} \cdot \nu_{j} |^2  (\Delta x)^{d-1} (\Delta t) &\leq  \delta_2^2 ,  \label{BB81} \\
\  \sum_k | \sum_{j }\rho_{j,k}  \dx - \sum_j (\rho_0^h)_j \dx|^2 (\Delta t) &\leq \delta_3^2  , \ &\sum_j |\rho_{j,0} - (\rho_0^h)_j |^2 \dx &\leq \delta_4^2 .  \label{BB82}
\end{align}
\end{customproblem}

\noindent The inequalities (\ref{BB81}) enforce the PDE constraint and the boundary condition; the inequalities (\ref{BB82}) enforce the mass constraint and the initial conditions.  Recall that, by definition of $\Phi$ in equation (\ref{Phidef}), $\Phi(\rho_{j,k}, m_{j,k})< +\infty$ only if $\rho_{j,k}$ is nonnegative. Consequently, if a minimizer $\rho_{j,k}$ exists, it must be nonnegative. 

\begin{remark}[relaxation of PDE constraints] \label{relaxpderemark}
A key element of our numerical method  is that we relax  the equality constraint (\ref{BB8}) at the fully discrete level. This reflects the fact that even an exact solution of the continuum PDE  will only satisfy the discrete   constraints (\ref{BB81}-\ref{BB82}) up to an error term depending on the order of the finite difference  operators.
 
We allow the choice of $\delta_i$ to vary for each of the above constraints. However, when the desired exact solution is sufficiently smooth, the optimal choice of $\delta_i$ for a second order discretization of the spatial and temporal derivatives is
\[ \delta_1 \sim (\Delta x)^2 + (\Delta t)^2  \tau \quad \text{ and } \quad \delta_2,  \delta _3,  \delta_4 \sim (\Delta x)^2 , \]
where $\tau >0$ is the size of the timestep in the outer time discretization; see equations (\ref{eqn:JKO-000}-\ref{2nd}).  As we will demonstrate in Figure \ref{fig:optimaldelta} of our numerics section, relaxing the PDE constraint accelerates convergence to a minimizer of the fully discrete   Problem \ref{ddynamicJKO} without any loss of accuracy with respect to the exact continuum solution.

Finally, note that while the discrete PDE constraint (\ref{BB81}) automatically enforces the mass constraint  up to order $\delta_1^2 + \delta_2^2$,  we  choose to impose the mass constraint separately via the first   equation in (\ref{BB82}). This leads to better performance in examples where the exact solution is not smooth enough to satisfy the discrete PDE constraint up to a high order of accuracy but imposing a stricter mass constraint   leads to a higher quality numerical solution; see Figure \ref{fig:palace}. 
\end{remark}

Under sufficient hypotheses on the discrete energy $\energy^h$ and the initial data $\rho_0^h$, minimizers of Problem \ref{ddynamicJKO} exist; see Theorem \ref{minimizersexist}. Furthermore, this discrete dynamic JKO scheme preserves the energy decreasing  property of the original JKO scheme. To see this, note that, given an energy $\energy^h$, time step $\tau>0$, and initial data $(\rho_0^h)_j$ we may define the fully discrete proximal map by
\begin{align*}
J_\tau^h((\rho_0)_j): = \brho_{j,N_t} ,
\end{align*}
where $(\brho_{j,k}, \bm_{j,k})$ is any minimizer of Problem \ref{ddynamicJKO}. Independently  of which minimizer is chosen, we have
\begin{align*} 
\energy^h(J_\tau^h((\rho_0)_j) \leq \energy^h((\rho_0)_j) ,
\end{align*}
which can be seen by comparing the value of the objective function at the minimizer $(\brho_{j,k},\bm_{j,k})$ to the value of the objective function at  $(\rho_{j,k}, m_{j,k}) = ((\rho_0)_j, 0) \in \mathcal{C}$ and using the fact that  $\Phi  \geq 0$.
Furthermore, by iterating the fully discrete proximal map, we may  construct a \emph{fully discrete gradient flow sequence}
\begin{align*} 
(\rho^n_\tau)_{j} = J^h_\tau((\rho^{n-1}_\tau)_{j})  \text{ for all } n \in \mathbb{N}, \quad \rho^0_\tau = \rho_0^h .
\end{align*}
In analogy with the continuum case, we will use this fully discrete JKO scheme to simulate gradient flows.  (See Algorithm \ref{alg:nonlinearJKO}.)

\subsection{Primal dual algorithms for fully discrete JKO} \label{proxsplitsec}
In order to find minimizers of Problem \ref{ddynamicJKO}, we apply a primal dual operator splitting method.
Since the constraints in Problem \ref{ddynamicJKO} are linear inequality constraints, we may rewrite them in the form  $\|\tilde\Amat_i {u} - \tilde b_i\|_2 \leq \delta_i$ for $i = 1,2,3,4$, where ${u} = (\vec{\rho}, \vec{m})$, and $\vec{\rho}$ and $\vec{m}$ are vector representations of the matrices $\rho_{j,k}$ and $m_{j,k}$. (See the Appendix \ref{numericalimplementation} for explicit formulas for $\tilde \Amat_i$ and $\tilde b_i$, in one spatial dimension). Similarly, we may rewrite the first term of the objective function  (\ref{ddynamicobj}) in terms of ${u}$, defining
\begin{align*}
\Phi({u})=  \sum_{k} \sum_{j} \Phi( {\rho}_{j,k}, m_{j,k}) \dx \Delta t .
\end{align*}

We consider two cases for the energy term in the objective function. When the energy is of the form $\G^h_{\rho_1}$, as in equation (\ref{Ehrho1}), we reframe the problem by removing the energy   from the objective function   and adding $ \sum_j |\rho_{j, N_t}- (\rho_1^h)_j |^2 (\Delta x)^d \leq \delta_5^2 $ to the constraints \eqref{BB81}--\eqref{BB82}, denoting $\| \Amat_i {u} -  b_i\|_2 \leq \delta_i$, for $i = 1,2,3,4,5$, as the modified constraints. On the other hand, when the energy is of the form  (\ref{Ehgradflow}) or (\ref{Hhgradflow}), we rewrite it in terms of ${u}$ as
\begin{align} \label{discreteenergydef}
F({u}) =   \sum_j \left({U}({\rho}_{j,N_t}) + V_j {\rho}_{j,N_t} \right) \dx +\frac{1}{2} \sum_{j,l} \left(   W_{j,l} {\rho}_{j,N_t} {\rho}_{l,N_t} \right) (\Delta x)^{2d}   ,\\
H({u}) =   \frac12 F(u) + \frac 12 \sum_j \left({U'}({\rho}_{j,0}) + V_j + \sum_{l} \  W_{j,l} {\rho}_{l,0}  (\Delta x)^d \right) {\rho}_{j,N_t} (\Delta x)^d .\nonumber
\end{align}
In particular, if we let $\Smat$ be the selection matrix
\[ \Smat:\R^N \to \R^{N_x} : {u} \mapsto \rho_{j,N_t} , \]
then $F({u}) = \F^h(\Smat {u})$ and $H(u) = \mathcal{H}_{\rho_0}^h(\Smat{u})$, where $\F^h$ and $\mathcal{H}_{\rho_0}^h$ are defined in \eqref{Ehgradflow} and \eqref{Hhgradflow}, respectively.

This leads to the following two optimization problems:
\begin{customproblem}{3(a)}  \label{3a}
$\min_{{u}} \Phi({u}) + \indi_{ {\bf{ \delta}}}(\Amat {u}),$ \hspace{2.03cm} $\indi_{ {\bf \delta}} ( \Amat {u})  = \left\{ \begin{array}{cl} 0 & \|  \Amat_i {u} -   b_i\|_{2} \leq \delta_i , \ i = 1, \dots ,5 \\ +\infty , & \text{ otherwise.} \end{array} \right.$
\end{customproblem}
\begin{customproblem}{3(b)}  \label{3b}
$\min_{{u}} \Phi({u}) + 2 \tau E({u}) + \indi_{\tilde{\bf \delta}} (\tilde\Amat {u})$, \quad $\indi_{\tilde {\bf \delta}} (\tilde\Amat {u})  = \left\{ \begin{array}{cl} 0 & \|\tilde\Amat_i {u} - \tilde b_i\|_{2} \leq \tilde \delta_i , \ i = 1,\dots,4 \\ +\infty , & \text{ otherwise.} \end{array} \right.$
\end{customproblem}
To compute the Wasserstein distance, we solve Problem \ref{3a}, and to compute the gradient flow of an energy, we iterate Problem \ref{3b} $O(\frac{1}{\tau})$ times, for either $E(u) = F(u)$ (classical JKO) or $E(u) = H(u)$ (Crank-Nicolson inspired scheme).

Primal-dual methods for solving optimization problems in which the objective function is the sum of two convex functions, as in Problem \ref{3a}, are widely available   \cite{CP11}. However, analogous methods for optimizations problems in which the objective function is the sum of three convex functions, as in Problem \ref{3b}, have only recently emerged \cite{DavisYin17, Yan17}. In particular, in Algorithm \ref{alg:dis}, for Problem \ref{3a}, we use Chambolle and Pock's well-known primal dual algorithm, and in Algorithm \ref{alg:nonlinear}, for Problem \ref{3b}, we use Yan's recent extension of this algorithm to objective functions with three convex terms. Both algorithms offer an extended range of primal and dual step sizes $\lambda$ and $\sigma$ and low per-iteration complexity, due to the sparseness of  $\Smat$, $\Amat$, and $\tilde{\Amat}$. 
Note specifically that the success of Algorithm 1 depends on the ease of computing the proximal operators related to $\phi$ and $\indi_\delta$, and therefore if we simply group the additional energy term in Problem \ref{3b} to either $\phi$ or $\indi_\delta$, it would violate such property. Instead, we shall consider $E(u)$ as a separate term and take advantage of its smoothness, as shown in Algorithm 2. 
Finally, in Algorithm \ref{alg:nonlinearJKO}, we describe how Algorithm \ref{alg:nonlinear} can be iterated to approximate the full JKO sequence and, consequently, solutions of a range of nonlinear partial differential equations of Wasserstein gradient flow type.

 \begin{algorithm}[h]
\caption{Primal-Dual for Wasserstein distance}\label{alg:dis}
\SetAlgoLined
\KwIn{${u}^{0}$, $\phi^{0}$, $\text{Iter}_{max}$,   $\lambda, \sigma >0$}
\KwOut{${u}^*=\big( {\rho}^*,  {m}^*\big)$ and the Wasserstein distance $\Phi({u}^*)^{1/2}$}

\BlankLine
Let $\bar{{u}}^{0}={u}^{0}$ and $l=0$; \\
\While{$l< \text{Iter}_{max}$}{
    \Repeat{stopping criteria are achieved}{
    $\phi^{(l+1)} = \prox_{\sigma \indi_\delta^*} (\phi^{(l)} + \sigma \Amat \bar{{u}}^{(l)})$,
    \\ ${u}^{(l+1)} = \prox_{\lambda \Phi} ({u}^{(l)} - \lambda \Amat^T \phi^{(l+1)})$,
    \\ $\bar{{u}}^{(l+1)} = 2{u}^{(l+1)} - {u}^{(l)}$\,,
    }
}
    ${u}^*={u}^{(l+1)}$
\end{algorithm}

\begin{algorithm}[h]
\caption{Primal-Dual for one step of dynamic JKO}\label{alg:nonlinear}
\SetAlgoLined
\KwIn{${u}^{0}$, $\phi^{0}$, $\text{Iter}_{\text{max}}$, $\lambda, \sigma, \tau  >0$}
\KwOut{${u}^*$, $ \phi^*$  }

\BlankLine

Let $\bar{{u}}^{0}={u}^{0}$ and $l = 0$;\\
\While{$l <\text{Iter}_\text{max}$}{
    \Repeat{stopping criteria is achieved}{
    $\phi^{(l+1)} = \prox_{\sigma i_\delta^*} (\phi^{(l)} + \sigma \tilde{\Amat} \bar{{u}}^{(l)})$,
    \\ ${u}^{(l+1)} = \prox_{\lambda {\Phi}} ({u}^{(l)} - \lambda  \nabla E( {u}^{(l)}) - \lambda \tilde{\Amat}^t \phi^{(l+1)})$,
    \\ $\bar{{u}}^{(l+1)} = 2{u}^{(l+1)} - {u}^{(l)} + \lambda  \nabla E( {u}^{(l)})- \lambda  \nabla E( {u}^{(l+1)})$\,,

}
${u}^* = {u}^{(l+1)}$ \\
$\phi^* = \phi^{(l + 1)}$}
\end{algorithm}

To initialize both algorithms, we choose $\phi^0$ and $m^0$ to be zero vectors, and for $\rho^0$, we let its components at the initial time (i.e., $k = 0$)   be $\rho_0(x)$ evaluated on an equally spaced grid of width $\Delta x$, and other times to be zero. The stopping criteria consists of checking the PDE constraint \eqref{BB81}--\eqref{BB82} along with the convergence monitors: 
\begin{eqnarray}
&&\frac{|F({u}^{(l)})-F({u}^{(l-1)})|}{|F({u}^{(l)})|}<\epsilon_1, \label{cri02}\\
&& \max \left\{ \frac{\|{u}^{(l)}-{u}^{(l-1)}\|}{\|{u}^{(l)}\|}, ~ \frac{\|\phi^{(l)}-\phi^{(l-1)}\|}{\|\phi^{(l)}\|} \right\} <\epsilon_2 . \label{cri03}
\end{eqnarray}

The proximal operator, which appears in Algorithms \ref{alg:dis} and \ref{alg:nonlinear}, is defined by
\[
\prox_h(x) = \text{argmin}_u \left\{ \frac{1}{2} \|u-x\|^2 + h(u) \right\}\,.
\]
For both $h = \sigma i_{\bf \delta}^*$ and $h = \lambda \Phi$, there are explicit formulas for the proximal operators.
By Moreau's identity, we may write $\prox_{\sigma i_{\bf \delta}^*}(x)$ in terms of  projections onto balls of radius $\delta_i$ centered at $b_i$ for the $i$th portion of vector $x$: 
\begin{equation}
 \label{proxiB}
\prox_{\sigma \indi^*}(x) =   x - \sigma\proj_{B_\delta}(x/\sigma) \, ,  \qquad
 \proj_{B_\delta} (x) = \begin{cases} x_i & \|x_i-b_i\|_2 \leq \delta_i \, , \\ \delta \frac{ x_i-b_i}{\|x_i-b_i\|_2}+b_i & \text{ otherwise,}  \end{cases}  i = 1, 2, 3, 4\,.
\end{equation}
For the proximal operator of $\Phi$, as shown by Peyr\'e, Papadakis, and Oudet \cite[Proposition 1]{PPO14},
\begin{equation} \label{proxJ}
\prox_{\lambda \Phi} ({u}) = \big(\prox_{\lambda \varphi}(\rho_{j,k},m_{j,k})\big)_{j,k} \quad \text{ for }  \quad
\prox_{\lambda \varphi}(\rho,m) =  \begin{cases} (\rho^*, m^*) &  \text{ if $\rho^*>0$,} \\ (0,0) & \text{ otherwise,} \end{cases}
\end{equation}
where $\rho^*$ is the largest real root of the cubic polynomial equation
$P(x) := (x-\rho)(x+\lambda)^2-\frac{\lambda}{2}|m|^2 = 0,$
and $m^*$ can be obtained by $m^* =  \rho^*m/(\rho^*+\lambda)$. By computing the proximal operator exactly, our primal dual method is positivity preserving, respecting a key property of the original Problems \ref{dynamicJKO} and \ref{ddynamicJKO}.   

As the computations of both proximal operators \eqref{proxiB}, \eqref{proxJ} are {\it component-wise}, they can easily be {\it parallelized}. Likewise, the computation of the gradient $\nabla E$ is also component-wise:
\begin{align*}
(\nabla_{{u}} F( {u}))_j =   (U'({\rho}_{j,N_t}) + V_j +   \sum_{l } W_{j,l} \rho_{l,N_t} \dx )  \dx , \\
(\nabla_{{u}} H( {u}))_j = \frac12  (\grad_u F(u))_j +   \frac12 (U'({\rho}_{j,0}) + V_j +   \sum_{l } W_{j,l} \rho_{l,0} \dx )  \dx .
\end{align*}

\begin{remark}[discrete convolution]
As written, the above functionals involves a computation of the convolutions $\sum_{l } W_{j,l} \rho_{l, N_t}$ and $\sum_{l } W_{j,l} \rho_{l,0}$, which can be achieved efficiently using the fast Fourier transform. Note that since the product of the discrete Fourier transforms of two vectors is the Fourier transform of the circular convolution and   the interaction potential $W_{j-k}=W(x_j-x_k)$ is not a periodic function, we need zero-padding for computing the convolution. For the 1D case, we can first use the fast Fourier transform to compute the circular convolution of $\vec{W}=(W_j)_{j=-N_x+2}^{N_x-2}$ and $(\vec{\rho}, ~ (\vec{0})_{N_x-2})$, and then extract the last $N_x-1$ elements, which are the desired convolution $\sum_{k } W_{j-k} \rho_k$ for $1\leq j \leq N_x-1$.
\end{remark}

Embedding Algorithm 2 to the JKO iteration we have the following algorithm for Wasserstein gradient flows.
\begin{algorithm}[h]
\caption{Primal-Dual for JKO sequence}\label{alg:nonlinearJKO}
\SetAlgoLined
\KwIn{$\rho(x,t_0)$, $\text{Iter}_{\text{max}}$, { $\lambda, \sigma, \tau, n >0$}}
\KwOut{$\rho(x,t_k)$ for $0\leq k \leq n$ { and the corresponding energy $\energy(\rho(x,t_k))$ }}

\BlankLine
Given ${u}^{0}$, $\phi^{0}$; \\
\For{$k=1,2,\ldots,n$}{
$u^*, \phi^* = $ \textbf{Algorithm \ref{alg:nonlinear}}($u^0, \phi^0, \text{Iter}_{\text{max}},\lambda, \sigma, \tau$)  \\
$\rho(x,t_k) = \Smat {u}^*$ \\
$\phi^0 = \phi^*$ \\
${u}^0 = \max \left\{{u}^{*} - [\mathsf{1}_{N_t+1}, \mathsf{0}_{N_t+1}]^t \Kron \rho(x, t_{k-1}) +  [\mathsf{1}_{N_t+1}, \mathsf{0}_{N_t+1}]^t \Kron \rho(x, t_{k})  , \quad 0\right\}$. \\
}
\end{algorithm}
Note that line 6 in Algorithm 3 is to construct a better initial guess for $\rho$ at each JKO iteration by applying an extrapolation.

 \begin{remark}[Comparison of our numerical method to previous work] \label{comparisonremark}
Our definition of the indicator function in Problems 3(a) and 3 (b) differs from   previous work, and as a result, our primal-dual algorithm does not require the inversion of the matrix $\Amat \Amat^T$  \cite{PPO14,BB00}, which makes it quite efficient in high dimensions thanks to the sparsity of $\Amat$. A similar approach is taken in a recent preprint \cite{LOG17} to compute the earth mover's distance $W_1$, though, in this context, the earth mover's distance is dissimilar from the Wasserstein distance, since it does not require an extra time dimension and is thus a lower dimensional problem.

A second difference between our method and the approach in previous works is that, since $P(x)$ has at most one strictly positive root, it can be obtained by the general solution formula for cubic polynomials with real coefficients. Therefore, in our numerical simulations, we may compute the proximal operator $\prox_{\lambda \Phi}(u)$ by using this general solution formula, rather than via Newton iteration \cite{PPO14}. As a consequence, our method is truly positivity preserving, as opposed to positivity preserving in the limit as $\Delta x, \Delta t \to 0$.
\end{remark}

We close this section by recalling sufficient conditions  on the primal and dual step sizes $\sigma$ and $\lambda$ that ensure Algorithms \ref{alg:dis} and \ref{alg:nonlinear} converge to minimizers of Problems \ref{3a} and \ref{3b}.

\begin{proposition}[Convergence of Algorithm \ref{alg:dis}, {c.f. \cite{CP11}}]
Suppose $\sigma\lambda<1/\lambda_{max}({\Amat}{\Amat}^t)$ and a minimizer of Problem \ref{3a} exists. Then, as $\rm{Iter_{max}} \to +\infty,$ and $\epsilon_1$, $\epsilon_2 \to 0$ in the stopping criteria \eqref{cri02} \eqref{cri03}, the output ${u}^*$ of Algorithm \ref{alg:dis} converges  to a minimizer of Problem \ref{3a}.
\label{propc1}
\end{proposition}

\begin{proposition}[Convergence of Algorithm   \ref{alg:nonlinear}, {c.f. \cite{Yan17}}]
Suppose that the discrete energy $E(u)$ defined in equation (\ref{discreteenergydef}) is proper, lower semi-continuous, convex, and there exists $\beta >0$ such that $\average{u_1-u_2, \nabla_u E(u_1) - \nabla_u E(u_2)} \geq \beta \| \nabla E(u_1) - \nabla E(u_2)\|^2$. Suppose further that   $\sigma\lambda<1/\lambda_{max}(\tilde{\Amat}\tilde{\Amat}^t)$, $\lambda<2\beta$, and a minimizer of Problem \ref{3b} exists. Then, as ${\rm Iter_{max}} \to +\infty$ and $\epsilon_1$, $\epsilon_2 \to 0$, the output $u^*$ converges to a minimizer of Problem \ref{3b}.
\label{propc2}
\end{proposition}
Note here that the co-coercivity requirement on $\nabla E$ in the above proposition is equivalent to require the Lipschitz continuity of $\nabla E$, i.e., $\|\nabla_u E(u_1) - \nabla_u E(u_2)\| \leq \frac{1}{\beta} \|u_1 - u_2\|$. For the energy of the form \eqref{eqn:energy}, this requirement reduces to the boundedness of $U''(\rho)$ and $W$, which can be satisfied independent of the numerical resolution if we consider bounded solution (no finite time blow up in $\rho$) and non-singular interaction kernel. In the  case when $W$ is singular, for example when $W$ is a  Newtonian interaction potential, we approximate $W$ by a continuous function via convolution with a mollifier; see Remark \ref{energyhypremark}.

\section{Convergence} \label{sec:convergence}

We now prove the convergence of solutions of the fully discrete JKO scheme, Problem \ref{ddynamicJKO}, to a solution of the continuum JKO scheme, Problem \ref{dynamicJKO}. We begin, in section \ref{hypsec}, by describing the hypotheses we place on the underlying domain $\Omega$, the energy $\energy$, the initial data $\rho_0$, and the discretization operators.  Then, in section \ref{existsec}, we show that minimizers of Problem \ref{ddynamicJKO} exist, provided the discretization is sufficiently refined. Finally, in section \ref{convsec}, we prove that any sequence of minimizers of Problem \ref{ddynamicJKO} has a subsequence that converges to a minimizer of Problem \ref{dynamicJKO}. In order for our finite difference approximation to converge, we assume throughout that a smooth, positive minimizer of the continuum JKO scheme Problem \ref{dynamicJKO} exists. See hypothesis (H\ref{smoothhyp}) and Remark \ref{smoothpositive} for further discussion of this assumption.

\subsection{Hypotheses} \label{hypsec}
 We impose the following hypotheses on the underlying domain, energy, and discretization operators.  

\begin{enumerate}[(H1)]
\item $\Omega = \Pi_{i=1}^d (a_i,b_i) \subseteq \Rd$, for $a_i< b_i \in \R$.  We assume that the spacing of the spatial discretization $(\Delta x) >0 $ and the temporal discretization $(\Delta t)>0$ are both functions of $h$ satisfying $\lim_{h \to 0} (\Delta x) = \lim_{h \to 0} (\Delta t) = 0$.
\label{domainhyp}  

\item For any piecewise constant function $\rho^h$ on $\Omega$, the discrete energy functional $\mathcal{E}^h$ has one of the following forms,  as described in section \ref{energydiscretizationsection}: \label{energyhyp}
\begin{enumerate}[(a)]
\item $ \F^h(\rho^h) = \sum_{j} \left( U( \rho_j) + V_j \rho_j \right) \dx+ \sum_{j,l}   W_{j, l} \rho_j \rho_{l}   (\Delta x)^{2d} $ 
\item $ \mathcal{H}^h_{\rho_0}(\rho^h) = \frac12 \F^h(\rho^h) + \frac12 \sum_{j} \left( U'( (\rho_0)_j) + V_j  +\sum_l W_{j, l} (\rho_0)_l   (\Delta x)^{d}  \right) \rho_j (\Delta x)^{d} $  
\item ${\mathcal{G}}_{\rho_1}^h(\rho_j) := \begin{cases} 0 &\text{ if } \sum_j |\rho_j - ( {\rho}^h_1)_j |^2 (\Delta x)^d \leq \delta_5^2  \\ +\infty &\text{ otherwise.} \end{cases} $  
\end{enumerate}

We place the following assumptions on  $U,V,$ and $W$ and the target measure $\rho_1$:
 \begin{enumerate}[(i)] 
 \item Either $U \equiv 0$ or $U \in C([0, +\infty))$ is  convex,   $U \in C^1((0,+\infty))$, $\lim_{r \to +\infty} \frac{U(r)}{r} = +\infty$,  and $U(0) = 0$; 
 \item $ V^h(x) := \sum_{j \in \mathbb{Z}^d} V_{j} 1_{Q_{j}}(x) $ and $W^h(x,y) := \sum_{(j,l) \in \mathbb{Z}^d \times \mathbb{Z}^d} W_{j,l} 1_{Q_{j}}(x) 1_{Q_{l}}(y)$ are piecewise constant approximations of  $V,W \in C( \overline{\Omega})$ converging uniformly on $\overline{\Omega}$.
 \item $\rho_1 \in C^1(\overline{\Omega})$ and $\rho_1^h$ is a pointwise piecewise constant approximation of $\rho_1$. \label{rho1ass}\end{enumerate}

\item $D^{h}_t$ and $D^{h}_x$ are finite difference approximations of the time derivative and spatial divergence. We assume that $D^h_t$ is a forward Euler method in time, whereas $D^h_x$ can be given by an explicit or implicit scheme of first or higher order. We denote by
 $D^{-h}_t$ and $D^{-h}_x$ the dual operators with respect to the $\ell^2$ inner product, and we assume  the following integration by parts formulas hold for all piecewise constant functions $\rho^h,f^h:  [0,1] \to \R$,
\[ \int_0^1 D^{h}_t \rho^h f^h   \rd t = \left(  \left. \rho^h f^h \right|_{0}^1  \right)   -\int_0^1   \rho^h D^{-h}_t f^h  \rd t \]
and if  $m^h:\Omega  \to \R^d$, $f^h : \Omega  \to \R$,
\[   \int_{\Omega} D^{h}_x m^h f^h  \rd x   =   \int_{\partial \Omega } f^h m^h \cdot \nu^h \rd x -  \int_{\Omega } m^h D^{-h}_x f^h \rd x   , \]
where $\nu^h: \overline{\Omega} \to \Rd$ is the discrete outer unit normal of $\Omega$.
Finally, we assume there exists $C>0$ depending on the domain $\overline{\Omega} \times [0,1]$, so that, for any $f \in C^1(\overline{\Omega} \times [0,1]; \R)$ and $v \in C^1(\overline{\Omega} \times[0,1]; \R^d)$, if $({f}^h,{v}^h)$ are  pointwise piecewise constant approximations,
\begin{align*}
  \|D^{h}_t {f}^h - \partial_t f\|_\infty &\leq C\|\partial_t^2 f \|_\infty (\Delta t) , &  \|D^{-h}_t {f}^h - \partial_t f \|_\infty  &\leq C\|\partial_t^2 f \|_\infty (\Delta t) \\
  \|D^{h}_x {v}^h - \grad \cdot v\|_\infty  &\leq C \| D^2 v\|_\infty (\Delta x), &  \|D^{-h}_x {f}^h - \grad f \|_\infty  &\leq C \|D^2 f \|_\infty (\Delta x) \\
  \| {v}^h \cdot \nu^h - v \cdot \nu \|_\infty &  \leq C \|  v\|_\infty (\Delta x) .
\end{align*}
 (See section \ref{discretederivativesection} for  finite difference approximations satisfying these hypotheses.)
\label{discreteophyp}

\item The constraint relaxation parameters $\delta_1, \delta_2, \delta_3, \delta_4 \geq 0$ are functions of $h$ with \label{deltahyp}$ \lim_{h \to 0} \delta_i = 0$, for all $i$. If the energy is of the form (H\ref{energyhyp}c), we require that $\delta_5$ is a function of $h$ satisfying $ \lim_{h \to 0} \delta_5 = 0$  and $ \lim_{h \to 0} \left( \Delta x + \Delta t\right)/\delta_5 = 0$. 

\item The initial data of the continuum problem satisfies $\rho_0 \in C^1(\overline{\Omega})$ and $ {\rho}_0^h$  is a pointwise piecewise constant approximation of $\rho_0$. 

\item Given the domain, energy, and initial data described in the previous hypotheses, there exists a minimizer $( {\rho}, {m})$ of the continuum Problem \ref{dynamicJKO} satisfying $ {\rho} \in C^2([0,1]; C^1(\overline{\Omega}))$, $ {\rho} >0$, and $ {m} \in C^1([0,1]; C^2(\overline{\Omega}))$. \label{smoothhyp}
\label{initialhyp}

\end{enumerate}

To ease notation in the following convergence proof, we observe that Problem \ref{ddynamicJKO} may be rewritten as follows in terms of $(\rho^h, m^h)$,  the piecewise constant functions on $\Omega \times [0,1]$ corresponding to the coefficients $(\rho_{j,k}, m_{j,k})$.
\begin{customproblem}{${1^h}$}[Discrete Dynamic JKO]  \label{ddynamicJKO2}
Fix $\tau, \delta_1, \delta_2, \delta_3, \delta_4 >0$,   $\energy^h$, and    $\rho_0^h$. Solve the constrained optimization problem,
\begin{align*}
\inf_{(\rho^h,m^h) \in \C^h}  \int_0^1 \int_\Omega \Phi( \rho^h, m^h) \rd x \rd t  +2 \tau \energy^h(\rho^h(\cdot, 1)) ,
\end{align*}
where $(\rho^h,m^h)$ belong to the constraint set $\C^h$ provided that they are piecewise constant functions  on $\Omega \times [0,1]$   and the following inequalities hold
\begin{align}
\|D^h_t \rho^h + D^h_x  m^h\|_{L^2(\Omega \times [0,1])}  & \leq \delta_1\,, & \|m^h \cdot \nu^h \|_{L^2(\partial \Omega \times [0,1])}  &\leq \delta_2\,, \label{BB82i} \\
\left\| \int_{\Omega} \rho^h(x, \cdot) \rd x - \int_{\Omega}  {\rho}^h_0(x) \rd x \right\|_{L^2([0,1])} &\leq \delta_3\,, & \|\rho^h(\cdot,0) -  {\rho}_0^h \|_{L^2(\Omega)}  &\leq \delta_4 . \label{BB82ii}
\end{align}
\end{customproblem}

Similarly, we may rewrite the definition of the discrete energies in hypothesis (H\ref{energyhyp}) in terms of a piecewise constant functions $\rho^h$  on $\Omega$ corresponding to $\rho_j$,
\begin{align*} 
 \F^h(\rho^h) &= \int_\Omega U( \rho^h(x)) +  {V}^h(x) \rho^h(x) \rd x + \frac12 \iint_{\Omega\times \Omega}  {W}^h(x,y) \rho^h(x) \rho^h(y) \rd x \rd y , \\ 
 \mathcal{H}^h_{\rho_0}(\rho^h) &= \frac12 \F^h(\rho^h) + \frac12 \int_\Omega \left( U'( \rho^h_0(x)) + V^h(x)  +\int_\Omega W^h(x,y) \rho^h_0(y) \rd y  \right) \rho(x)\rd x , \\
\mathcal{G}_{\rho_1}^h(\rho^h) &= \begin{cases} 0 &\text{ if } \|\rho^h - \rho^h_1 \|_{L^2(\Omega)} \leq \delta_5  \\ +\infty &\text{ otherwise.} \end{cases}
\end{align*}
\noindent Recall that, by definition of $\Phi$ in equation (\ref{Phidef}), $\Phi(\rho^h, m^h)< +\infty$ only if $\rho^h$ is nonnegative. Consequently, if a minimizer $\rho$ exists, it must be nonnegative.

We conclude this section with several remarks on the sharpness of the preceding hypotheses.

\begin{remark}[assumption on domain $\Omega$]
In hypothesis (H\ref{domainhyp}), we assume that $\Omega$ is an n-dimensional hyperrectangle. We impose this assumption for simplicity, as it provides an natural interpretation of the discretized outer unit normal $\nu^h$, which is essential in imposing the boundary conditions for our PDE constraint at the discrete level. More generally, our convergence result can be  extended to any Lipschitz domain, as long as sufficient care is taken to define the discrete outer unit normal and the corresponding no flux boundary conditions.
\end{remark}

\begin{remark}[assumption on energy] \label{energyhypremark}
As described in hypothesis (H\ref{energyhyp}),  our convergence result applies to internal  $U$, drift $V$, and interaction $W$ potentials that are sufficiently regular  on $\overline{\Omega}$. Our assumptions on $U$ are classical and ensure that   the internal energy  is lower semicontinuous with respect to weak-* convergence  \cite[Remark 9.3.8]{AGS}. 
Our assumptions on $V$ and $W$, on the other hand, are somewhat stronger, and in practice, one often encounters partial differential equations for which the corresponding choices of $V$ and $W$ are not continuous. However, there are  robust methods for approximating these potentials by continuous functions   that ensure  convergence of the gradient flows. For example, the second author and Topaloglu provide sufficient conditions on discontinuous interaction potentials $W$ for which gradient flows of the regularized interaction potential, $W_\eps := W*\varphi_\eps$ for a smooth mollifier $\varphi_\eps$, converge to gradient flows of the original interaction potential $W$, as well as conditions that ensure minimizers of $W_\eps$ converge to minimizers of $W$ \cite{craig2016convergence}. (The convergence of general stationary points of $W_\epsilon$ that are not global minimizers to stationary point of $W$ remains open.)
\end{remark}

\begin{remark}[assumption on  $\delta_5$] \label{deltaremark}
In hypothesis (H\ref{deltahyp}), it is essential that $\delta_5$ not vanish too quickly with respect to other parameters in the discretization. A simple illustration of this fact arises in the case that $\delta_1 \equiv \delta_2 \equiv \delta_3 \equiv \delta _4 \equiv 0$. In this case, we cannot choose $\delta_5 \equiv 0$, since our pointwise piecewise approximation of the initial data $\rho_0^h$ will not generally have the same mass as our pointwise piecewise approximation of the target measure $\rho_1^h$, and if they do not have the same mass,  \emph{minimizers of the discrete problem do not exist}. Consequently, it would be impossible to prove that minimizers of the fully discrete problem converge to minimizers of the continuum problem.
 On one hand, this does not greatly impact the performance of our numerical method, as can be seen by considering previous work by Papadakis, P\'eyre, and Oudet, which numerically implements this approach \cite{PPO14}. On the other hand, our numerical simulation in Figure \ref{fig:optimaldelta} indicates that poor choice of the relaxation parameters can cause the method to iterate longer than necessary, without any improvement in accuracy. 
 
Our requirement that $\lim_{h \to 0} (\Delta x + \Delta t)/\delta_5 =0$ is sufficient to fix this problem and ensure convergence of the method, and   this requirement is nearly sharp. To see this, note that, for an arbitrary pointwise piecewise approximation $\rho_0^h$ of a continuous function $\rho_0$, we cannot in general achieve accuracy of $| \int_\Omega \rho_0^h - \int_\Omega \rho_0|$ better than $O (\Delta x)$. If either $\delta_1$ and $\delta_3$, the parameters for the PDE constraint and the mass constraint, are chosen arbitrarily small, then $|\int_\Omega \rho^h(\cdot,1) - \int_\Omega \rho_0^h|$ can likewise be made arbitrarily small. Thus, since $\rho_0, \rho_1 \in \P_{ac}(\Omega)$,
\begin{align*} O(\Delta x) \approx \left| \int_\Omega \rho_0^h - \int_\Omega \rho_0 \right| &\approx \left| \int_\Omega \rho_0^h - \int_\Omega \rho_0 \right| - \left|\int_\Omega \rho_0 - \int_\Omega \rho_1 \right| - \left| \int_\Omega \rho^h(\cdot, 1) - \int_\Omega \rho_0^h \right|   \\
&\leq \left| \int_\Omega \rho^h(\cdot, 1) - \int_\Omega \rho_1 \right| \leq |\Omega|^{1/2} \| \rho^h(\cdot, 1) - \rho_1\|_{L^2(\Omega)} \leq |\Omega|^{1/2} \delta_5 ,
\end{align*}
so we much have $\delta_5 \geq O(\Delta x)$.
While a CFL-type condition is not necessary for the stability of our discretization of the PDE constraint, since $\rho$ and $m$ indeed become coupled in the continuum limit (see equations (\ref{BB1}) and (\ref{BB4})), one should expect $(\Delta t) \leq O(\Delta x)$ to give the best balance between computational accuracy and cost, and we indeed observe this numerically. Combining these facts shows that enforcing that $\delta_5$ cannot decay faster than $O(\Delta x + \Delta t)$ by assuming $\lim_{h \to 0} (\Delta x + \Delta t)/\delta_5 =0$ is nearly optimal.
\end{remark}

\begin{remark}[assumption on existence of smooth, positive minimizer] \label{smoothpositive}
In hypothesis (H\ref{smoothhyp}), we suppose that there exists a sufficiently regular minimizer $(\overline{\rho},\overline{m})$, $\bar{\rho}>0$, of the continuum problem. Our proof of the existence of minimizers of the fully discrete problem and our proof that minimizers of the discrete problems converge to a minimizer of the continuum problem as $h \to 0$ strongly rely on this assumption. In particular, the smoothness assumption allows us to use convergence of the finite difference operators, described in hypothesis (H\ref{discreteophyp}),   to construct an element of $\mathcal{C}^h$ in Proposition \ref{specialelement}. The positivity assumption allows us to conclude that $\grad_{\rho, m} \Phi$ is uniformly bounded on the range of $\bar{\rho}$, which we use to prove the $\limsup$ inequality for the recovery sequence  in Theorem \ref{gammaconvergence}(\ref{limsupineq}). 

From the perspective of approximating  gradient flows, which are solutions of diffusive partial differential equations (\ref{eqn:000}), such regularity and positivity can be guaranteed as long as the initial data is smooth and positive and either the diffusion is sufficiently strong or  the drift and interaction terms do not cause loss of regularity. On the other hand, developing conditions on the energy and initial data that ensure such regularity and positivity holds at the level of the JKO scheme, for minimizers of Problem \ref{dynamicJKO}, remains largely open: results  on the propagation of  $L^p(\Rd)$ or BV bounds along the scheme have only recently emerged \cite{blanchet2012functional, carrillo2018L, de2016bv}.
 
From the perspective of approximating Wasserstein geodesics, the now classical regularity theory developed by Caffarelli and Urbas ensures that if the  source and target measures $\rho_0$ and $\rho_1$ are smooth and strictly positive, then the minimizer of Problem \ref{dynamicJKO} $(\bar{\rho},\bar{m})$ is also smooth and strictly positive. (See, for example, \cite[Section 4.3]{Villani03a} and \cite[Section 8.3]{AGS}.)

Along with this analytical justification for our smoothness and positivity assumptions, our numerical results also indicate that such assumptions are in general necessary. For example in Figure \ref{fig:palace}, we observe that if the source and target measure of a Wasserstein geodesic are not sufficiently smooth, the numerical solution introduces artificial regularity. Likewise, even in Figure \ref{fig:PM1D}, we observe that the numerical simulation is strictly positive (though very close to zero in places), while the exact solution is identically zero outside of its support. Still, in spite of the fact that our theoretical convergence result requires smoothness and positivity assumptions, in practice our numerical method  still performs well on nonsmooth or nonpositive problems, provided that the spatial and temporal discretization are taken to be sufficiently small; see Figures \ref{fig:pacman}-\ref{fig:KS3}. 

Finally, these types of smoothness and positivity assumptions are typically needed in convergence proofs for numerical methods based on the JKO scheme. For example, in a method based on the Monge Amp\'ere approximation of the Wasserstein distance, the exact solution is required to be uniformly bounded above and below  \cite{BCMO16}. Likewise, while rigorous convergence results for fully discrete numerical methods based on entropic or Fisher information regularization remain open, since these methods correspond to introducing numerical diffusion at the level of the PDE, they automatically enforce smoothness and positivity \cite{CPSV16, CDPS17,LYO17}.
\end{remark}

\subsection{Existence of minimizers} \label{existsec}
We now show that, under the hypotheses described in the previous section,  minimizers of the fully discrete JKO scheme, Problem \ref{ddynamicJKO2}, exist for all $h>0$ sufficiently small. We begin with the following proposition, which constructs a specific element in the constraint set $\mathcal{C}^h$, which we will use both in our proof of existence of minimizers and in our $\Gamma$-convergence results in the next section.

\begin{proposition}[construction of element in $\mathcal{C}^h$] \label{specialelement}
Suppose that hypotheses (H\ref{domainhyp})-(H\ref{smoothhyp})   hold, and choose $(\rho, m) \in \mathcal{C}$  satisfying $ {\rho} \in C^2([0,1]; C^1(\overline{\Omega}))$, $ {\rho} >0$, and $ {m} \in C^1([0,1]; C^2(\overline{\Omega}))$. Then for $h>0$ sufficiently small, there exists $(\tilde{\rho}^h, \tilde{m}^h) \in \mathcal{C}^h$ satisfying  $(\tilde{\rho}^h, \tilde{m}^h) \xrightarrow{h \to 0} (\rho,m)$ uniformly on $\Omega \times [0,1]$ and
\begin{align} \label{rhohunifpos}
 \inf_{h >0, (x,t) \in \Omega \times [0,1]} \tilde{\rho}^h(x,t) >0 .
 \end{align}

If, in addition, the energy satisfies hypothesis (H\ref{energyhyp}c) and $\energy(\rho(\cdot, 1))<+\infty$, then we have
\begin{align} \label{specialelementc}
  \| \tilde{\rho}^h(\cdot, 1) - \rho^h_1\|_{L^2(\Omega)} \leq \delta_5 ,
  \end{align}
for all $h>0$ sufficiently small.
\end{proposition}

\begin{proof}
We construct $(\rho^h,m^h) \in \mathcal{C}^h$ as follows. Let  
 $\hat{m}^h$  be a pointwise piecewise constant approximation of $m$; see equation (\ref{pointwisepiecewise}). Recall that  $\nu^h$ is the discrete outer unit normal vector. We define $\tilde{m}^h: \Omega \times[0,1] \to \Rd $ component-wise to respect the no flux boundary conditions, letting $(\tilde{m}^h)_l$ denote the $l$th component of the vector for $l = 1, \dots, d$. If $x \in \partial \Omega$, then we define
 \[ (\tilde{m}^h(x,t))_l = \begin{cases}  (\hat{m}^h(x,t))_l &\text{ for } e_l \cdot \nu^h(x) = 0 , \\ 0 &\text{ for } e_l \cdot \nu^h(x) \neq 0 . \end{cases}  \]
Otherwise, we take $\tilde{m}^h(x,t) = \hat{m}^h(x,t)$.
Define $\tilde{\rho}^h: \Omega \times [0,1] \to \R$ so that $\tilde{\rho}^h(x,0) =  {\rho}_0^h$    and $D^{h}_t \tilde{\rho}^h(x,t) + D^h_x \tilde{m}^h(x,t) \equiv 0$.

We begin by showing that  $(\tilde{\rho}^h, \tilde{m}^h) \in \mathcal{C}^h$. By construction, for all $h >0$,
\begin{align*}
\|D^h_t \tilde{\rho}^h + D^h_x \tilde{m}^h\|_{L^2(\Omega \times [0,1])}   &= 0 \\
 \|\tilde{m}^h \cdot \nu^h \|_{L^2(\partial \Omega \times [0,1])} &=0 \\
  \|\tilde{\rho}^h(\cdot,0) -  {\rho}_0^h \|_{L^2(\Omega)} &=0  .
\end{align*}
Taking $f^h \equiv 1$ in Hypothesis (H\ref{discreteophyp}) and applying the PDE constraint ensures that, for all $s \in [0,1]$ and $k \in \mathbb{N}$ so that $k (\Delta t) \leq s < (k+1) \Delta t$,
\begin{align*}
 \int_{\Omega} \tilde{\rho}^h(x, s) \rd x - \int_\Omega \tilde{\rho}^h(x,0) \rd x  &=   \int_0^{k (\Delta t)} \int_{\Omega} D^h_t \tilde{\rho}^h(x,t) f^h(x,t) \rd x \rd t   \\
&= - \int_0^{k (\Delta t)} \int_{\Omega} D^h_x \tilde{m}^h(x,t) f^h(x,t) \rd x \rd t  \\
&=  - \int_0^{k (\Delta t)} \int_{\partial \Omega}   \tilde{m}^h(x,t) \cdot \nu^h(x,t) \rd x \rd t    = 0 .
\end{align*}
Thus, we also obtain 
\[ \left\| \int_\Omega \tilde{\rho}^h(x, \cdot) \rd x - \int_\Omega {\rho}_0^h(x) \rd x  \right\|_{L^2([0,1])}  = 0 , \text{ for all } h >0 .\]
This concludes the proof that $(\tilde{\rho}^h, \tilde{m}^h) \in \mathcal{C}^h$. 

We now show that $(\tilde{\rho}^h, \tilde{m}^h) \to (\rho,m)$ uniformly on $\Omega \times [0,1]$ as $h \to 0$.
 We begin by proving convergence of $\tilde{m}^h$ to $m$.  Due to hypothesis (H\ref{domainhyp}) on our domain $\Omega$,  whenever $e_i \cdot \nu^h(x) \neq 0$, there exists $y \in \partial \Omega$ so that $|y-x| \leq 2\sqrt{d} (\Delta x)$ and $  \nu(y) = e_i$.  
Thus, whenever $e_i \cdot \nu^h(x) \neq 0$, the continuum boundary condition $m(y,t) \cdot \nu(y) = 0$ ensures that for all $t \in [0,1]$,
\begin{align*}
|(\tilde{m}^h(x,t) - m(x,t))_i|= |m(x,t)\cdot e_i| &\leq   |(m(x,t) - m(y,t)) \cdot e_i| + |m(y,t)\cdot e_i  | \leq 2\sqrt{d} (\Delta x) \|Dm\|_\infty  .
\end{align*}
We also have that, for all $(x,t) \in \Omega \times [0,1]$,
\begin{align*}
 |\hat{m}^h(x,t) - m (x,t)| \leq(\Delta x) \|D m \|_\infty + (\Delta t) \| \partial_t m \|_\infty     .
\end{align*}
Therefore, for all $(x,t) \in \Omega \times [0,1]$, there exists $C_m = C_m(d, \|Dm\|_\infty, \| \partial_t m \|_\infty)>0$ so that
\begin{align*} 
 |\tilde{m}^h(x,t) - m (x,t)| \leq C_m(\Delta t + \Delta x) \xrightarrow{h \to 0} 0.
\end{align*}

We now prove the convergence of $\tilde{\rho}^h$  to $\rho$.  Since $(\rho,m)$  is a classical solution of the  PDE constraint and $\tilde{\rho}^h: \Omega \times [0,1] \to \R$ is defined by the conditions that $\tilde{\rho}^h(x,0) = \hat{\rho}_0^h$ and $D^{h}_t \tilde{\rho}^h(x,t) + D^h_x \tilde{m}^h(x,t) \equiv 0$,  for  $(x,t) \in \Omega \times[0,1]$ and $k \in \mathbb{N}$ so that $k (\Delta t) \leq t < (k+1) (\Delta t)$, we have
\begin{align} \label{unifrhoconv}
&|\tilde{\rho}^h(x,t) - \rho(x,t)|  \nonumber \\
&\quad = \left| \tilde{\rho}^h(x,0) + \int_0^{k (\Delta t)} D^h_t \tilde{\rho}^h(x,s) \rd s - \rho(x,0)  - \int_0^t \partial_s   \rho(x,s) \rd s \right| \nonumber \\
&\quad = \left|  {\rho}_0^h(x) - \int_0^{k (\Delta t)} D^h_x \tilde{m}^h(x,s) \rd s - \rho(x,0)  + \int_0^t \grad \cdot m(x,s) \rd s \right| \nonumber \\
&\quad = \left|  {\rho}_0^h(x) - \rho(x,0) \right| + \left| \int_0^{k (\Delta t)} D^h_x \tilde{m}^h(x,s) \rd s  - \int_0^{k (\Delta t)} \grad \cdot m(x,s) \rd s \right| +  \left| \int_{k (\Delta t)}^t \grad \cdot m(x,s) \rd s \right| \nonumber \\
&\quad \leq \|\grad \rho\|_\infty (\Delta x) +   C \|D^2 m \|_\infty (\Delta x) + \| \grad \cdot m\|_\infty (\Delta t) \xrightarrow{h \to 0} 0.
\end{align}
Since $\tilde{\rho}^h \to \rho$ uniformly and $\rho >0$, we immediately obtain (\ref{rhohunifpos}).

Finally, suppose  the energy satisfies (H\ref{energyhyp}c). Since $\energy(\rho(\cdot, 1))  = \G_{\rho_1}(\rho(\cdot, 1))< +\infty$, we have $\rho(\cdot, 1) = \rho_1$. By inequality  (\ref{unifrhoconv}) and the fact that $\rho^h_1$ is a pointwise piecewise approximation of $\rho(\cdot, 1)$,
\begin{align*}
\| \tilde{\rho}^h(\cdot, 1) - \rho^h_1\|_{L^2(\Omega)} \leq |\Omega|^{1/2} \left(\|\tilde{\rho}^h(\cdot, 1) - \rho(\cdot, 1)\|_\infty + \|\rho(\cdot, 1) -\rho^h_1 \|_\infty \right) \leq C_{\rho,m} (\Delta x + \Delta t)
\end{align*}
where $C_{\rho,m} = C_{\rho,m}(\Omega,\|\grad \rho\|_\infty,\|\grad \cdot m \|_\infty, \|D^2m\|_\infty)>0$.   By hypothesis (H\ref{deltahyp}), $\lim_{h \to 0} \frac{(\Delta x + \Delta t)}{\delta_5} \to 0$. Thus, for $h$ sufficiently small,
\[ \| \tilde{\rho}^h(\cdot, 1) - \rho^h_1\|_{L^2(\Omega)} \leq \delta_5 , \]
which completes the proof.

\end{proof}

\begin{theorem}[minimizers of discrete dynamic JKO exist]  \label{minimizersexist}
Suppose that hypotheses (H\ref{domainhyp})-(H\ref{smoothhyp})   hold. Then for all $h>0$ sufficiently small, a minimizer of Problem \ref{ddynamicJKO2} exists.
\end{theorem}

\begin{proof}
First, we note that Proposition \ref{specialelement} ensures that, for $h>0$ sufficiently small, the constraint set $\mathcal{C}^h$ is nonempty and contains some $(\rho^h,m^h)$ satisfying $\rho^h > 0$. 
If the energy satisfies (H\ref{energyhyp}a) or (H\ref{energyhyp}b), then we immediately obtain $\energy^h(\rho^h( \cdot, 1))<+\infty$. Similarly,   if the energy satisfies (H\ref{energyhyp}c), then inequality (\ref{specialelementc}) in Proposition \ref{specialelement} again ensures that $\energy^h(\rho^h(\cdot,1)) < +\infty$.

Since $\Phi(\rho^h,m^h)<+\infty$ whenever $\rho^h \geq 0$, this ensures that value of the objective function in the discrete minimization problem \ref{ddynamicJKO2} is not identically $+\infty$ on the constraint set.
Therefore,
\begin{align} \label{minimizingsequence1}
\inf_{(\rho^h,m^h) \in \C^h}  \int_0^1 \int_\Omega \Phi( \rho^h (x,t), m^h (x,t)) \rd x \rd t  +2 \tau \energy^h(\rho^h(\cdot, 1))  <+\infty ,
\end{align}
and we may choose a minimizing sequence $(\rho^h_n, m^h_n) \in \mathcal{C}^h$ that converges to the infimum. We may assume, without loss of generality, that
\begin{align} \label{minimizingsequence1p5}
\sup_n \int_0^1 \int_\Omega \Phi( \rho_n^h (x,t), m_n^h (x,t)) \rd x \rd t  +2 \tau \energy^h(\rho_n^h(\cdot, 1))  <+\infty ,
\end{align}

To conclude the proof of the theorem, we will now show that there exists $(\rho^h_*, m^h_*)$ so that a subsequence of $(\rho^h_n, m^h_n)$ converges to $(\rho^h_*, m^h_*)$ uniformly on $\Omega \times [0,1]$. Then, since the objective functional $\energy^h$ is lower semi-continuous along uniformly convergent sequences \cite[Example 2.36]{ambrosio2000functions} and the constraint set $\mathcal{C}^h$ is closed under uniform convergence for fixed $h>0$, $ (\rho^h_*, m^h_*)$ must be a minimizer of the fully discrete problem. 

In order to obtain compactness of $(\rho^h_n, m^h_n)$, first note that (\ref{minimizingsequence1}) ensures $ \Phi( \rho^h, m^h) < +\infty$ on $\overline{\Omega} \times [0,1]$, so $\rho^h \geq 0$ on $\overline{\Omega}$. Furthermore, the mass constraint (\ref{BB82ii}) ensures that there exists $R = R(h) >0$, depending on $\Omega$, $(\Delta x)$, $(\Delta t)$, and $\delta_3$ so that $|\rho^h_n(x,t)| \leq R$ for all $(x,t) \in \Omega \times [0,1]$. Therefore, the vector of coefficients $(\rho^h_n)_{j,k}$ for this piecewise constant function satisfies $(\rho^h_n)_{j,k} \in B_R(0) \subseteq \mathbb{R}^{N_x^d N_t}$. Consequently, by the Heine-Borel theorem, there exists a vector $(\rho^h_*)_{j,k} \in \mathbb{R}^{N_x^d N_t}$ so that, up to a subsequence, $(\rho^h_n)_{j,k} \to (\rho^h_*)_{j,k}$. Therefore, if $\rho^h_*$ denotes the corresponding piecewise constant function, we have that, up to taking a subsequence which we again denote by $\rho^h_n(x,t)$,  $\lim_{n \to +\infty} \rho^h_n(x,t) = \rho^h_*(x,t)$ uniformly on $\Omega \times [0,1]$.

Next, we show that
\begin{align} \label{fixedhenergybd}
 \inf_{n} \energy^h(\rho^h_n(\cdot, 1)) >-\infty . 
 \end{align}
 If the energy satisfies (H\ref{energyhyp}c), then $\energy^h(\rho^h_n(\cdot, 1)) \geq 0$ for all $n$, and the above inequality is immediate. If the energy satisfies (H\ref{energyhyp}a) or (H\ref{energyhyp}b), then this follows from the fact that $U$ is bounded below on $[0, +\infty]$, $V$ and $W$ are bounded below on $\overline{\Omega}$ and $\rho^h_n(x,t) \to \rho^h_*(x,t)$ uniformly.
 
 Combining  (\ref{minimizingsequence1p5}) and (\ref{fixedhenergybd}), we obtain   
 \begin{align} \label{fixedhPhibd1}
 \sup_{n} \int_0^1\int_\Omega \Phi( \rho^h_n(x,t), m^h_n(x,t)) \rd x \rd t < +\infty .
 \end{align}
 Furthermore, since $0 \leq \rho^h_n(x,t) \leq R$ for all $(x,t) \in \Omega \times [0,1]$, $n \in \mathbb{N}$, we have 
 \begin{align} \label{fixedhPhibd}
  \Phi( \rho^h_n(x,t), m^h_n(x,t)) \geq |m^h_n(x,t)|^2/R .
  \end{align}
  Therefore, combining (\ref{fixedhPhibd1}) and (\ref{fixedhPhibd}), we obtain that there exists $R' = R'(h)>0$,   depending on $\Omega$, $(\Delta x)$, $(\Delta t)$, and $\delta_3$, so that $|m^h_n(x,t)| \leq R'$ for all $(x,t) \in \Omega \times [0,1]$. Arguing as before, the Heine-Borel theorem ensures that, up to a subsequence, $\lim_{n \to +\infty} m^h_n(x,t) = m^h_*(x,t)$ uniformly on $\Omega \times [0,1]$, for some piecewise constant function $m^h_*(x,t)$.
This gives the result.

\end{proof}

\subsection{Convergence of minimizers} \label{convsec}

We now prove that minimizers of the discrete dynamic JKO scheme, Problem \ref{ddynamicJKO2}    converge to minimizers of Problem \ref{dynamicJKO} as  $h \to 0$. 
We begin with the following lemma, showing that any $(\rho^h, m^h) \in \mathcal{C}^h$  satisfies a weak form of the PDE constraint, in the limit as $h \to 0$.

\begin{lemma}[properties of $\mathcal{C}^h$] \label{weakPDElem}
Suppose that hypotheses (H\ref{domainhyp})-(H\ref{smoothhyp})   hold, and fix $(\rho^h,m^h) \in \mathcal{C}^h$ so that $\int_0^1 \int_\Omega \Phi(\rho^h, m^h) <+\infty$ for each $h>0$. Then  $\rho^h(\cdot, 0) \to \rho_0$ in $L^2(\Omega)$, and there exist $\rho \in \P(\Omega \times [0,1])$ and $\mu \in \P(\Omega)$ so that, up to a subsequence, $\rho^h \wsto \rho$ and $\rho^h(\cdot, 1) \wsto \mu$. Furthermore, for any piecewise constant function $f^h$  with $\sup_{h>0} \|f^h \|_{L^2(\Omega \times [0,1])} +\|f^h \|_{L^2(\partial \Omega \times[0,1])}  < +\infty$, we have
\begin{align} \label{weakPDEconstraint}
\int_0^1 \int_{\Omega} \left( D_t^{-h}f^h  {\rho}^h + D_x^{-h} f^h \cdot m^h   \right)  \rd x \rd t+ \int_{\Omega} \left(  f^h(\cdot,0) \rho^{h}(\cdot,0)- f^h(\cdot,1) \rho^h( \cdot,1)  \right)  \rd x \rd t\xrightarrow{h \to 0} 0\,.
\end{align}
\end{lemma}

\begin{proof}
By hypothesis (H\ref{initialhyp}),  $ \rho^h_0 \to \rho_0$ uniformly on $\overline{\Omega}$. Likewise, the constraint on the initial data (\ref{BB82ii}) and  (H\ref{deltahyp}) ensure  $\lim_{h \to 0} \| \rho^h(\cdot, 0) -  {\rho}_0^h \|_{L^2(\Omega)} \leq \lim_{h \to 0} \delta_4 = 0$. Thus,  $\rho^h(\cdot, 0) \to  \rho_0$ in ${L^2(\Omega)}$. 

We now turn to equation (\ref{weakPDEconstraint}). By the PDE constraint and boundary conditions  (\ref{BB82i}) and  summation by parts, via hypotheses (H\ref{discreteophyp}),
\begin{align*}
&\left| \int_0^1 \int_{\Omega} \left( D_t^{-h}f^h  {\rho}^h + D_x^{-h} f^h \cdot m^h   \right)  \rd x \rd t + \int_{\Omega} \left(  f^h(\cdot,0) \rho^{h}(\cdot,0)- f^h(\cdot,1) \rho^h( \cdot,1)  \right)  \rd x \right| \\
&\quad=  \left| \int_0^1 \int_\Omega \left( f^h D_t^h \rho^h + f^h D_x^{h} m ^h \right) \rd x \rd t - \int_0^1 \int_{\partial \Omega } f^h m^h \cdot \nu^h \rd x \right| \\
& \quad \leq \|f^h \|_{L^2(\Omega \times[0,1])} \| D_t^h \rho^h+ D_x^{h} m ^h \|_{L^2(\Omega \times[0,1])} + \|f^h \|_{L^2(\partial \Omega \times [0,1])} \|m^h \cdot \nu^h \|_{L^2(\partial \Omega \times [0,1])} \\
&\quad \leq \delta_4 \|f^h \|_{L^2(\Omega \times[0,1])} + \delta_2 \|f^h \|_{L^2(\partial \Omega \times[0,1])} \xrightarrow{h \to 0} 0 ,
\end{align*}
where, in the last line, we use that (H\ref{deltahyp}) ensures $\delta_2, \delta_4 \to 0$ and the fact that $f^h$ is bounded uniformly in $h$ in $L^2(\Omega \times [0,1]$ and $L^2(\partial \Omega \times[0,1])$.

Next, we show that  there exist $\rho \in \P(\Omega \times [0,1])$ and $\mu \in \P(\Omega)$ so that, up to a subsequence, $\rho^h \wsto \rho$ and $\rho^h(\cdot, 1) \wsto \mu$.
By H\"older's inequality and the mass constraint (\ref{BB82ii}),
\begin{align*}
  \left\| \int_\Omega \rho^h(x,\cdot ) \rd x -  \int_{\Omega}  {\rho}^h_0(x) \rd x \right\|_{L^1([0,1])} \leq  \left\| \int_\Omega \rho^h(x,\cdot ) \rd x -  \int_{\Omega}  {\rho}^h_0(x) \rd x \right\|_{L^2([0,1])}  \leq  \delta_3 \xrightarrow{h \to 0}  0 ,
\end{align*}
where, in the last line, we use that (H\ref{deltahyp}) ensures $\delta_3 \to 0$.  
Since hypothesis (H\ref{initialhyp}) ensures $\rho_0^h \to \rho_0$ uniformly and $\int_\Omega \rho_0 = 1$, we obtain,
\begin{align*} 
 \int_0^1 \int_\Omega \rho^h(x, s) \rd x \rd s \to 1 .
\end{align*}
Furthermore, since  $\int_0^1 \int_\Omega \Phi(\rho^h,m^h) < +\infty$ for each $h >0$, we must have  $\rho^h \geq 0$ on $\Omega \times [0,1]$, and  the above equation ensures $\sup_{h >0} \| \rho^h\|_{L^1(\Omega \times[0,1])} < +\infty$. Thus,  classical functional analysis results ensure there exists a subsequence that converges to some $\rho \in \P(\Omega \times [0,1])$ in the   weak-* topology  (see, e.g., \cite[Section 3]{brezis2010functional}).

Finally, taking $f^h \equiv 1$ in equation (\ref{weakPDEconstraint}) gives,
\begin{align*}
 \lim_{h \to 0} \int_\Omega     \rho^{h}(\cdot,0) - \rho^{h}(\cdot,1)  \rd x= 0 \quad \implies \quad \lim_{h \to 0} \int_\Omega \rho^{h}(\cdot,1) \rd x = 1  \implies \sup_{h >0} \|\rho^h(\cdot, 1)\|_{L^1(\Omega)} < +\infty.
  \end{align*}
Arguing as above, we obtain that, up to a further subsequence, $\rho^h(\cdot, 1) \wsto \mu(\cdot)$ for $\mu \in \P(\Omega)$. 
\end{proof}

We now prove that the  discrete energies $\energy^h$ are lower semicontinuous along weak-* convergent sequences.
\begin{proposition}[Lower semicontinuity  of energies along weak-* convergent sequences] \label{energygamma}
Suppose that hypotheses (H\ref{domainhyp})-(H\ref{smoothhyp})   hold. Then, for any sequence of piecewise constant functions $\rho^h: \Omega \to \mathbb{R}$ such that $\rho^h \wsto \rho$, we have $\liminf_{h \to 0} \energy^h(\rho^h) \geq  \energy(\rho )$.
\end{proposition}
\begin{proof}
First, suppose the energy satisfies (H\ref{energyhyp}a).
Since   the piecewise constant approximations $\hat{V}^h$ and $\hat{W}^h$  converge to $V$ and $W$ uniformly, for any sequence $\rho^h \wsto \rho$,
\begin{align} \label{Vargue}
& \lim_{h \to 0} \int_\Omega V^h \rho^h \rd x =   \int (V^h - V) \rho^h \rd x + \int V \rho^h  \rd x = \int V \rho ~ \rd x , \\
& \lim_{h \to 0} \int_{\Omega \times \Omega} W^h(x,y) \rho^h(x) \rho^h(y) \rd x  \rd y = \int_{\Omega \times \Omega} W(x-y) \rd \rho(x) \rd \rho(y) . \label{Wargue}
\end{align}
Furthermore,  our assumptions on $U$ guarantee that the internal energy term is lower semicontinuous with respect to weak-* convergence \cite[Remark 9.3.8]{AGS}, so $\liminf_{h \to 0} \int_\Omega U(\rho^h(x)) \rd x \geq \int_\Omega U(\rho(x)) \rd x$. Combining this with equations (\ref{Vargue}-\ref{Wargue}) gives the result.

Next, suppose the energy satisfies (H\ref{energyhyp}b).  Since  $\rho_0>0$ on the compact set $\overline{\Omega}$ and $U'$ is uniformly continuous on  $\rho_0(\overline{\Omega}) \subset (0, +\infty)$, the fact that hypothesis (H\ref{initialhyp}) ensures $\hat{\rho}_0^h \to \rho_0$ uniformly ensures $U'(\hat{\rho}_0^h) \to U'(\rho_0)$ uniformly. Therefore,
\begin{align*} \lim_{h \to 0} \int_\Omega U'(\hat{\rho}_0^h) \rho^h  \rd x= \int_\Omega U'(\rho_0) \rho \rd x. 
\end{align*}
Likewise, since  $\hat{V}^h$ and $\hat{W}^h$ converge to $V$ and $W$ uniformly, we also have
\begin{align}
 \lim_{ h \to 0} \int_\Omega \left(  \hat{V}^h(x)  +\int_\Omega \hat{W}^h(x,y) \rho^h_0(y) \rd y  \right) \rho^h(x)\rd x =  \int_\Omega \left(   V(x)  +\int_\Omega W(x,y) \rho_0(y) \rd y  \right) \rho(x) \rd x .  \label{VWargue2}
 \end{align}
Combining these   limits with the $\liminf$ inequality for energies of the form (H\ref{energyhyp}a) gives the result.

Finally, suppose the energy satisfies (H\ref{energyhyp}c). Without loss of generality, we may assume that $\liminf_{h \to 0} \G^h_{\rho_1}(\rho^h) <+\infty$, so that up to a subsequence, $\G^h_{\rho_1}(\rho^h) \equiv 0$ and  $ \lim_{h \to 0} \|\rho^h - \rho^h_1 \|_{L^2(\Omega)} = 0$.
By uniqueness of limits,  $\rho = \rho_1$. Thus, since $\G^h_{\rho_1} \geq 0$, we have $\liminf_{h \to 0} \G^h_{\rho_1}(\rho^h) \geq 0 = \G_{\rho_1}(\rho) $.

\end{proof}

We now apply Proposition \ref{energygamma} to prove the $\Gamma$-convergence of  Problem \ref{ddynamicJKO2}   to Problem \ref{dynamicJKO}.

\begin{theorem}[$\Gamma$-convergence of discrete to continuum JKO] \label{gammaconvergence}
Suppose hypotheses (H\ref{domainhyp})-(H\ref{smoothhyp}) hold.

\begin{enumerate}[(a)]
\item If $(\rho^h, m^h) \in \mathcal{C}^h$  with $(\rho^h, m^h) \wsto (\rho,m)$, then $(\rho,m) \in \mathcal{C}$  and
\[ \liminf_{h \to 0} \int_0^1 \int_\Omega \Phi( \rho^h, m^h) \rd x \rd t + 2 \tau  \energy^h(\rho^h(\cdot, 1)) \geq \int_0^1 \int_\Omega \Phi( \rho , m )  \rd x \rd t +  2 \tau \energy(\rho(\cdot, 1)) . \] \label{liminfineq}
\item For any $(\rho, m) \in \mathcal{C}$   satisfying $\rho \in C^2([0,1]; C^1(\overline{\Omega}))$, $\rho >0$, and $m \in C([0,1]; C^2(\overline{\Omega})$, there exists a sequence $(\tilde{\rho}^h,
\tilde{m}^h) \in \mathcal{C}^h$   so that $(\tilde{\rho}^h ,\tilde{m}^h) \to (\rho,m)$ uniformly and
\[ \limsup_{h \to 0} \int_0^1 \int_\Omega \Phi( \tilde{\rho}^h, \tilde{m}^h) \rd x \rd t  +  2 \tau \energy^h(\tilde{\rho}^h(\cdot, 1)) \leq \int_0^1 \int_\Omega \Phi( \rho , m )  \rd x \rd t +  2 \tau \energy(\rho(\cdot, 1))  . \] \label{limsupineq}
\end{enumerate}
\end{theorem}
\begin{proof}
We first prove part (\ref{liminfineq}). Suppose $(\rho^h, m^h) \in \mathcal{C}^h$,  with $\rho^h \wsto \rho$ and $m^h \wsto m$. We begin by showing that the limit $(\rho,m)$ belongs to $\mathcal{C}$. Fix $f \in C^\infty(\Omega \times [0,1])$ and let $f^h$ be a pointwise piecewise constant approximation of $f$. (See equation (\ref{pointwisepiecewise}).) By Lemma \ref{weakPDElem} and hypothesis (H\ref{discreteophyp}),
\begin{align*}
\int_0^1 \int_\Omega \left( f_t  {\rho} +\grad f \cdot m \right) \rd x \rd t  + \int_\Omega f(\cdot, 0) \rho(\cdot, 0) - f(\cdot, 1) \mu  \rd x  = 0 .
\end{align*}
We conclude that $(\rho,m)$ satisfies the PDE constraint in the sense of distributions (\ref{BB8}), which gives  $ {\rho} \in AC([0,1], \P(\Omega))$ \cite[Lemma 8.1.2]{AGS}. In particular, since $\rho$ is continuous in time, we have that the $\mu$ defined in Lemma \ref{weakPDElem} satisfies $\mu = \rho(\cdot, 1)$.

We now consider the inequality in part (a). Since the integral functional $(\rho,m) \mapsto \int_0^1 \int_\Omega \Phi(\rho,m)$ is lower semicontinuous with respect to weak-* convergence of measures \cite[Example 2.36]{ambrosio2000functions}, we immediately obtain
\begin{align*} 
\liminf_{h \to 0} \int_0^1 \int_\Omega \Phi( \rho^h, m^h)  \rd x \rd t \geq \int_0^1 \int_\Omega \Phi( \rho , m ) \rd x \rd t   . \end{align*}
This ensures $m \in L^1([0,1],L^2(\rho^{-1}))$ and completes the proof that $(\rho,m) \in \mathcal{C}$.
Finally, since Lemma \ref{weakPDElem} ensures $\rho^h(\cdot, 1) \wsto \mu = \rho(\cdot, 1)$, applying Proposition \ref{energygamma}  gives
\[ \liminf_{h \to 0}    \energy^h(\rho^h(\cdot, 1))  \geq   \energy(\rho(\cdot, 1)) , \]
which completes the proof of part (\ref{liminfineq}).

We now turn to part (\ref{limsupineq}). Let $(\tilde{\rho}^h, \tilde{m}^h) \in \mathcal{C}^h$ be the sequence constructed in Proposition \ref{specialelement}, so $(\tilde{\rho}^h, \tilde{m}^h) \to (\rho,m)$ uniformly. By inequality (\ref{rhohunifpos}), there exists $c>0$ so that  $ {\rho}^h(x,t) \geq c $ for $h$ sufficiently small.  Therefore, 
\begin{align*}
\left| \int_0^1 \int_\Omega \Phi(\tilde{\rho}^h, \tilde{m}^h) - \int_0^1 \int_\Omega \Phi(\rho, m) \right| 
&\leq |\Omega| \| \grad_{\rho, m} \Phi \|_{L^\infty(\{ \rho \geq c\})} \left( \|\tilde{m}^h - m\|_{\infty} +  \|\tilde{\rho}^h - \rho\|_{\infty} \right) \xrightarrow{h \to 0} 0.
\end{align*}

It remains to show that
\begin{align*} 
 \limsup_{h \to 0} \energy^h (\tilde{\rho}^h(\cdot, 1)) \leq \energy (\rho(\cdot, 1)) . 
 \end{align*}
 First, suppose  the energy satisfies either (H\ref{energyhyp}a) or (H\ref{energyhyp}b). By equations (\ref{Vargue})-(\ref{VWargue2}), which hold for any weak-* convergent sequence, and the fact that $U'(\tilde{\rho}^h(\cdot, 0)) \to U'(\rho(\cdot, 0) )$ uniformly,   it suffices to show 
 \[ \limsup_{h \to 0} \int_\Omega U(\tilde{\rho}^h(\cdot,1)) \rd x   \leq \int_\Omega U(\rho(\cdot,1)) \rd x. \]
 Since $U \in C([0, +\infty])$, $\rho(\cdot, 1) \in L^\infty(\Rd)$, and $\tilde{\rho}^h(\cdot, 1) \to \rho(\cdot,1)$ uniformly, $ U(\tilde{\rho}^h(\cdot,1))  \to \int U(\rho(\cdot,1)) $ uniformly, which gives the result.

Finally, suppose  the energy satisfies (H\ref{energyhyp}c). Without loss of generality, suppose $\energy(\rho(\cdot,1)) = \G_{\rho_1}(\rho(\cdot, 1))< +\infty$. Inequality (\ref{specialelementc}) ensures that, for $h$ sufficiently small,
\[ \| \tilde{\rho}^h(\cdot, 1) - \rho^h_1\|_{L^2(\Omega)} \leq \delta_5 . \]
By definition of  $\G_{\rho_1}^h$, this implies $\G^h_{\rho_1}(\tilde{\rho}^h(\cdot, 1)) \equiv 0$. Therefore,
\[ \limsup_{h \to 0} \G^h_{\rho_1}(\tilde{\rho}^h(\cdot, 1)) = 0 \leq \G_{\rho_1}(\rho_1) , \]
which gives the result.
\end{proof}

We conclude this section by applying the $\Gamma$-convergence proof from Theorem \ref{gammaconvergence} to prove that any sequence of minimizers of the discrete Problem \ref{ddynamicJKO2}   converges, up to a subsequence, to a minimizer of the continuum Problem \ref{dynamicJKO}.

\begin{theorem}[Convergence of minimizers] \label{convergence}
Suppose that hypotheses (H\ref{domainhyp})-(H\ref{initialhyp})   hold. Then, for any  sequence of minimizers $(\rho^h, m^h)$ of Problem \ref{ddynamicJKO2}, we have,   up to a subsequence, $\rho^h  \wsto \rho $   and $m^h  \wsto m $, where  $(\rho,m)$ is a minimizer of Problem \ref{dynamicJKO}.
\end{theorem}

Note that, if the minimizer of the continuum Problem \ref{dynamicJKO} is unique, then this theorem ensures that any sequence of minimizers of the discrete Problem \ref{ddynamicJKO} has a further subsequence that converges to this minimizer. Therefore, the sequence itself must converge to the unique minimizer of the continuum problem. (See Remark \ref{uniquenessmincont} for sufficient conditions that ensure the minimizer of the continuum problem is unique.)

\begin{proof}[Proof of Theorem \ref{convergence}]
First, note that Lemma \ref{weakPDElem} ensures that there exist $\rho \in \P(\Omega \times[0,1])$ and $\mu \in \P(\Omega)$ so that, up to a subsequence, $\rho^h \wsto \rho$ and $\rho^h(\cdot, 1) \wsto \mu$. 
In order to prove an analogous weak-* compactness result for $m^h$ we  first prove that, up to a subsequence,
\begin{align}
 \sup_{h>0}\int_0^1 \int_\Omega \Phi( \rho^h, m^h) <+\infty. \label{energybound}
 \end{align}
 By (H\ref{smoothhyp}), there exists a minimizer $(\bar{\rho}, \bar{m})$ of the continuum Problem \ref{dynamicJKO} satisfying $ \bar{\rho} \in C^2([0,1]; C^1(\overline{\Omega}))$, $ \bar{\rho} >0$, and $ \bar{m} \in C^1([0,1]; C^2(\overline{\Omega}))$. Comparing the recovery sequence  $( \tilde{\rho}^h, \tilde{m}^h) \in \mathcal{C}^h$ from  Theorem \ref{gammaconvergence}(\ref{limsupineq}) for $(\bar{ \rho}, \bar{ m})$  with the discrete minimizer $(\rho^h,m^h) \in \mathcal{C}^h$, we obtain
\begin{align} \label{almost2}
\limsup_{h \to 0} \int_0^1 \int_\Omega \Phi( {  \rho}^h,  {m}^h)  + 2\tau \energy^h( {\rho}^h(\cdot, 1)) &\leq \limsup_{h \to 0} \int_0^1 \int_\Omega \Phi( \tilde{ \rho}^h,  \tilde{ m}^h) +  2 \tau \energy^h( \tilde{ \rho}^h(\cdot, 1)) \\
&\leq \int_0^1 \int_\Omega \Phi( \bar{ \rho} , \bar{ m} ) +2 \tau  \energy(\bar{ \rho}(\cdot, 1))  .  \nonumber\end{align}
Furthermore, Proposition \ref{energygamma} ensures that
\[ \liminf_{h \to 0} 2 \tau  \energy^h(\rho^h(\cdot, 1)) \geq 2\tau \energy(\mu) , \]
which is bounded below by some constant, since hypothesis (H\ref{energyhyp}c) ensures $\energy \geq 0$ and hypotheses (H\ref{energyhyp}a) or (H\ref{energyhyp}b) ensures $\energy(\mu) >-\infty$, since $U$, $V$, and $W$ are bounded below and $U'$ is bounded below on the range of the strictly positive density $\rho_0$. Therefore, up to a subsequence, we obtain (\ref{energybound}).
  
We now deduce weak-* convergence of $m^h$.
By H\"older's inequality, the fact that $\rho^h \wsto \rho$, and the definition of $\Phi$, we have
\[ \sup_{h >0} \| m^h\|_{L^1(\Omega\times[0,1])}    \leq  \sup_{h >0 } \left( \int_0^1 \int_\Omega \Phi(\rho^h,m^h)  \right)  \left( \int_0^1 \int_\Omega 1^2 \rho^h \right)^{1/2} < +\infty .\]
Thus, up to another subsequence, $m^h \wsto m$ on $\Omega \times [0,1]$.  

It remains to show that the limit $(\rho,m)$ of $(\rho^h, m^h)$ is a minimizer of   Problem \ref{dynamicJKO}. By Theorem \ref{gammaconvergence}, part (\ref{liminfineq}), we have $(\rho, m) \in \mathcal{C}$ and
\begin{align*} 
\int_0^1 \int_\Omega \Phi( \rho , m ) + 2 \tau  \energy(\rho(\cdot, 1)) \leq \liminf_{h \to 0} \int_0^1 \int_\Omega \Phi( \rho^h, m^h) + 2 \tau  \energy^h(\rho^h(\cdot, 1))   .
\end{align*}
Combining this with inequality (\ref{almost2}) above, we conclude that  $(\rho,m) \in \mathcal{C}$ is also a minimizer of Problem \ref{dynamicJKO}, which completes the proof.
\end{proof}

\section{Numerical results}\label{sec:examples}

In this section, we provide several examples demonstrating the efficiency and accuracy of our algorithms. We begin by using algorithm \ref{alg:dis} to compute Wasserstein geodesics between given source and target measures, and we then turn to algorithm \ref{alg:nonlinearJKO} to compute solutions of nonlinear gradient flows. In the following simulations, we take our computational domain $\Omega$ to be a square, imposing the no flux boundary conditions on $m$ dimension by dimension. In practice, unless otherwise specified, we always impose the discrete PDE constraint via the Crank-Nicolson finite difference operators \eqref{crankspace}, and we choose $\epsilon_1 = \epsilon_2 = \epsilon$ in the stopping criteria to be $10^{-5}$ unless otherwise specified. For the relaxation of the constraints in \eqref{BB81} and \eqref{BB82}, we choose $\delta_1 = \delta _2 = \delta_4 = \delta_5 = \delta$, and $\delta_3$ differently, as specified in each example. 
\subsection{Wasserstein geodesics}
As described in Remark \ref{W2georemark}, a particular case of our numerical scheme provides a method for computing the Wasserstein geodesic between two probability densities.
We begin by computing the Wasserstein geodesic between rescaled Gaussians in one dimension:
\begin{equation} \label{gaussian}
g_{\mu,\theta}(x) = \frac{1}{(2\pi \theta^2)^{d/2}} e^{-\frac{(x-\mu)^2}{\theta^2}}\,.
\end{equation}
The target measure is simply a translation and dilation of the initial measure, $\rho_0(x)=(0.5) g_{\mu_0, \theta_0 }(x)$ and $\rho_1(x) = (0.5) g_{\mu_1, \theta_1}(x)$.
The optimal transport map $T(x)$ from $\rho_0(x)$ to $\rho_1(x)$ is given explicitly by\footnote{One way to see that this is the unique optimal transport map from $\rho_0$ to $\rho_1$ is to note that $T \# \rho_0 = \rho_1$ and $T(x)$ is the gradient of a convex function; see, for example, \cite[Section 6.2.3]{AGS}.}
\[ 
T(x) = \frac{\theta_1}{\theta_0} (x-\mu_0) + \mu_1 . 
\]
Rewriting equation (\ref{W2geoform}) for the geodesic $\rho(x,t)$ and velocity $v(x,t)$ induced by this transport map, via the definition of the push forward, we obtain
\begin{align*} 
\rho(x,t)= \rho_0(T_t^{-1}(x)) \textrm{det}(\nabla_x T^{-1}_t)  \quad &\text{ and } \quad m(x,t) = \rho(x,t) v(x,t) =  \rho(x,t)  (T\circ T_t^{-1}(x) - T_t^{-1}(x))) , \\
T_t^{-1}(x) = &\, \frac{x+(\frac{\theta_1}{\theta_0}\mu_0-\mu_1)t}{ 1-t+ t \frac{\theta_1}{\theta_0}},  \  \textrm{det}(\nabla_x T_t^{-1}) = \frac{1}{1-t+t\frac{\theta_1}{\theta_0}} .
\end{align*}

\begin{figure}[h!]
\centering
 \textbf{Wasserstein geodesic between Gaussians } \\
\includegraphics[width=0.48\textwidth,trim={.5cm .05cm 1.3cm .6cm},clip]{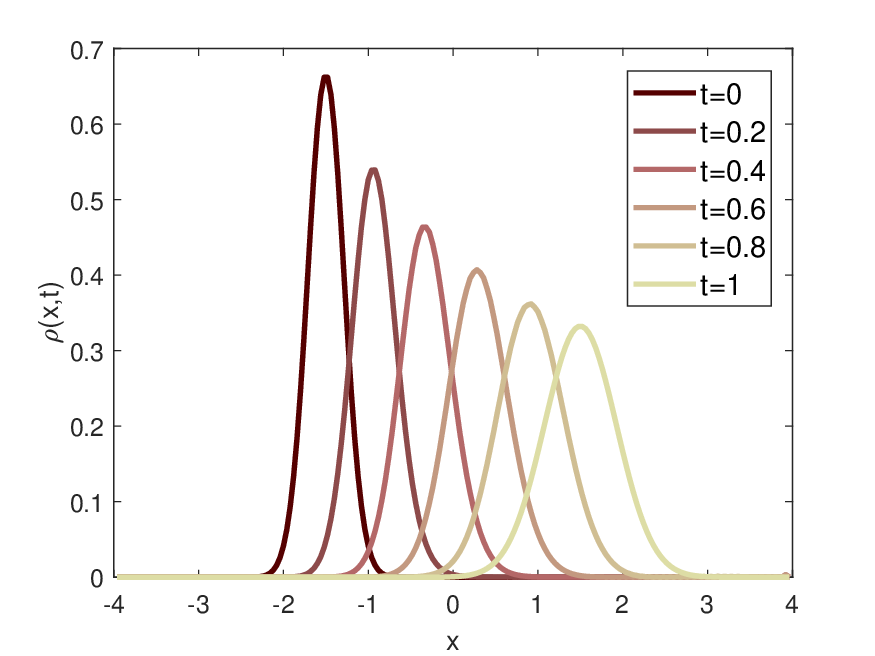}  \quad\includegraphics[width=0.48\textwidth,,trim={.5cm .05cm 1.3cm .6cm},clip]{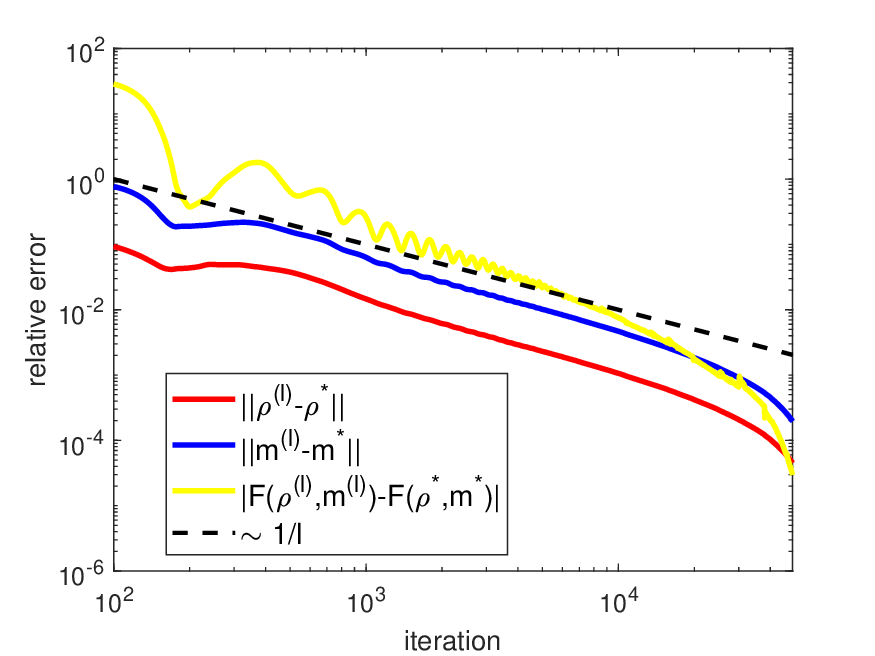}
\caption{We compute the Wasserstein geodesic between two Gaussians on the domain $\Omega = [-4,4]$, with $N_t = 20$ temporal grid points ($\Delta t = \frac{1}{20}$)  and $N_x = 200$ spatial points ($\Delta x = \frac{8}{200}$). We choose  $\sigma=0.1$ and $\sigma \lambda=1.5/\lambda_{max}(AA^t)$ and compute $10^5$ iterations. Left: evolution of geodesic from time $t =0$ to $t=1$. Right: rate of convergence of numerical solution to exact solution, as a function of the number of iterations in algorithm \ref{alg:dis}. }
\label{fig:dis1Da}
\end{figure}

In Figure \ref{fig:dis1Da}, we apply algorithm \ref{alg:dis} to compute the Wasserstein geodesic $\rho(x,t)$ between the initial and target densities (\ref{gaussian}), with means and variances $\mu_0 = -1.5, \theta_0 = 0.3, \mu_1 = 1.5,$ and $\theta_1 = 0.6$. On the left, we plot the evolution of the geodesic at various times. On the right, we plot the $\ell^1$ error of the densities, momenta, and Wasserstein distance  as a function of the number of iterations, $l$, observing a rate of convergence of order $\mathcal{O}(1/l)$ (dashed black line). Here the error is defined as 
\begin{align} \label{defnoferroriter}
\|\rho ^{(l)}- \rho^* \| = \frac{1}{N_x(N_s+1)}\sum_{k=0}^{N_s}\sum_{j = 1}^{N_x-1} |\rho_{j,k}^{(l)} - \rho^*_{j,k}| \ .
\end{align}

\begin{figure}[h!]
\centering
 \textbf{Optimal scaling of relaxation parameter $\delta$} \\
\includegraphics[width=0.48\textwidth,trim={1cm .05cm 1.1cm .6cm},clip]{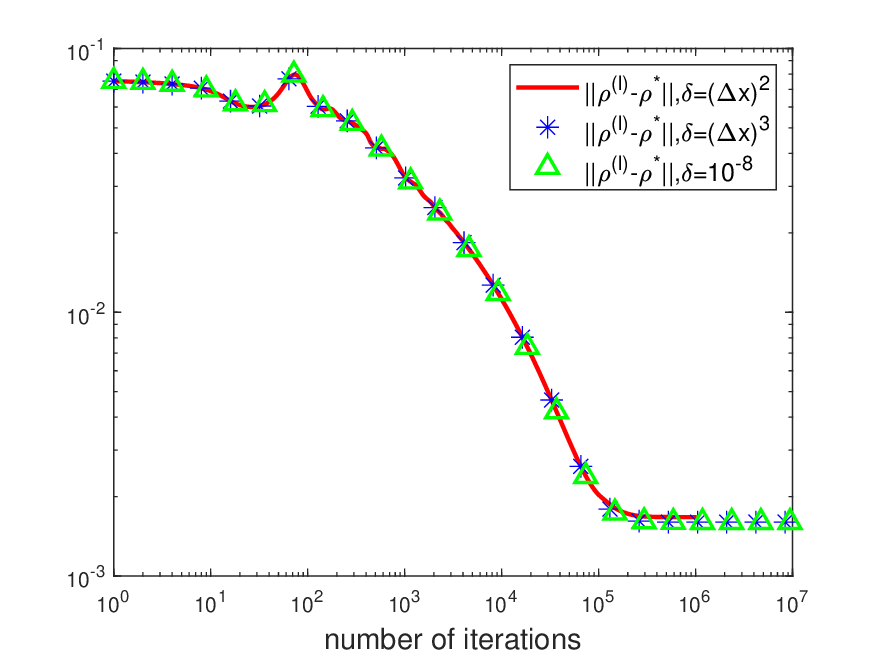}  \quad\includegraphics[width=0.48\textwidth,,trim={1cm .05cm 1.1cm .6cm},clip]{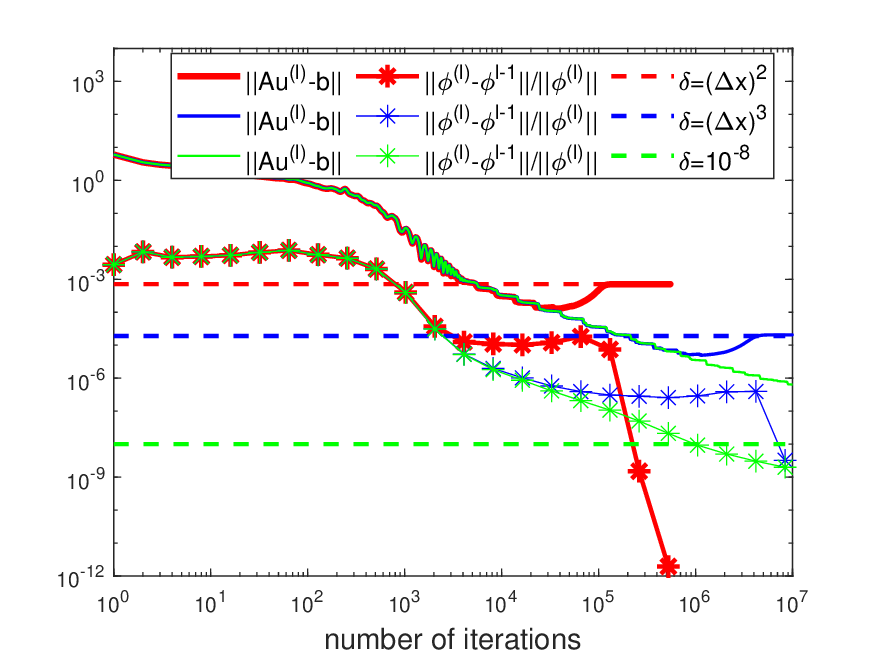}
\caption{Analysis of how the   scaling relationship between  the relaxation parameter $\delta$ and the spatial discretization $(\Delta x)$ affects the accuracy of the numerical method and the number of iterations required to converge. We contrast the choices $\delta = (\Delta x)^2$, $\delta = (\Delta x)^3$ and $\delta = 10^{-8}$ for the example of the Wasserstein distance between geodesics, illustrated in Figure \ref{fig:dis1Da}. We take  $N_t = 30$, $N_x = 300$, $\sigma=1$, $\sigma \lambda=0.99/\lambda_{max}(AA^t)$  and $\delta_3 = \delta$.}
\label{fig:optimaldelta}
\end{figure} 

In Figure \ref{fig:optimaldelta}, we illustrate how choosing the optimal scaling relationship between the relaxation parameter $\delta$ and the spatial and temporal discretizations $(\Delta x),  (\Delta t)$ allows the method to converge in fewer iterations. We contrast the choices $\delta  = (\Delta x)^2  $,  $\delta   = (\Delta x)^3$, and $\delta = 10^{-8}$, for the example of the Wasserstein distance between geodesics, illustrated in Figure \ref{fig:dis1Da}, where the outer time step $\tau = 1$, $(\Delta x) \sim (\Delta t)$, and $\delta_3 = \delta$. Based on the order of accuracy of our Crank-Nicolson approximation of the PDE constraint, we expect that  $\delta = (\Delta x)^2$ should give the optimal balance between accuracy and computational efficiency. (See Remark \ref{relaxpderemark}.) 

In the  plot on the left, we observe that for all choices of $\delta$, the  error between the numerical solution $\rho^{(l)}$ and the exact solution $\rho^*$ is identical, with the error saturating after $10^5$ iterations. Thus all three choices of $\delta$ provide the same level of accuracy, and the best way to distinguish between them is to identify which choice of $\delta$ causes the stopping criteria (\ref{cri02}-\ref{cri03}) to be satisfied in the least number of excess iterations after $10^5$. The behavior of two key stopping criteria are shown in the plot on the right--- the PDE constraint $\| A u^{(l)} - b \|$ and the convergence monitor for the relative error of the dual variables $\| \phi^{(l)} - \phi^{(l-1)} \|/\|\phi^{(l)} \|$. Of the four stopping criteria we consider (PDE constraint and three convergence monitors), these two are the last to be satisfied in all of the numerical simulations contained in this manuscript, hence these determine when our method terminates its iterations.

For the case of $\delta = (\Delta x)^2$ (red lines), we indeed observe that the PDE constraint (solid line) satisfies its stopping criteria (dashed line) by $10^4$ iterations and the dual variables (starred line) satisfy their stopping criteria of $10^{-5}$ by $10^5$ iterations. On the other hand, for the cases of $\delta = (\Delta x)^3$ (blue lines) and $\delta = 10^{-8}$ (green lines), we see that while the dual variables (starred lines) have satisfied their stopping criteria of $10^{-5}$ by $10^4$ iterations, the PDE constraints (solid lines) do not satisfy their stopping criteria (dashed lines) until   later---it takes more than $10^5$ iterations for $\delta = (\Delta x)^3$ and more than $10^7$ iterations for $\delta = 10^{-8}$. This example shows that choosing $\delta$  without respecting the order of accuracy of the finite difference approximation in the PDE constraint, one wastes computational effort without improving the accuracy of the numerical solution.

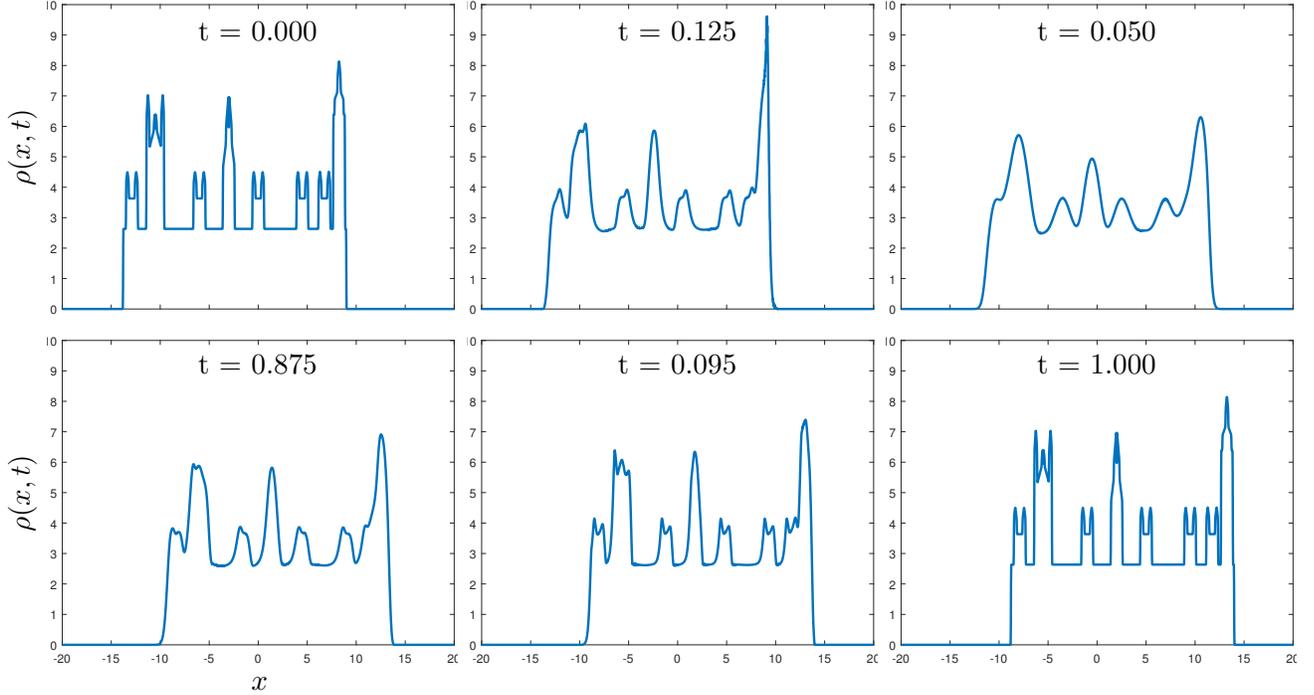
\begin{figure}[h!]
{\hspace{2cm} \textbf{Evolution of geodesic between translations of British Parliament} }
\\
\begin{picture}(155,130)
\put(15,0){\includegraphics[width=0.33\textwidth,trim={1.5cm 1.1cm 1.3cm .7cm},clip]{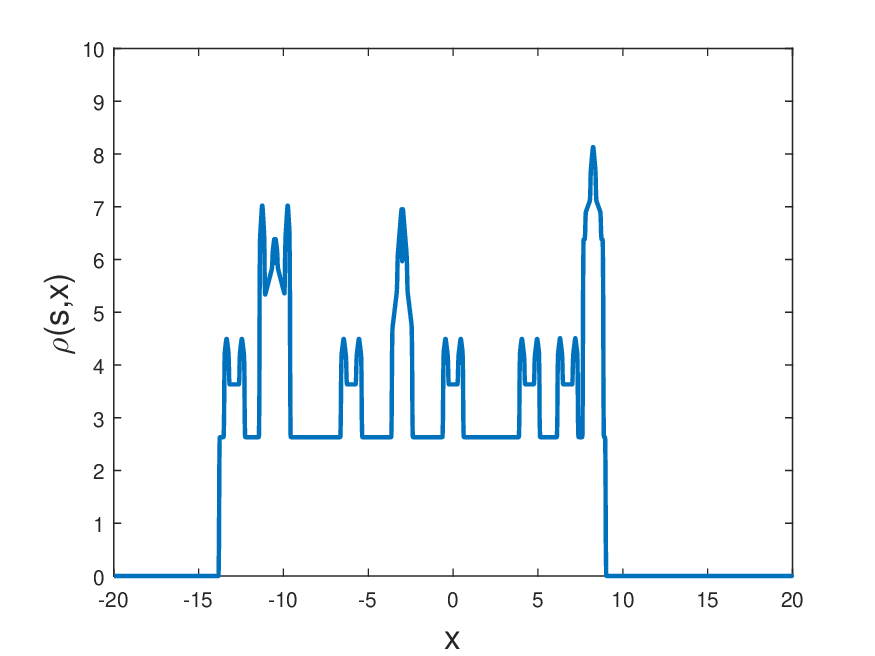}}
\put(72,105){t = 0.000}
\put(0,50){\rotatebox{90}{$\rho(x,t)$}}
\end{picture}
\begin{picture}(155,130)
\put(15,0){\includegraphics[width=0.33\textwidth,trim={1.5cm 1.1cm 1.3cm .7cm},clip]{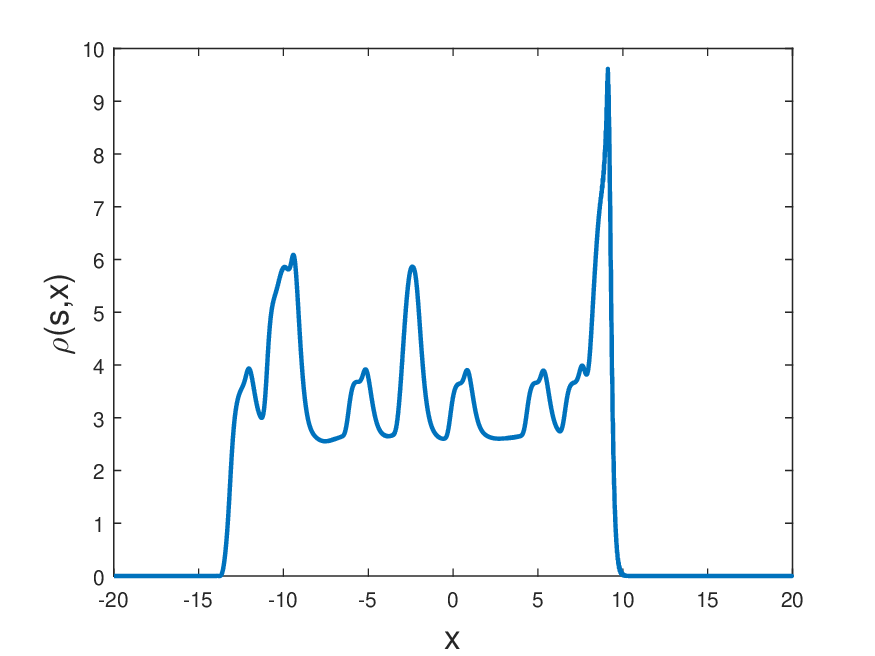}}
\put(72,105){t = 0.125}
\end{picture}
\begin{picture}(150,130)
\put(15,0){\includegraphics[width=0.33\textwidth,trim={1.5cm 1.1cm 1.3cm .7cm},clip]{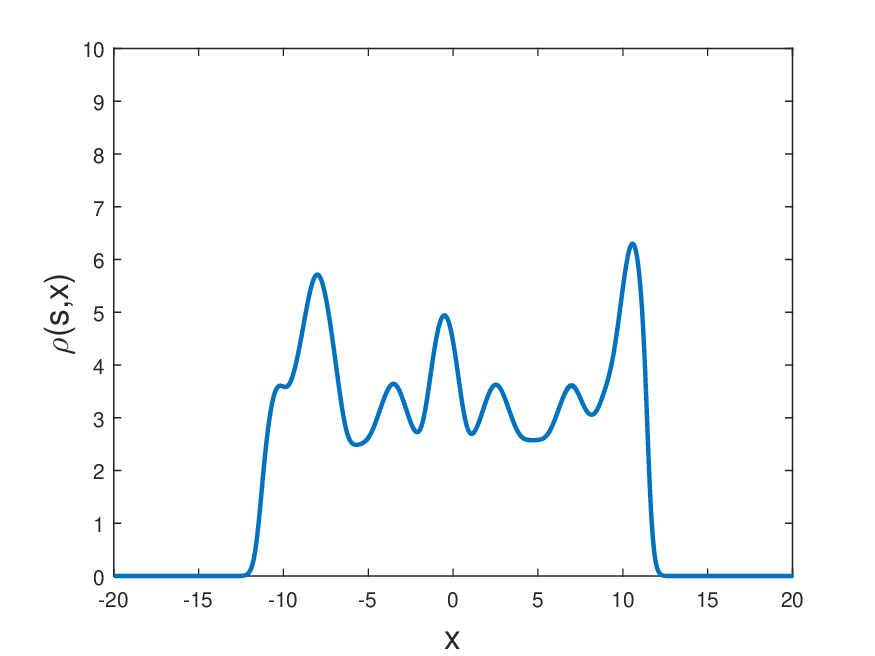}}
\put(72,105){t = 0.050}
\end{picture}
\\
\begin{picture}(155,130)
\put(15,0){\includegraphics[width=0.33\textwidth,trim={1.5cm .8cm 1.3cm .7cm},clip]{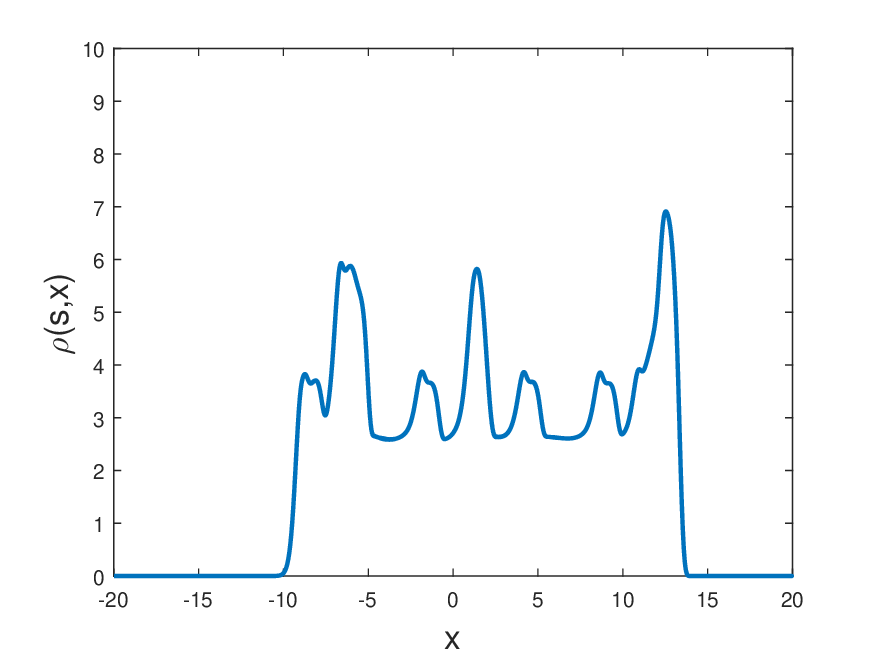}}
\put(72,110){t = 0.875}
\put(0,50){\rotatebox{90}{$\rho(x,t)$}}
\put(92,-10){$x$}
\end{picture}
\begin{picture}(155,130)
\put(15,0){\includegraphics[width=0.33\textwidth,trim={1.5cm .8cm 1.3cm .7cm},clip]{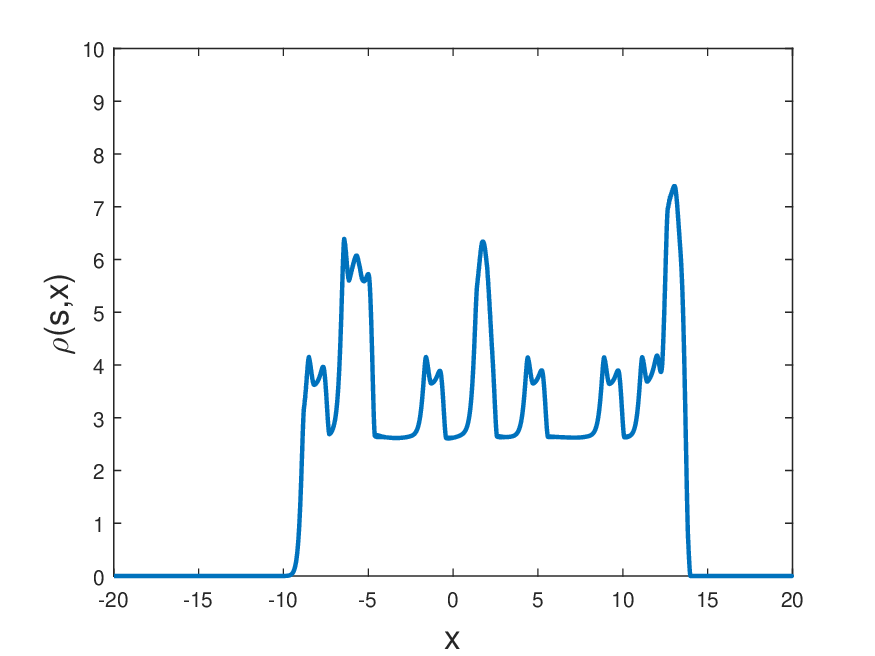}}
\put(72,110){t = 0.095}
\end{picture}
\begin{picture}(150,130)
\put(15,0){\includegraphics[width=0.33\textwidth,trim={1.5cm .8cm 1.3cm .7cm},clip]{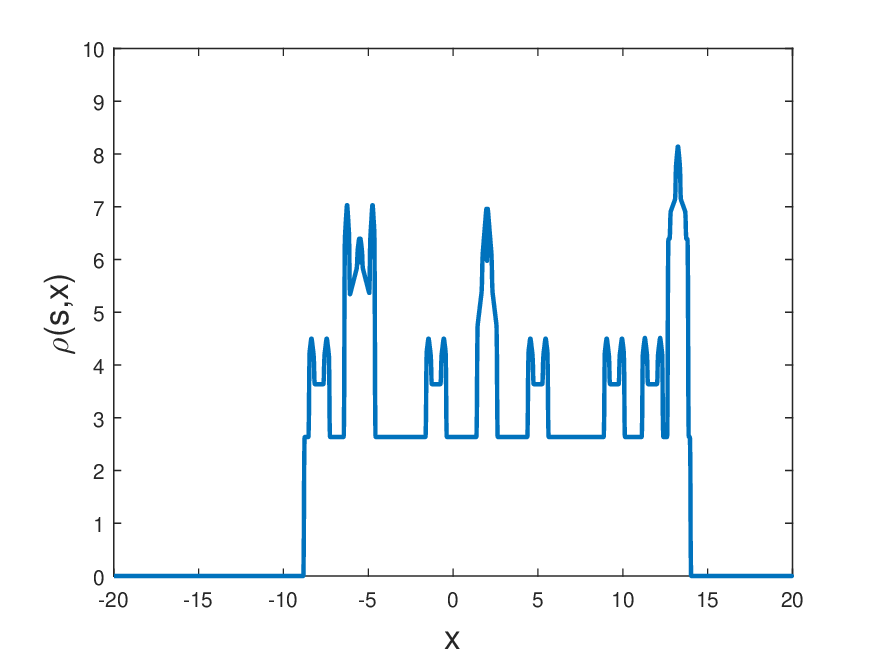}}
\put(72,110){t = 1.000}
\end{picture} \\

\caption{Computation of the Wasserstein geodesic between two translations of British parliament on the domain, with  $N_t = 40$ temporal grid points ($\Delta t = \frac{1}{40}$) and $N_x = 2000$ spatial grids ($\Delta x = \frac{40}{2000}$). Here $\sigma=0.1$, $\sigma\lambda=0.99/\lambda_{max}(AA^t)$ and then $\lambda=0.9727$, $\delta = 10^{-5}$, and $\delta_3 = 10^{-8}$. }\label{fig:palace}
\end{figure}

Next, we compute Wasserstein geodesics between initial and target measures when neither are smooth nor strictly positive. In Figure~\ref{fig:palace}, we compute the geodesic between a profile of the British Parliament and its translation. We do not observe convergence to the exact geodesic, which would be a constant speed translation, and instead observe degradation of the parliamentary building at intermediate times, due to numerical smoothing. Similarly, in Figure \ref{fig:pacman}, we compute the geodesic between Pac-Man and a ghost, visualized as characteristic functions on sets in two dimensions. Again, we observe numerical smoothing around the edges of discontinuity.
Both of these examples offer a  numerical justification for the smoothness assumption we impose in our main convergence Theorem \ref{convergence}.  In the absence of such smoothness, it appears that the method does not converge. Similar smoothness assumptions are required in the other numerical methods for Wasserstein geodesics for which rigorous convergence has been analyzed, including  Monge Amp\'ere type methods \cite{GM96, BFO14}.

\begin{figure}[h!]
\centering
{ \textbf{Evolution of geodesic between Pac-Man and a ghost} } \\
\includegraphics[width=0.2\textwidth,trim={3cm 2cm 2.5cm 1.25cm},clip]{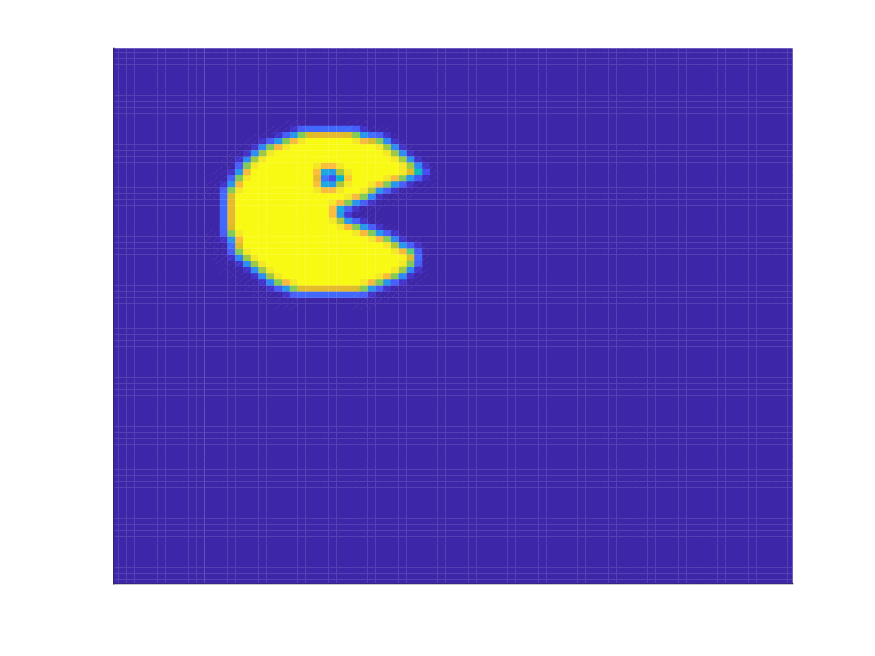}
\includegraphics[width=0.2\textwidth,trim={3cm 2cm 2.5cm 1.25cm},clip]{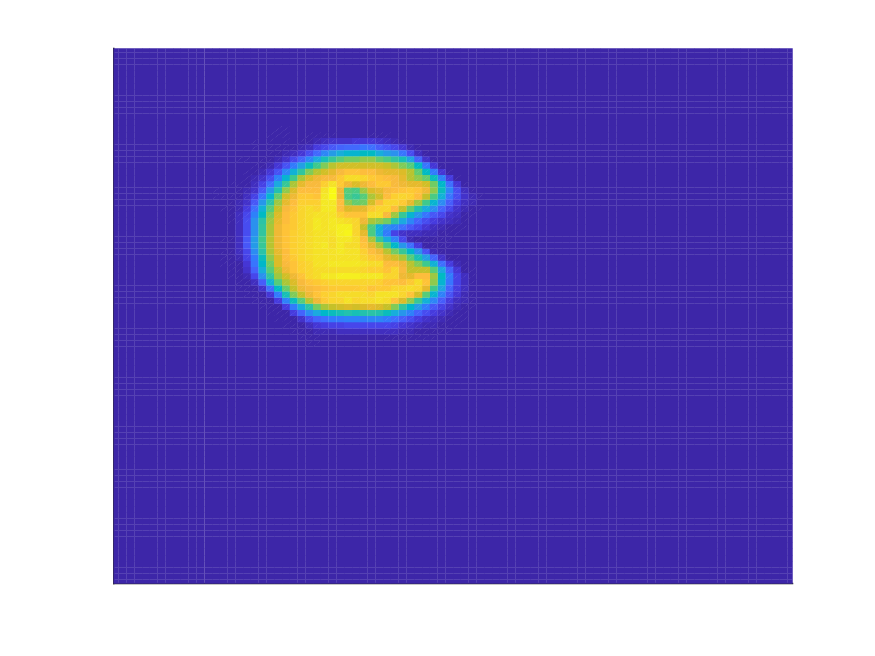}
\includegraphics[width=0.2\textwidth,trim={3cm 2cm 2.5cm 1.25cm},clip]{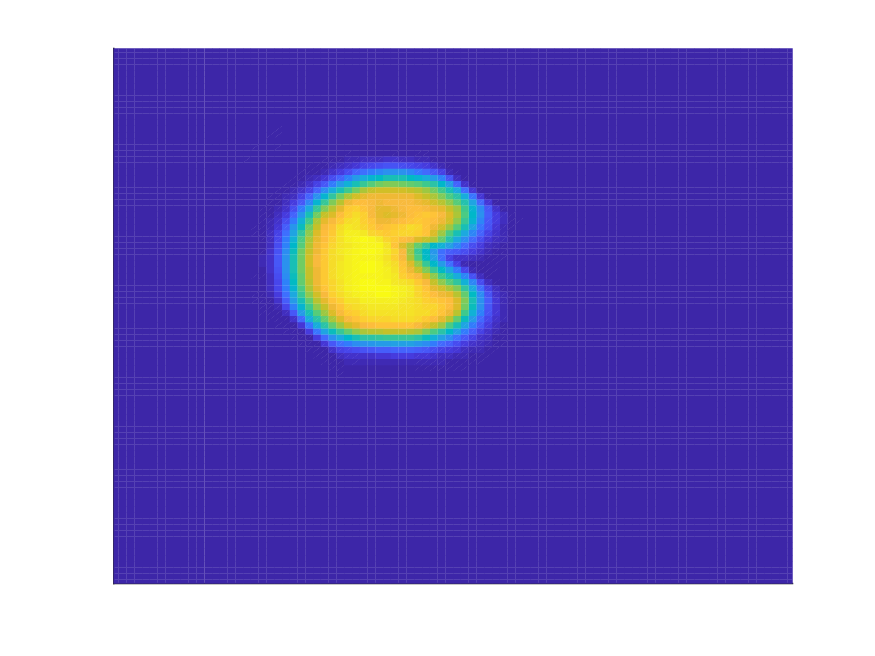}\\
\includegraphics[width=0.2\textwidth,trim={3cm 2cm 2.5cm 1.25cm},clip]{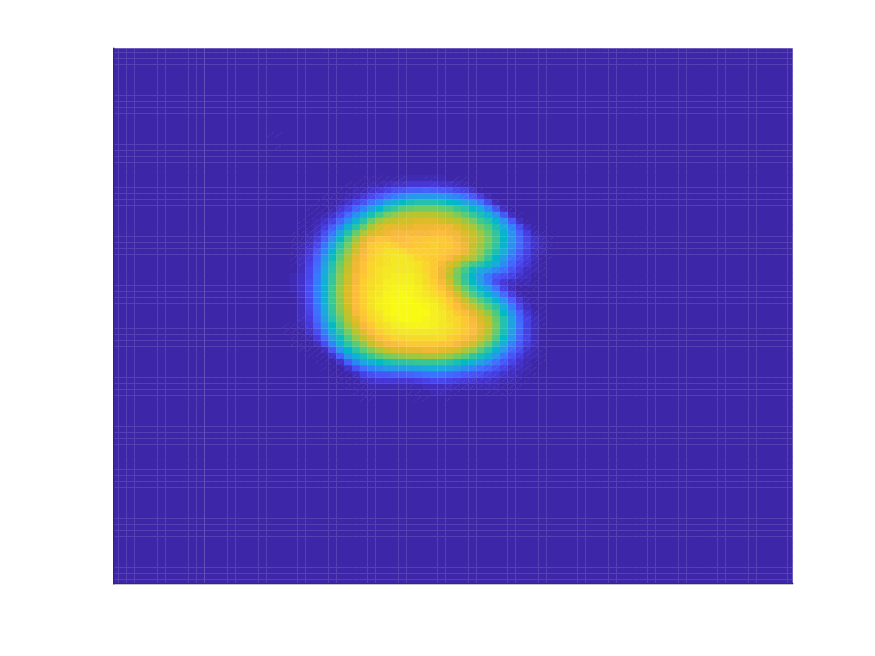}
\includegraphics[width=0.2\textwidth,trim={3cm 2cm 2.5cm 1.25cm},clip]{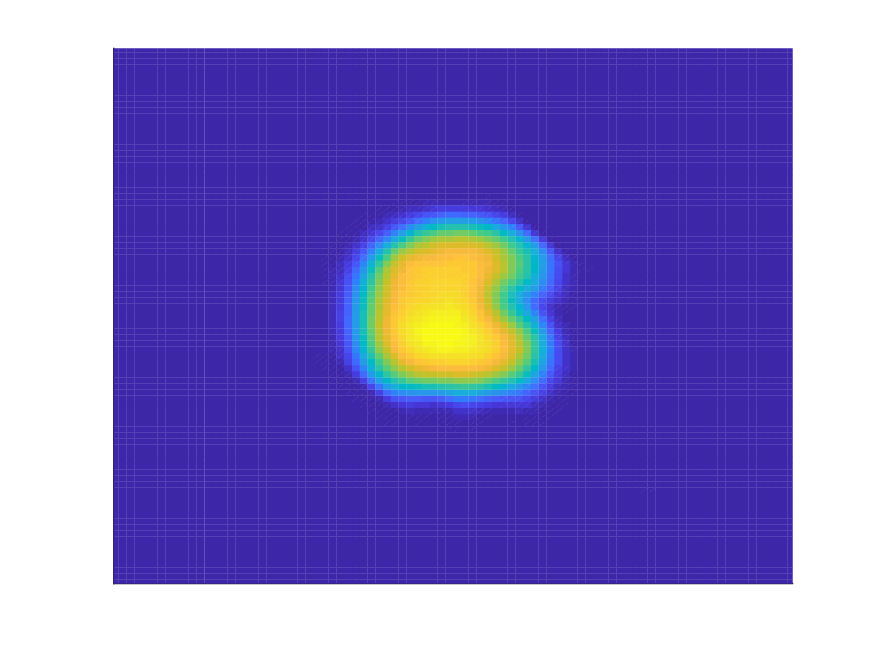}
\includegraphics[width=0.2\textwidth,trim={3cm 2cm 2.5cm 1.25cm},clip]{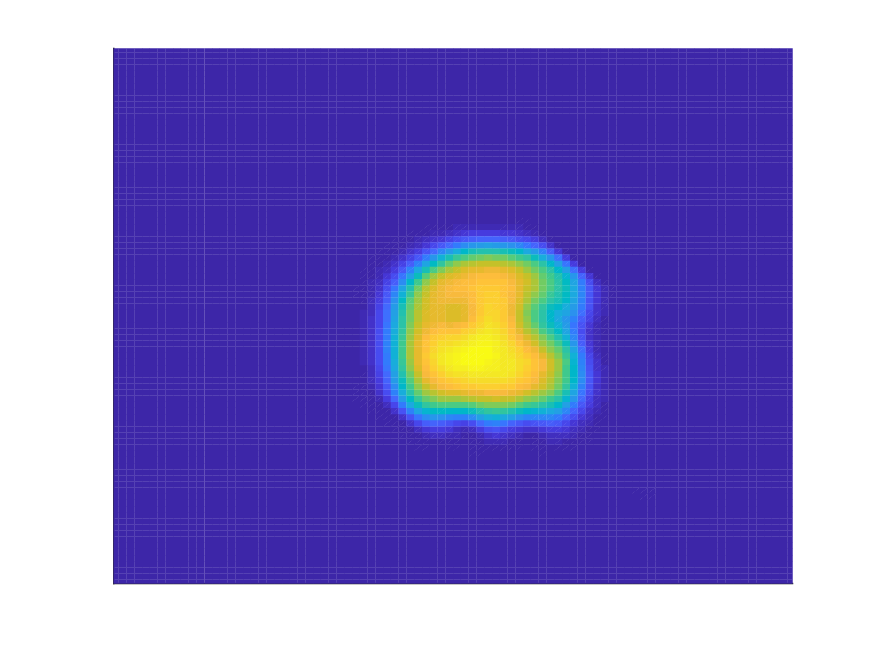}\\
\includegraphics[width=0.2\textwidth,trim={3cm 2cm 2.5cm 1.25cm},clip]{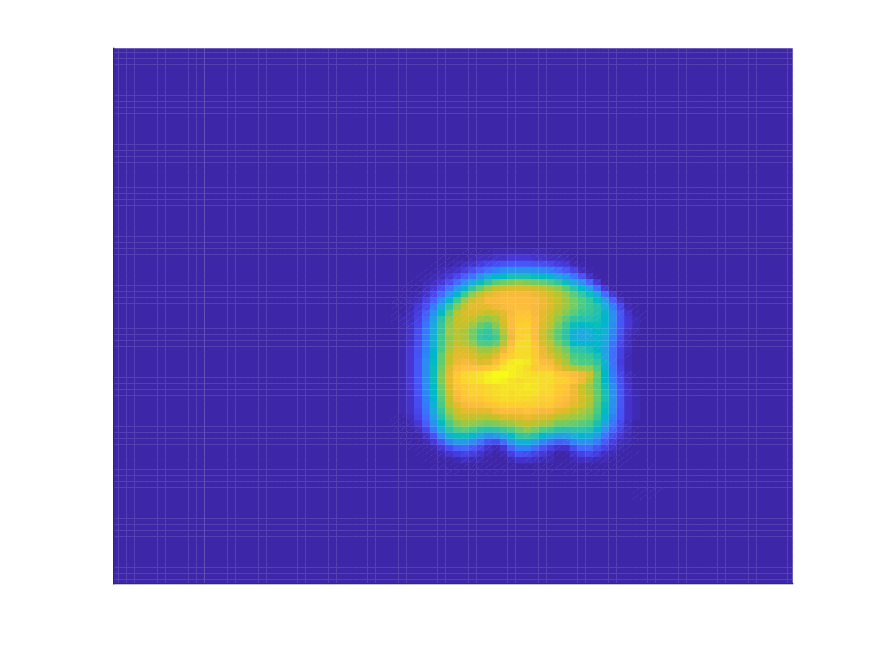}
\includegraphics[width=0.2\textwidth,trim={3cm 2cm 2.5cm 1.25cm},clip]{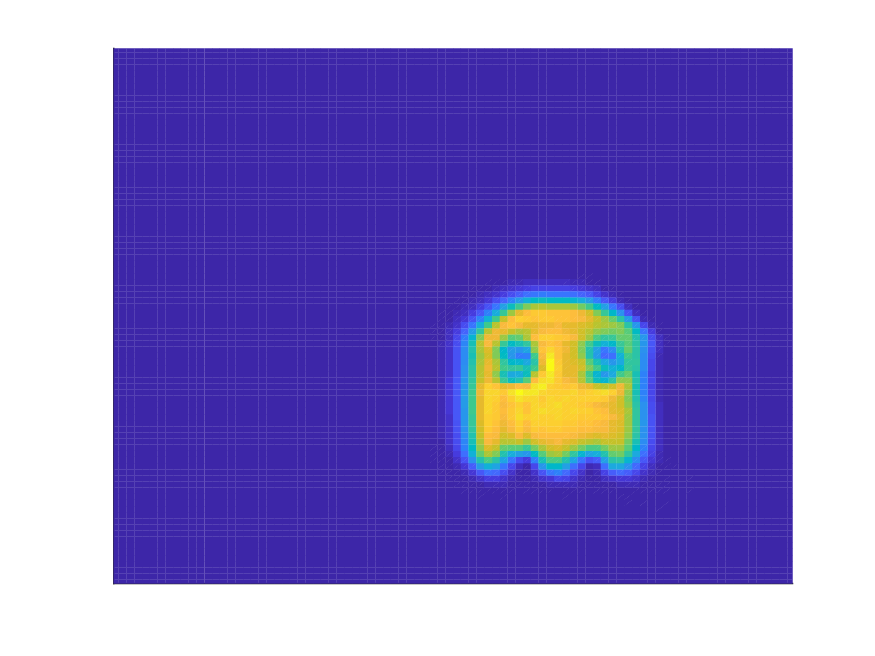}
\includegraphics[width=0.2\textwidth,trim={3cm 2cm 2.5cm 1.25cm},clip]{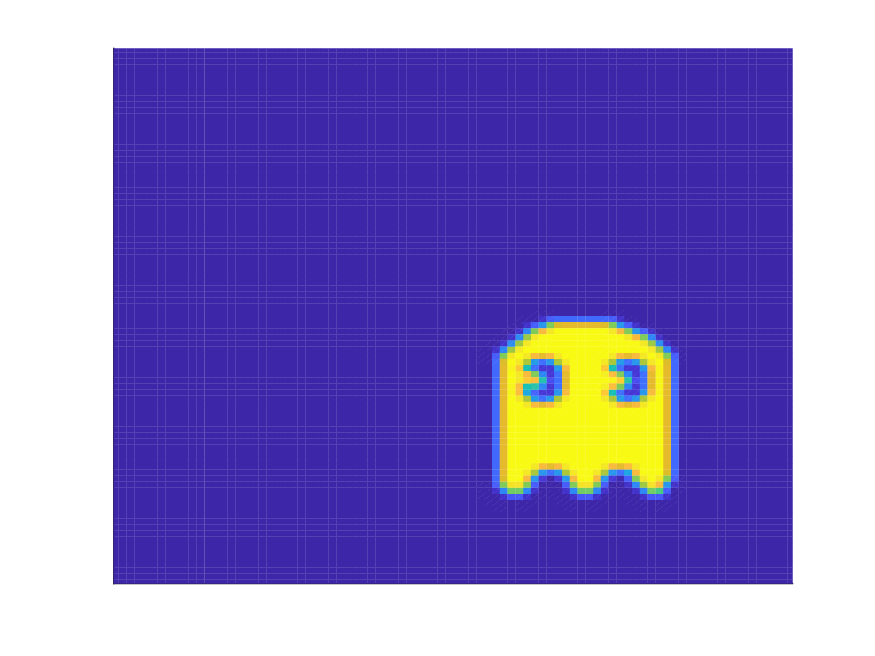}
\caption{Computation of the Wasserstein geodesic between Pac-Man and a ghost, with  $N_t = 40$ temporal points ($\Delta t = \frac{1}{40}$) and $N_x = N_y=120$ spatial grids ($\Delta x = \Delta y = 0.0458$). From left to right, up to down, the plots correspond to $t=0$, $t=0.15$, $t=0.275$, $t = 0.4$, $t = 0.525$, $t = 0.65$, $t = 0.775$, $t = 0.9$, and $t=1$. Here $\lambda=40$, $\sigma\lambda=1.2/\lambda_{max}(AA^t)$ then $\sigma=0.0036$, $\delta = 10^{-5}$, and $\delta_3 = 10^{-5}$.}\label{fig:pacman}
\end{figure}

\subsection{Wasserstein gradient flows: one dimension}
In this and the next section, we consider several examples of Wasserstein gradient flows, including some which have appeared in previous numerical studies \cite{CCH15, CRW16, SCS18,BCH}, to demonstrate the performance of our method for simulating solutions of nonlinear partial differential equations.

\subsubsection{Porous medium equation} \label{sec:PME}
The porous medium equation
\begin{equation} \label{PMEeqn}
\partial_t \rho = \Delta \rho^{m}\,, \quad m>1, 
\end{equation}
is the Wasserstein gradient flow of the energy (\ref{eqn:energy}), with  ${U}(\rho) = \frac{1}{m-1} \rho^m$ and $V = W = 0\,$. A well-known family of exact solutions is given by Barenblatt profiles (c.f. \cite{VazquezPME}), which are densities of the form
\begin{equation} \label{eqn:Barenblatt}
\rho(x,t) = (t+t_0)^{-\frac{1}{m+1}} \left( C - \alpha \frac{m-1}{2m(m+1)}  x^2 (t+t_0)^{-\frac{2}{m+1}} \right)_{+}^{\frac{1}{m-1}} , \qquad \text{ for } C, t_0 > 0 .
\end{equation}

\begin{figure}[h!]
\centering{\textbf{Evolution of solution to porous medium equation}} \\
\centering{\includegraphics[width=0.4\textwidth,trim={.8cm .15cm 1.5cm .8cm},clip]{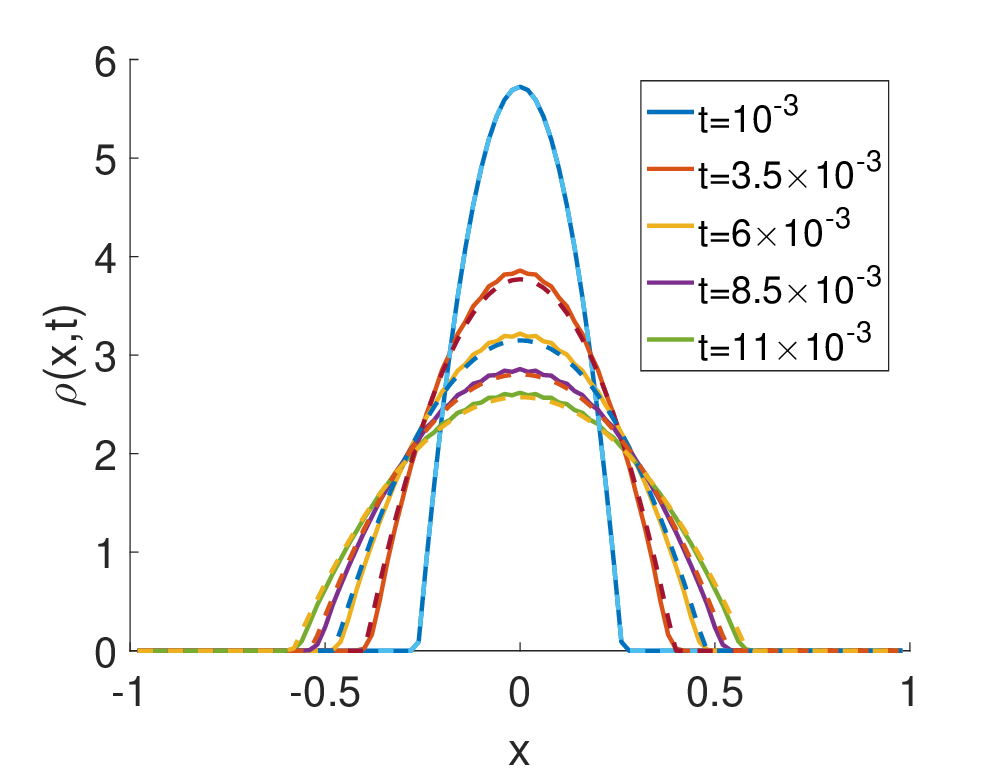}}
\caption{Evolution of the solution $\rho(x,t)$ to the one dimensional porous medium equation, with $m=2$ on the domain $\Omega = [-1,1]$. We choose $\tau = 0.5\times10^{-3}$, $ \Delta x = 0.02$,  $\Dt = 0.1$, $\lambda=0.2$, $\sigma\lambda=1.1/\lambda_{max}(AA^t)$ then $\sigma=0.1954$, $\delta = 10^{-5}$, and $\delta_3 = 10^{-5}$.}\label{fig:PM1D}
\end{figure}
We now apply algorithm \ref{alg:nonlinearJKO} to simulate solutions of the $m=2$  porous medium equation with  Barenblatt initial data, $t_0 = 10^{-3}$  and $C = (3/16)^{1/3}$. Here the Euler discretization (\ref{forwardEuler}) is used. In Figure~\ref{fig:PM1D}, we plot the evolution of the numerical solution over time, and we observe good agreement with the exact solution of the form (\ref{eqn:Barenblatt}), which is displayed in dashed curve.

In Figure~\ref{fig:convg}, we analyze how the rate of convergence depends on the   inner time step $\Delta t$, the spacial discretization $\Delta x$, and outer time step of the JKO scheme $\tau$. We compute the error between the exact solution and the numerical solution in the $\ell^1$ norm, i.e.
\[ \|\rho- \rho^* \|_{\ell^1} = \frac{1}{N_x(n+1)}\sum_{k=0}^{N_t}\sum_{j = 1}^{N_x-1} |\rho_{j}(k\tau) - \rho^*_{j}(k\tau)| . \]
In the plot on the left of Figure~\ref{fig:convg}, we consider two fixed values of $\tau$ and examine how the error depends on $N_t$ and $N_x  = 10 N_t$.  In both cases, the error quickly saturates, indicating that the outer time step $\tau$ dominates the error. In the plot on the right, we fix $N_t = 20$ and $N_x = 200$ and consider how the error depends on $\tau$. We observe slightly less than first order convergence in $\tau$ for the classical JKO scheme ($\energy^h = \F^h$) and higher order convergence for the Crank-Nicolson inspired scheme ($\energy^h = \H^h$). We believe these slower rates of convergence are due to the lower regularity of solutions to the porous medium equation with compactly supported initial data, which are merely H\"older continuous.

\begin{figure}[h!]
\centering
\textbf{Rate of convergence to porous medium equation} \\
\includegraphics[width=0.49\textwidth, trim={.1cm .05cm 1.2cm .6cm},clip]{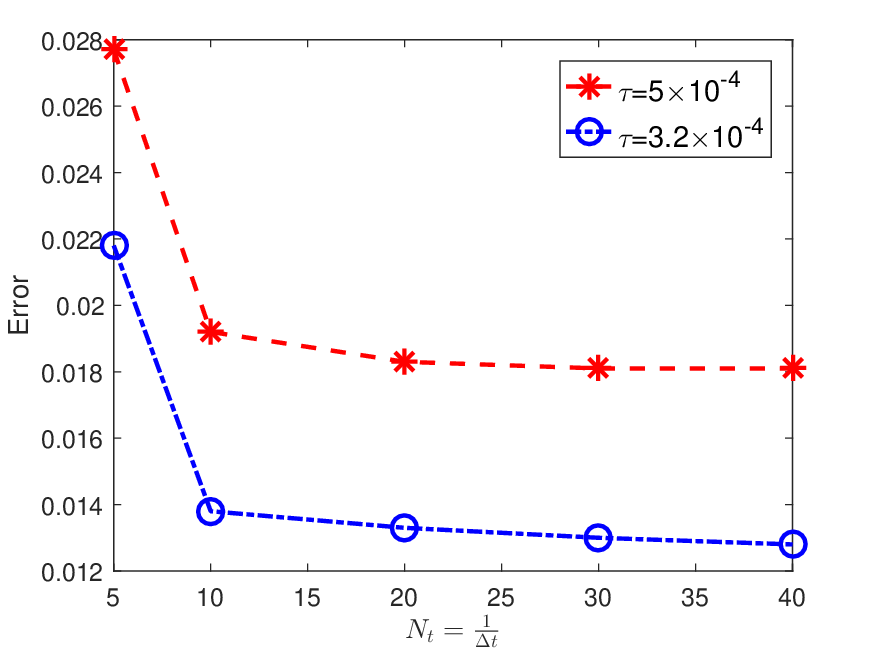}
\includegraphics[width=0.49\textwidth, trim={.1cm .05cm 1.2cm .6cm},clip]{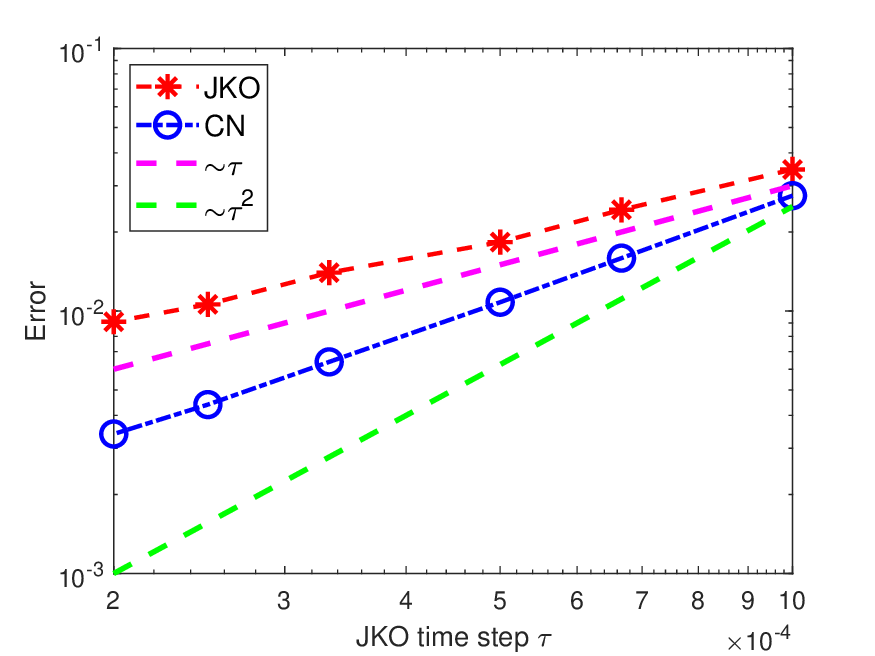}
\caption{Analysis of rate of convergence for  a solution of porous medium equation, as in Figure \ref{fig:PM1D}. Left: Convergence  to exact solution for  $N_x/N_t=10$ for choices of $\tau$. Right: Convergence to exact solution for $N_t=20$ and $N_x=200$ and various choices of $\tau$, contrasting the traditional first-order JKO scheme with the new Crank-Nicolson inspired scheme. }
\label{fig:convg}
\end{figure}

In  Figure~\ref{fig:convgsmooth}, we consider the case of smooth, strictly positive initial data, given by a Gaussian with mean $\mu =0$ and variance $\theta = 0.2$ (\ref{gaussian}),  in which case solutions of the PDE remain smooth over time.   On the left, we show the evolution of the solutions over time, and on the right, we illustrate that the classical JKO scheme indeed attains first order accuracy, though the Crank-Nicolson inspired scheme is still less than second order accurate.

\begin{figure}[h!]
	\centering
	\textbf{Porous medium equation with smooth positive initial density} \\
	\includegraphics[width=0.49\textwidth, trim={.1cm .05cm 1.2cm .6cm},clip]{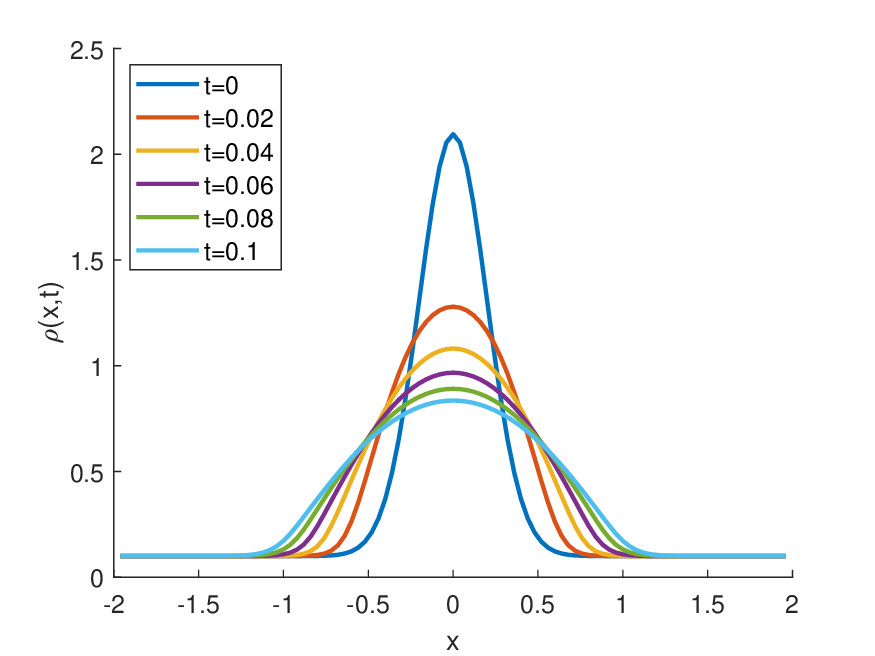}
	\includegraphics[width=0.49\textwidth, trim={.1cm .05cm 1cm .6cm},clip]{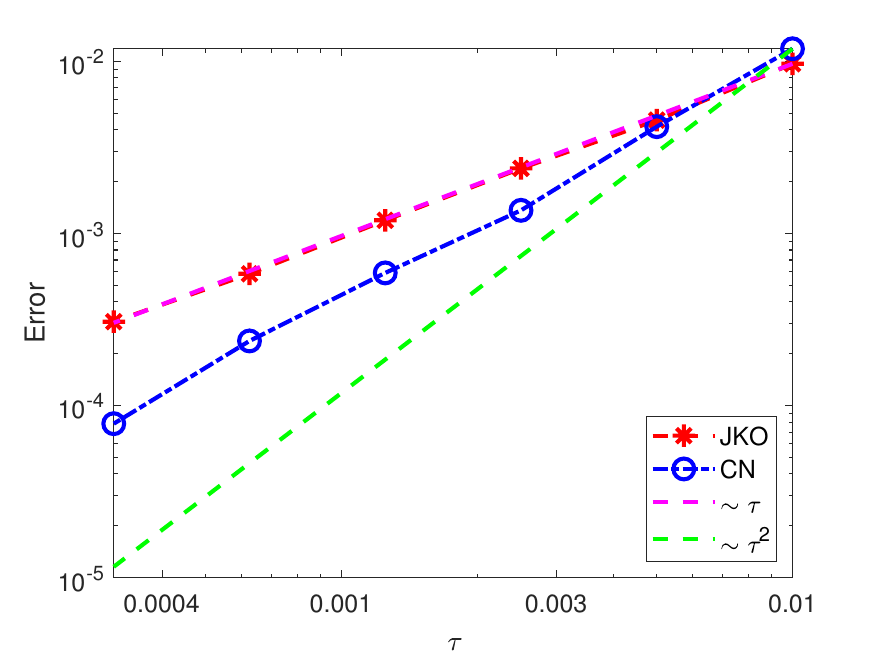}
	\caption{Evolution and the rate of convergence for  a solution of porous medium equation with smooth positive initial density. We choose $N_t=10$, $N_x=100$, $\sigma = 10$, $\lambda = 0.0148$. Left: Evolution of the solution $\rho(x,t)$ to the one dimensional porous medium equation, with $m=2$ on the domain $\Omega = [-2,2]$ for $\tau = 0.005$. Right: The rate of convergence for  various choices of $\tau$, contrasting the traditional first-order JKO scheme with the new Crank-Nicolson inspired scheme. For each choice of $\tau$ in our computation of the higher-order method, we choose our stopping criteria $\epsilon = 10^{-4} * 2^{-0.01/\tau}$.}
	\label{fig:convgsmooth}
\end{figure}

\subsubsection{Nonlinear Fokker-Planck equation} \label{sec:FP}
We now consider a nonlinear variant of the Fokker-Planck equation,
\begin{equation*} 
\partial_t \rho = \grad\cdot (\rho \grad V) + \Delta \rho^{m}\,, \quad V: \Rd \to \R, \quad m>1, 
\end{equation*}
inspired by the porous medium equation described in the previous section (\ref{PMEeqn}). When $V$ is a confining drift potential, all solutions approach the unique steady state
\[
\rho_\infty(x) =  \left( C -  \frac{m-1}{m}{V(x)} \right)_{+}^{\frac{1}{m-1}} \,,
\]
where $C>0$ depends on the mass of the initial data, so that $\int \rho_\infty \rd x = \int \rho_0 \rd x $, see \cite{CaTo00,CJMTU}.

\begin{figure}[h!]
\centering
\textbf{Evolution of solution to nonlinear Fokker-Planck equation}
\includegraphics[width=0.49\textwidth,trim={.6cm .05cm 1.3cm .6cm},clip]{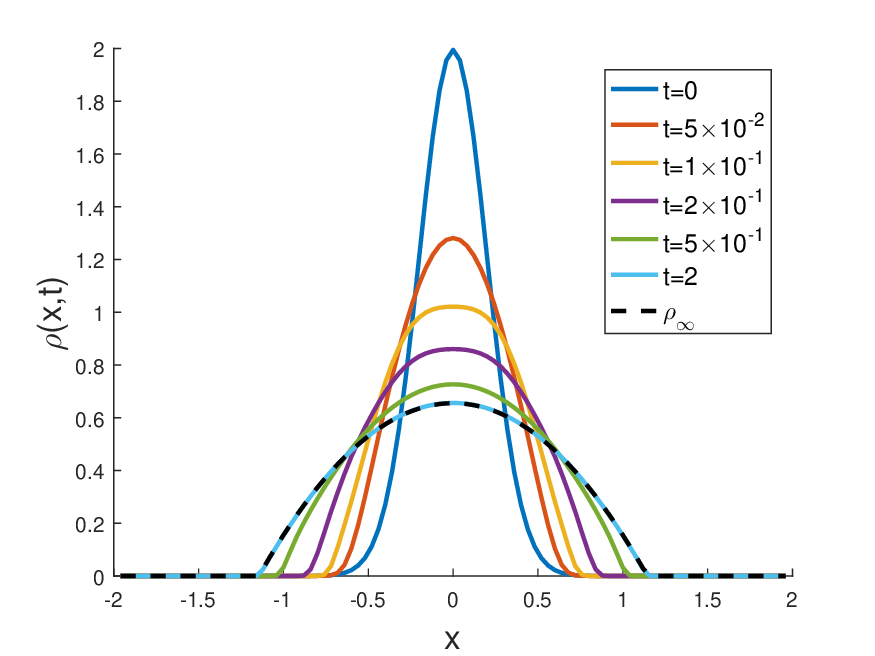}
\includegraphics[width=0.49\textwidth,trim={.6cm .05cm 1.3cm .6cm},clip]{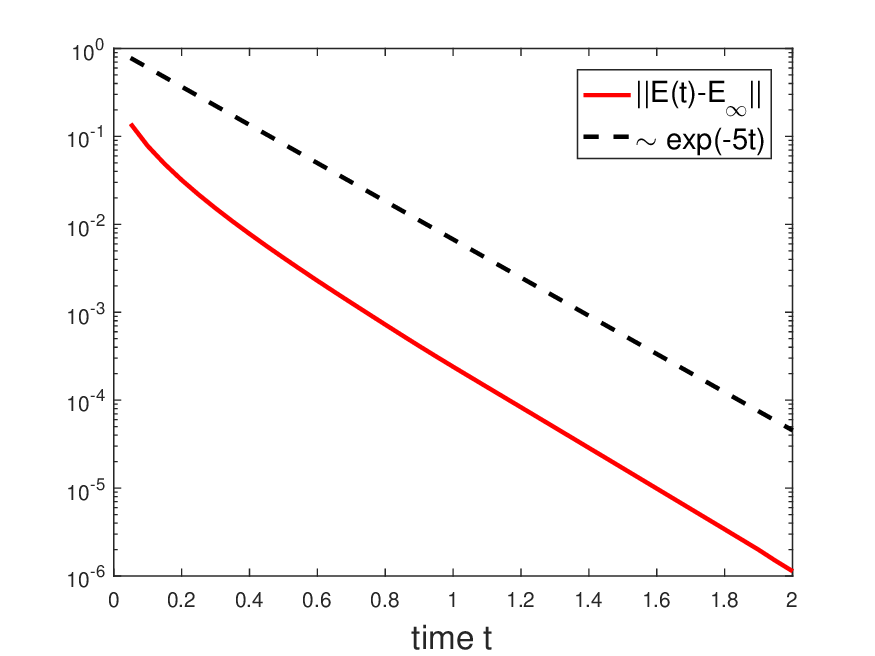}
\caption{Evolution of the solution $\rho(x,t)$ to the one dimensional nonlinear Fokker-Planck equation, with $m=2$ and $V(x) = x^2$. We choose $\tau = 0.05$, $ \Delta x = 0.04$,  $\Dt = 0.1$, $\lambda=0.1641$, $\sigma=1$, $\delta = 10^{-5}$, and $\delta_3 = 10^{-5}$. Left: evolution of density $\rho(x,t)$ towards equilibrium $\rho_\infty(x)$. Right: Rate of decay of corresponding energy with respect to time.}  \label{fig:FK1D}
\end{figure}
In Figure \ref{fig:FK1D}, we simulate the evolution of solutions to the nonlinear Fokker-Planck equation with $V(x) =  x^2 $, $m=2$, and initial data given by a Gaussian with  mean $\mu =0$ and variance $\theta = 0.2$ (\ref{gaussian}). On the left, we plot the evolution of the density $\rho(x,t)$ towards the steady state $\rho_\infty(x)$. On the right, we compute the rate of decay of the corresponding energy (\ref{eqn:energy}) as a function of time, observing exponential decay as the solution approaches equilibrium. In this way, our method recovers analytic results on convergence to equilibrium of Carrillo, DiFrancesco, and Toscani \cite{CaTo00, CDT07}.

\begin{figure}[h!]
\centering{\textbf{Rate of convergence to nonlinear Fokker-Planck equation}} \\
\centering{\includegraphics[width=0.49\textwidth,trim={.4cm .1cm 1.0cm .6cm},clip]{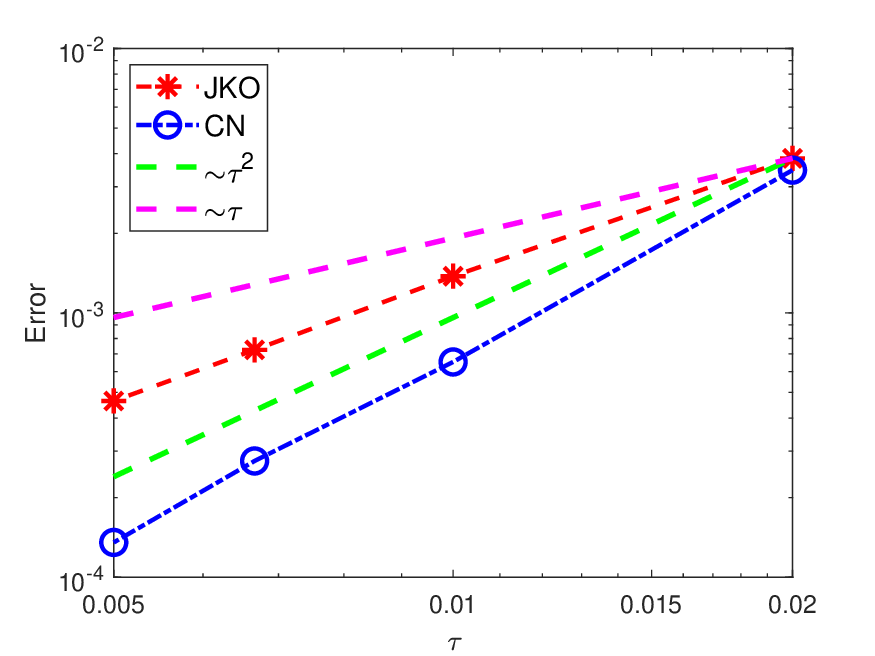}}
\caption{Analysis of rate of convergence for  a solution of the nonlinear Fokker-Planck equation, as in Figure \ref{fig:FK1D}. We choose $\Delta t = 0.1$, $\Delta x = 0.04$ and consider the error \eqref{eqn:error00} for various choices of  $\tau$, contrasting the traditional first-order JKO scheme with the new Crank-Nicolson inspired scheme.} \label{fig:FK1D-conv}
\end{figure}

In Figure~\ref{fig:FK1D-conv}, we analyze how the rate of convergence depends on the outer time step $\tau$ of the   scheme, for sufficiently small inner time step $\Delta t = 0.1$ and spacial discretization $\Delta x = 0.04$.
We compute the error
\begin{equation} \label{eqn:error00}
e_{\tau} = \| \rho_{\tau} (x,t) - \rho_{\tau/2} (x,t)   \|_{\ell^1}
\end{equation}
We observe slightly faster than first order convergence for the traditional JKO scheme ($\energy^h = \F^h$), and higher order convergence for the new Crank-Nicolson inspired scheme  ($\energy^h = \H^h$). We believe this improvement in the rate of convergence as compared to our previous example for the porous medium equation, Figure \ref{fig:convg}, is due to the rapid convergence to the steady state $\rho_\infty$.

\subsubsection{Aggregation equation} \label{aggeqn1d}
In this section, we consider a nonlocal partial differential equation of Wasserstein gradient flow-type, known as the aggregation equation
\begin{align} \label{aggeqn1}
\partial_t \rho = \grad\cdot (\rho \grad W*\rho)  \,, \quad W: \Rd \to \R\,.
\end{align}
In recent years, there has been significant interest in interaction kernels $W$ that are repulsive at short length scales and attractive at longer distances, such as the kernel with logarithmic repulsion and quadratic attraction
\begin{equation} \label{Weqn1}
W(x) = \frac{|x|^2}{2} - \text{ln}(|x|)\,.
\end{equation}
For this particular choice of $W$, there exists a unique equilibrium profile \cite{CFP12}, given by
\[\rho_\infty (x) =   \frac{1}{\pi} \sqrt{(2-x^2)_+}    .\]

\begin{figure}[h!]
\centering
\textbf{Evolution of solution to aggregation equation } \\
\includegraphics[width=0.49\textwidth,trim={.4cm .1cm 1.3cm .4cm},clip]{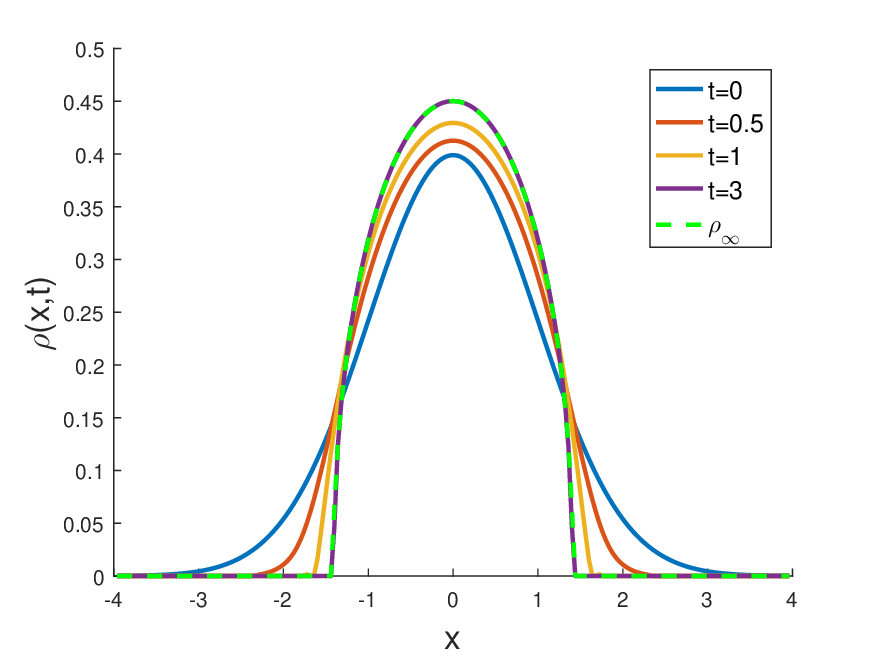}
 \includegraphics[width=0.49\textwidth,trim={.4cm .1cm 1.3cm .4cm},clip]{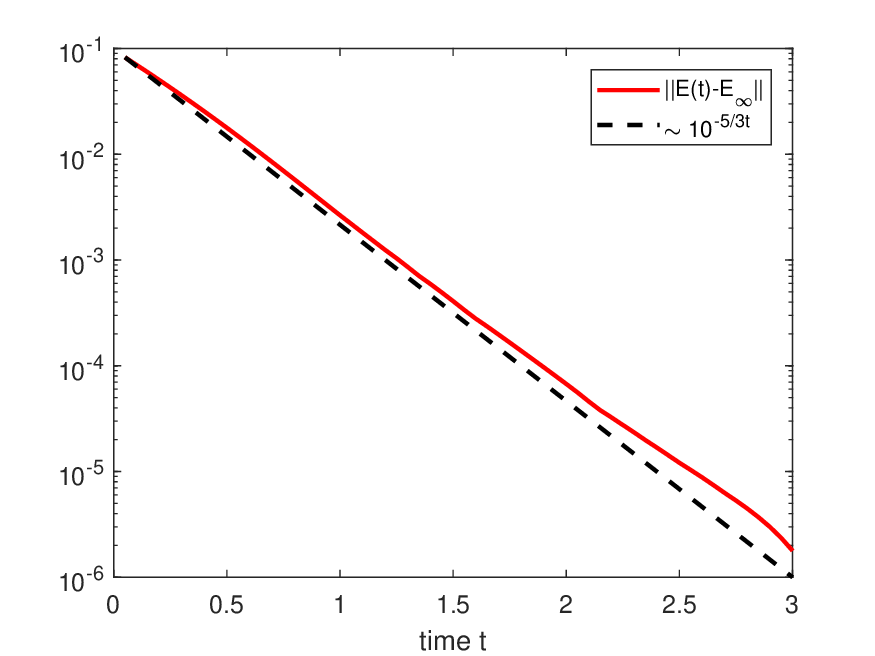}
\caption{Evolution of the solution $\rho(x,t)$ to the one dimensional aggregation equation, with $W(x) = x^2/2 - \ln(|x|)$, $x\in [-4,4]$. We choose $\tau = 0.05$, $ \Delta x = 0.04$,  $\Dt = 0.05$, $\lambda=0.01$, $\sigma\lambda=0.99/\lambda_{max}(AA^t)$ then $\sigma=18.8$, $\delta = 10^{-6}$, and $\delta_3 = 10^{-6}$. Left: evolution of density $\rho(x,t)$ towards equilibrium $\rho_\infty(x)$. Right: Rate of decay of corresponding energy with respect to time. }  \label{fig:Agg1D}
\end{figure}
In Figure~\ref{fig:Agg1D}, we simulate the solution to the aggregation equation with Gaussian initial data (\ref{gaussian}) with mean $\mu =0$ and variance $\theta =1$, analyzing  convergence to equilibrium. On the left, we plot the evolution of the density $\rho(x,t)$ at varying times, observing convergence to the equilibrium profile $\rho_\infty(x)$. On the right, we compute the rate of the decay of the energy as a function of time, observing exponential decay as obtained by Carrillo, Ferreira, and Precioso \cite{CFP12} with a slightly slower numerical rate.

As the interaction potential $W$ defined in equation (\ref{Weqn1}) is not continuous, we make the following modifications to our discretization of the JKO scheme. To avoid evaluation of $W(x)$ at $x=0$, we set $W(0)$ to equal the average value of $W$ on the cell of width $2h$ centered at 0, i.e., $W(0) = \frac{1}{2h} \int_{-h}^{h} W(x) \rd x$, where we apply Gauss-Legendre quadrature rule with four grid points to evaluate the integral. In addition to modifying the interaction kernel in this way, we also introduce   an artificial diffusion term of the form $\epsilon \partial_x( \rho \partial_x \rho)$ with $\epsilon=1.6\times(\Delta x)^2$ to the right hand side of (\ref{aggeqn1}), to avoid the possible overshoot at the boundary. (See also \cite{CCH15} for a similar treatment.)

\subsection{Wasserstein gradient flows: two dimensions}
In the following, we consider a few gradient flows in two dimensions. Here the constraint relaxation parameters are always chosen as $\delta = \delta _3 = 10^{-6}$.

\subsubsection{Aggregation equation}
We now continue our study of the aggregation equation (\ref{aggeqn1}) with repulsive-attractive interaction potentials in two dimensions, with interaction kernels of the form
\begin{equation} \label{attrep1}
W(x) = \frac{|x|^a}{a} - \frac{|x|^b}{b}, \quad a > b \geq 0\, ,
\end{equation}
using the convention that $\frac{|x|^0}{0} = \text{ln}(|x|)$. It is well-known that the repulsion near the origin of the potential determines the dimension of the support of the steady state measure, see \cite{BCLR13-2,CDM}. In the following simulations, we take the initial data to be a gaussian (\ref{gaussian}) with mean $\mu =0$ and variance  $\theta=0.25$.

\begin{figure}[h!]
\centering
\textbf{Aggregation equation, smooth interaction potential} \\ \hspace{-.4cm}
\includegraphics[width=0.33\textwidth, trim={.9cm .5cm 1.3cm 1.2cm},clip]{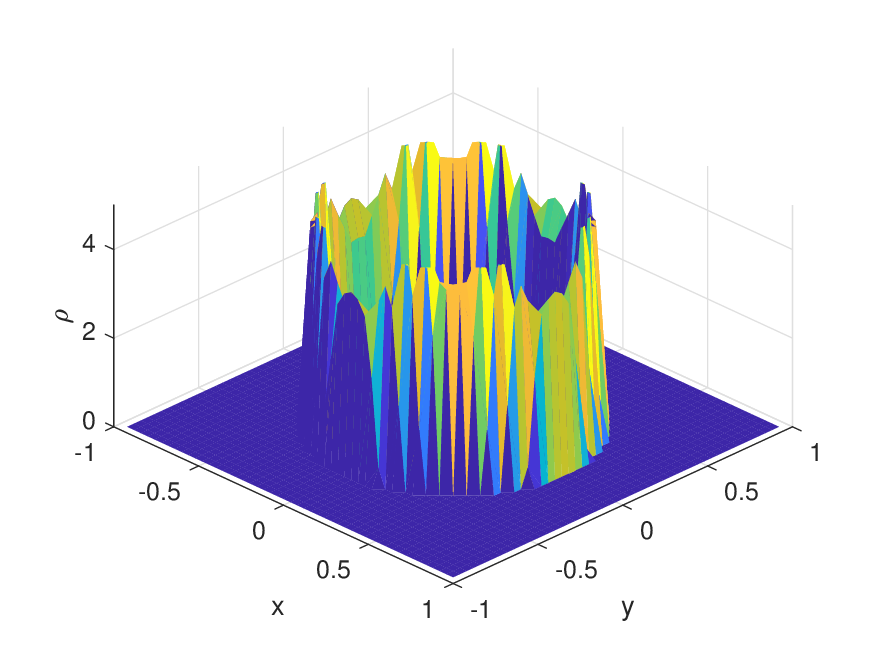}
\includegraphics[width=0.33\textwidth,trim={.5cm .05cm 1.3cm .4cm},clip]{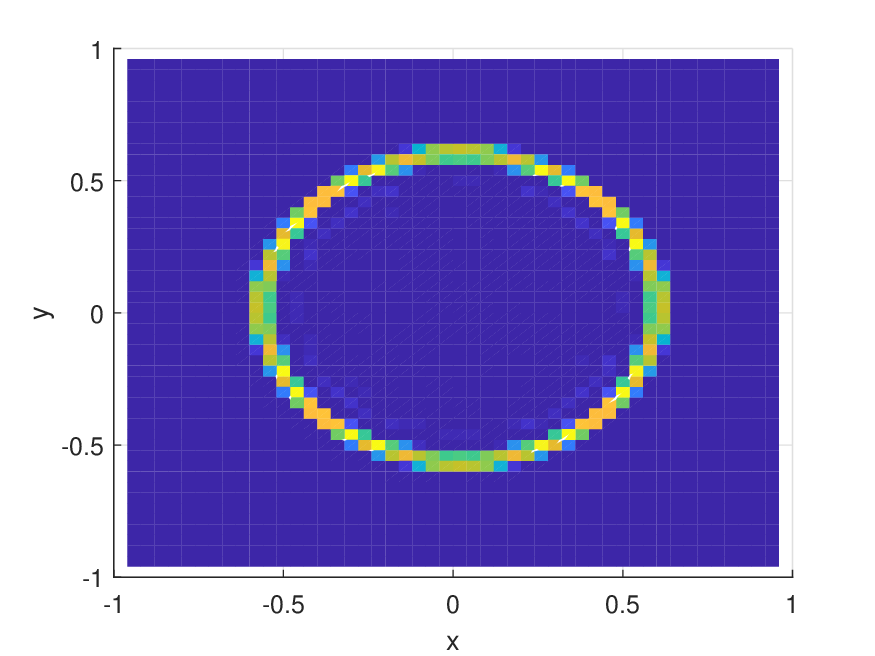}
\includegraphics[width=0.33\textwidth,trim={.1cm .1cm 1.2cm .4cm},clip]{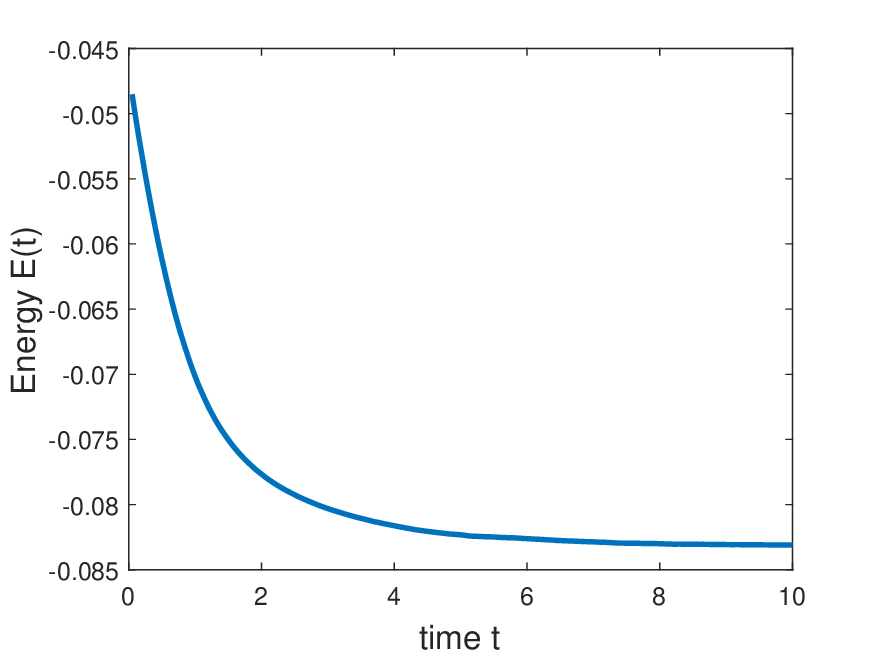}
\caption{We compute the steady state of a solution to the two dimensional aggregation equation with interaction potential $W(x) = |x|^4/4 - |x|^2/2$, which is a Dirac ring of radius 0.5, centered at the origin. The computational domain is [-1,1]$\times$[-1,1]. We choose $\tau = 0.05$, $ \Delta x = \Delta y = 0.04$,  $\Dt = 0.1$, $\lambda = 20$, $\sigma = 0.0052$, and $\epsilon_1=\epsilon_2=10^{-6}$. The steady state shown is the solution at time $t=10$. Left: side view of equilibrium. Center: bird's eye view of equilibrium. Right: rate of decay of energy as solution approaches equilibrium.} \label{AD}
\end{figure}
In Figure \ref{AD}, we simulate the evolution of solutions to the aggregation equation, with   $a=4$ and $b=2$ in the interaction potential $W$, defined in equation (\ref{attrep1}). We observe that the solution concentrates on a Dirac ring with radius $0.5$ centered at the origin, recovering analytical results on the existence of a stable Dirac ring equilibrium for these values of $a$ and $b$ \cite{BCLR13, BKSUV15}.

\begin{figure}[h!]
\centering
\textbf{Aggregation equation, singular interaction potential} \\ \hspace{-.2cm}
\includegraphics[width=0.33\textwidth, trim={.9cm .5cm 1.3cm 1.2cm},clip]{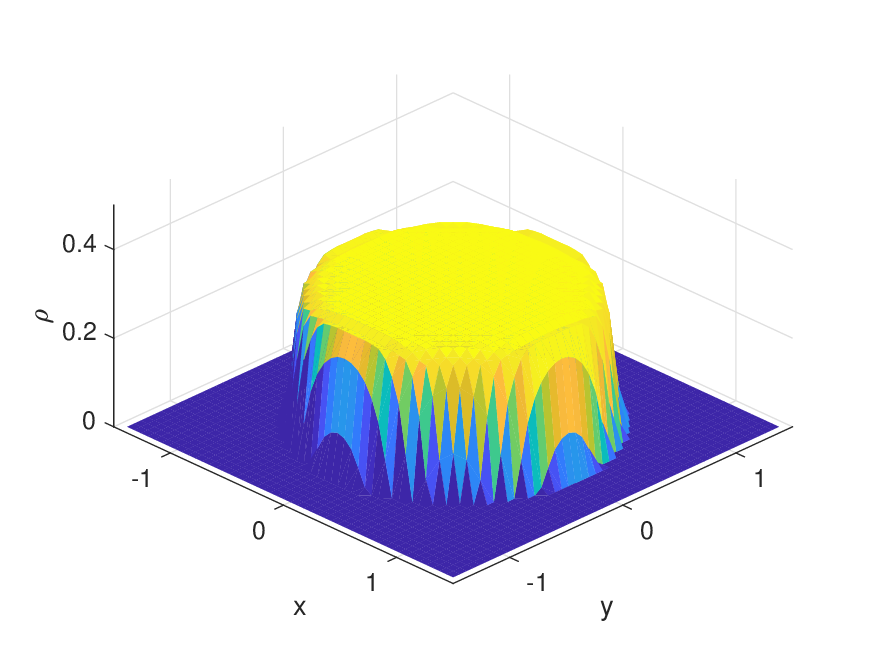}
\includegraphics[width=0.33\textwidth, trim={.5cm .05cm 1.3cm .4cm},clip]{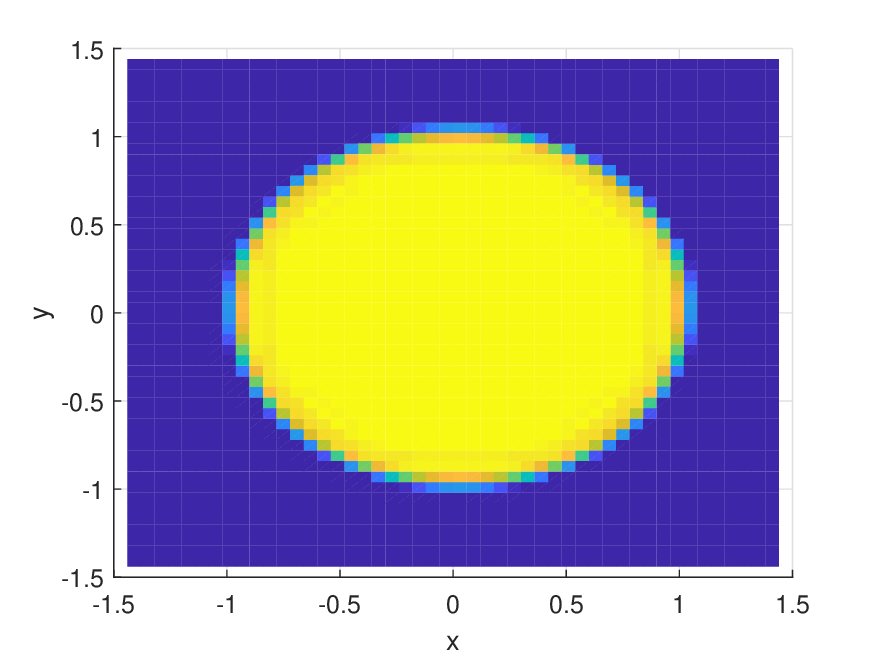}
\includegraphics[width=0.33\textwidth, trim={.1cm .1cm 1.2cm .4cm},clip]{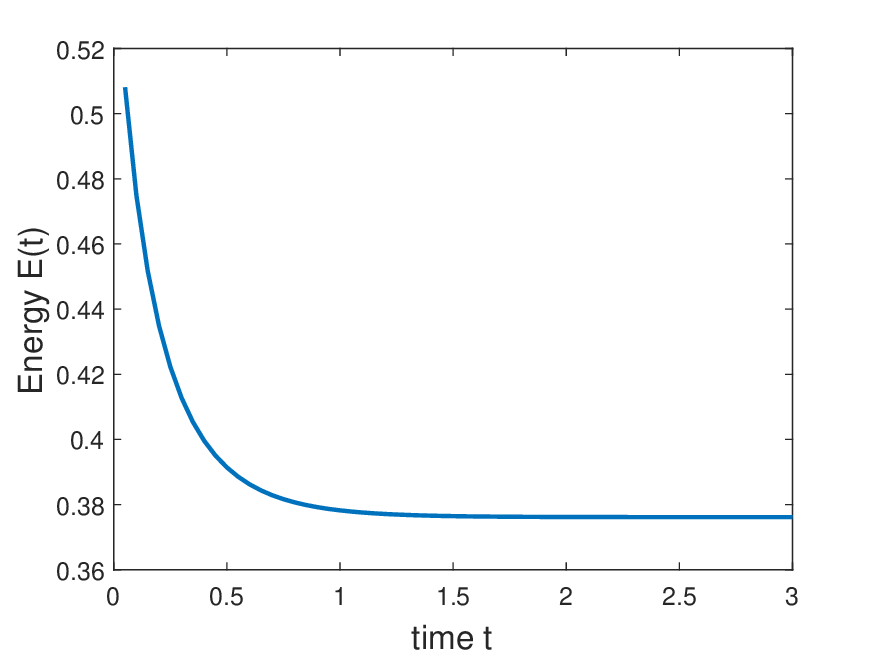}
\caption{
We compute steady state of a solution to the two dimensional aggregation equation with interaction potential $W(x) = |x|^2/2 - \ln(|x|)$, which is the characteristic function on a disk of radius 1, centered at the origin. The computational domain is [-1.5,1.5]$\times$[-1.5,1.5]. We choose $\tau = 0.05$, $ \Delta x = \Delta y = 0.06$,  $\Dt = 0.05$, $\lambda = 50$, and $\sigma = 0.0037$. The steady state is plotted at time t=3. Left: side view of equilibrium. Center: bird's eye view of equilibrium. Right: rate of decay of energy as solution approaches equilibrium.} \label{AD2}
\end{figure}
In Figure \ref{AD2}, we simulate the evolution of solutions to the aggregation equation, with $a = 2$ and $b = 0$. We observe that the solution converges to a characteristic function on the disk of radius 1, centered at the origin, recovering analytic results on solutions of the aggregation equation with Newtonian repulsion \cite{FHK11, BLL12, CDM16}. We follow the same strategy described in section \ref{aggeqn1d} with $\epsilon=1.6\times(\Delta  x^2+\Delta y^2)$ to overcome the singularity of the interaction potential at $x=0$ and potential overshooting.

\subsubsection{Aggregation drift equation}
Next, we compute solutions of aggregation-drift equations
\begin{align*} 
\partial_t \rho = \grad\cdot (\rho \grad W*\rho) + \grad \cdot (\rho \grad V) ,
\end{align*}
where  $W(x) = |x|^2/2 - \ln(|x|)$ and $V(x) = - \frac{\alpha}{\beta} \ln (|x|)$. As shown in several analytical and numerical results \cite{CK14, CHM14, CCH15}, the steady state is a characteristic function on a torus or ``milling profile'', with inner and outer radius given by
\[
R_i = \sqrt{\frac{\alpha}{\beta}},  \quad R_o = \sqrt{\frac{\alpha}{\beta} + 1}\,.
\]

\begin{figure}[h!]
\centering
\textbf{Equilibrium of aggregation-drift equation} \\ \hspace{-.2cm}
\includegraphics[width=0.33\textwidth, trim={.9cm .5cm 1.3cm 1.2cm},clip]{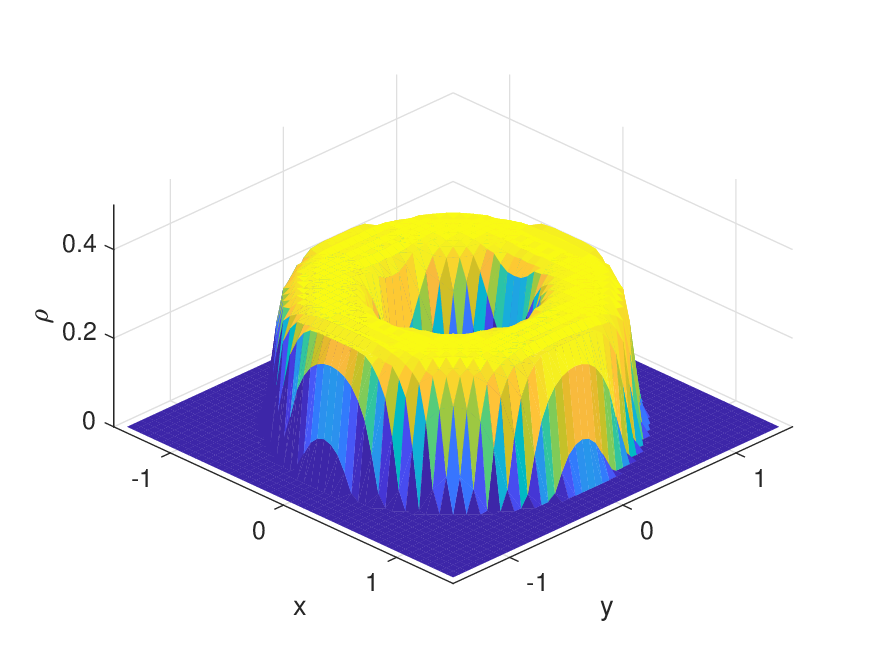}
\includegraphics[width=0.33\textwidth, trim={.5cm .05cm 1.3cm .4cm},clip]{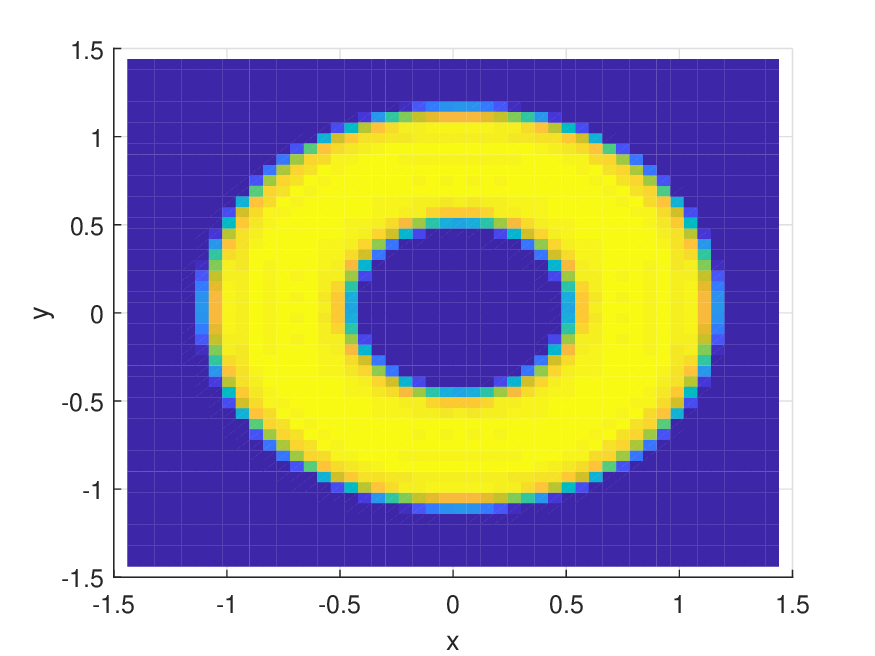}
\includegraphics[width=0.33\textwidth, trim={.1cm .1cm 1.2cm .4cm},clip]{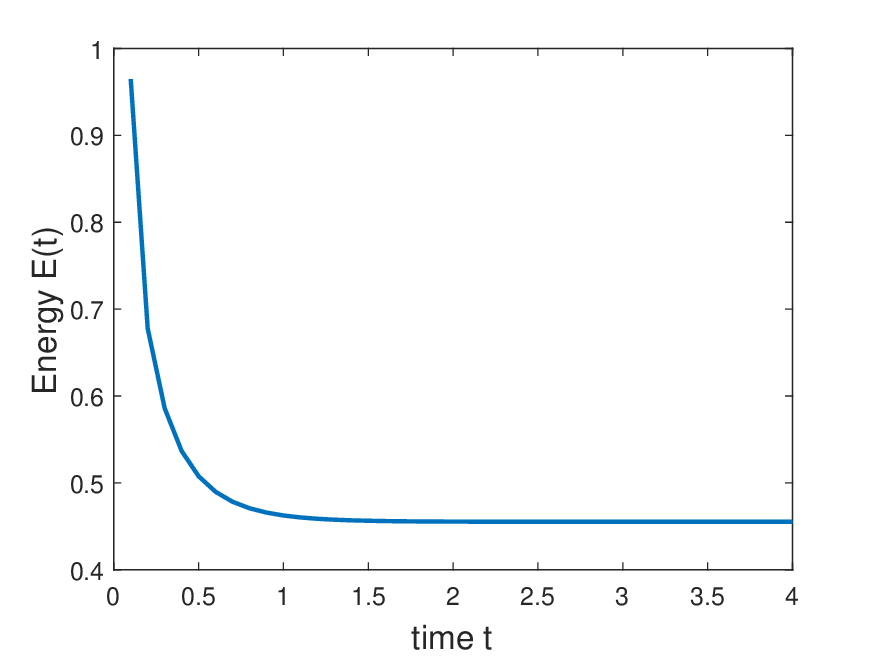}
\caption{
We compute steady state of a solution to the two dimensional aggregation-drift equation with interaction potential $W(x) = |x|^2/2 - \ln(|x|)$ and drift potential $V(x) = -(1/4) \ln(|x|)$, which is the characteristic function on a torus, centered at the origin. The computational domain is [-1.5,1.5]$\times$[-1.5,1.5]. We choose $\tau = 0.1$, $ \Delta x = \Delta y = 0.06$,  $\Dt = 0.05$, $\lambda = 40$, and $\sigma = 0.0046$. The steady state is the solution at time t=4. Left: side view of equilibrium. Center: bird's eye view of equilibrium. Right: rate of decay of energy as solution approaches equilibrium. } \label{AD3}
\end{figure}
In Figure \ref{AD3}, we simulate the long time behavior of a solution of the aggregation-drift equation with $\alpha =1$ and $\beta =4$ and Gaussian initial data (\ref{gaussian}), $\mu = 0$, $\theta = 0.25$, as well as the rate of the decay of the entropy as the solution converges to equilibrium.  In Figure (\ref{fig:AR4}), we plot the evolution of the density from a non-radially symmetric initial data, given by five Gaussians to the same equilibrium profile. We follow the same strategy described in section \ref{aggeqn1d} to overcome the singularity of the interaction potential at $x=0$ and potential overshooting ($\epsilon= 2\times(\Delta x^2+ \Delta y^2)$ in Figure \ref{AD3}  and $\epsilon=2.6\times(\Delta x^2+ \Delta y^2)$ in Figure \ref{fig:AR4}.)

\begin{figure}[h!]
\centering
\textbf{Evolution of aggregation-drift equation} \\ \hspace{-.3cm}
\begin{picture}(155,130)
\put(0,0){\includegraphics[width=0.345\textwidth,trim={.6cm 1.1cm 1.3cm 2.2cm},clip]{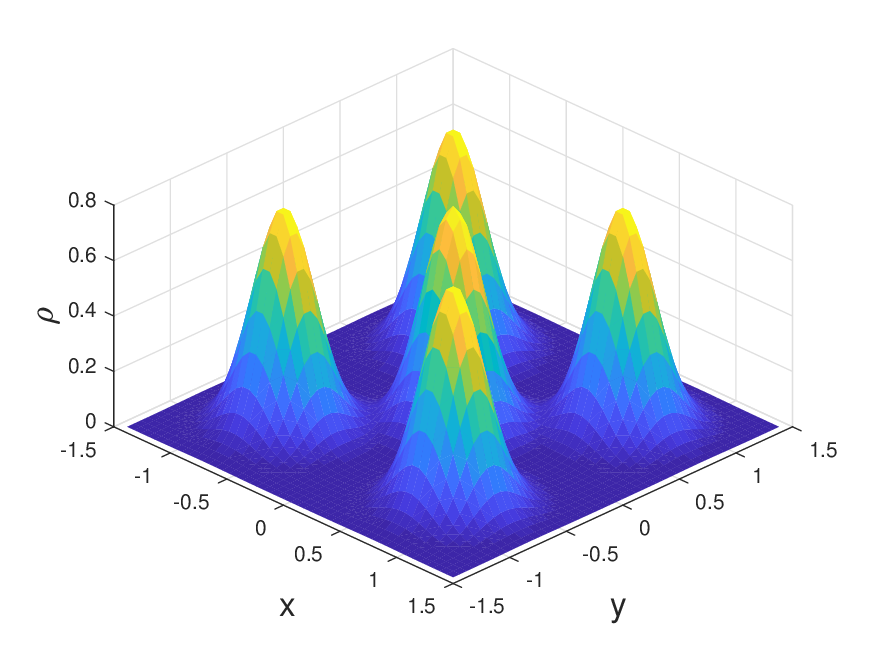}}
\put(72,105){t = 0.0}
\end{picture}
\begin{picture}(155,130)
\put(5,0){\includegraphics[width=0.33\textwidth,trim={1.9cm 1.1cm .7cm 2.2cm},clip]{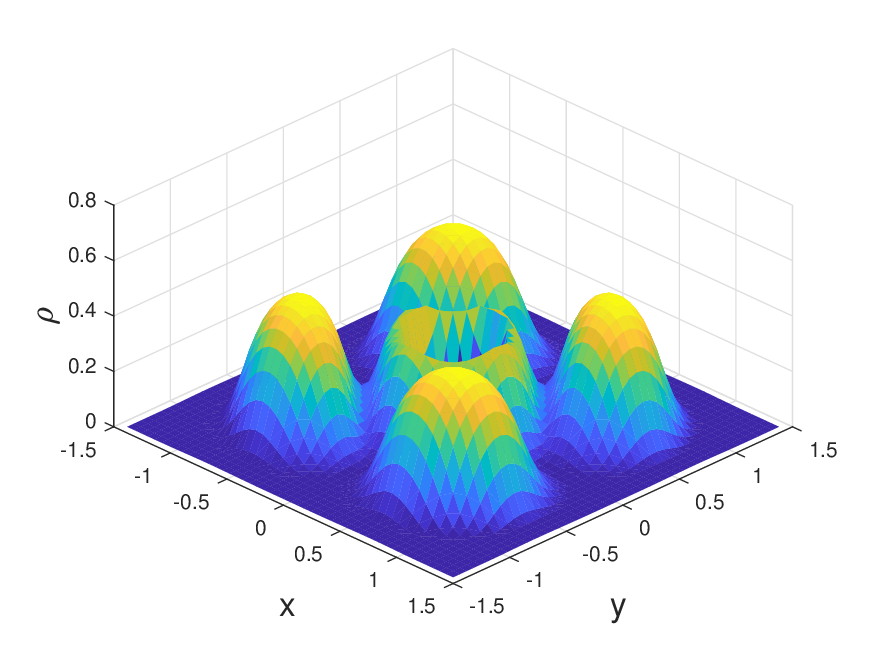}}
\put(63,105){t = 0.4}
\end{picture}
\begin{picture}(155,130)
\put(0,0){\includegraphics[width=0.33\textwidth,trim={1.9cm 1.1cm .7cm 2.2cm},clip]{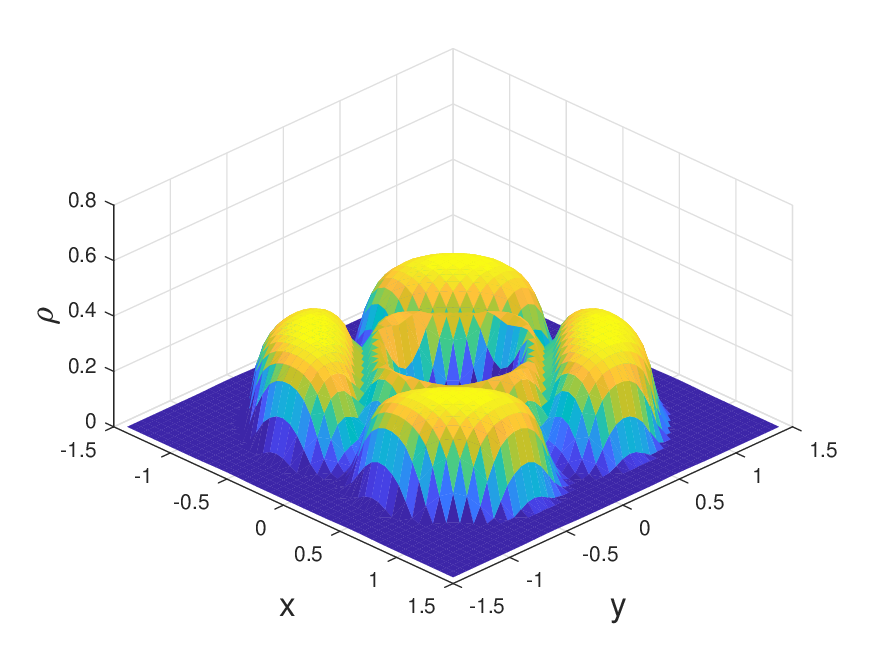}}
\put(58,105){t = 1.0}
\end{picture} \\  \hspace{-.3cm}
\begin{picture}(155,130)
\put(0,0){\includegraphics[width=0.345\textwidth,trim={.6cm 1.1cm 1.3cm 2.2cm},clip]{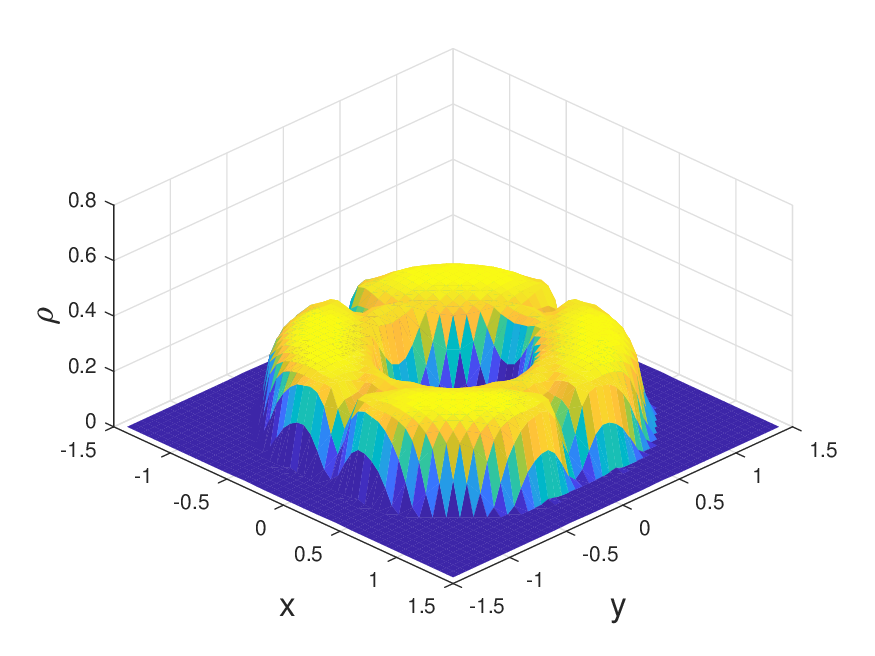}}
\put(72,105){t = 2.0}
\end{picture}
\begin{picture}(155,130)
\put(5,0){\includegraphics[width=0.33\textwidth,trim={1.9cm 1.1cm .7cm 2.2cm},clip]{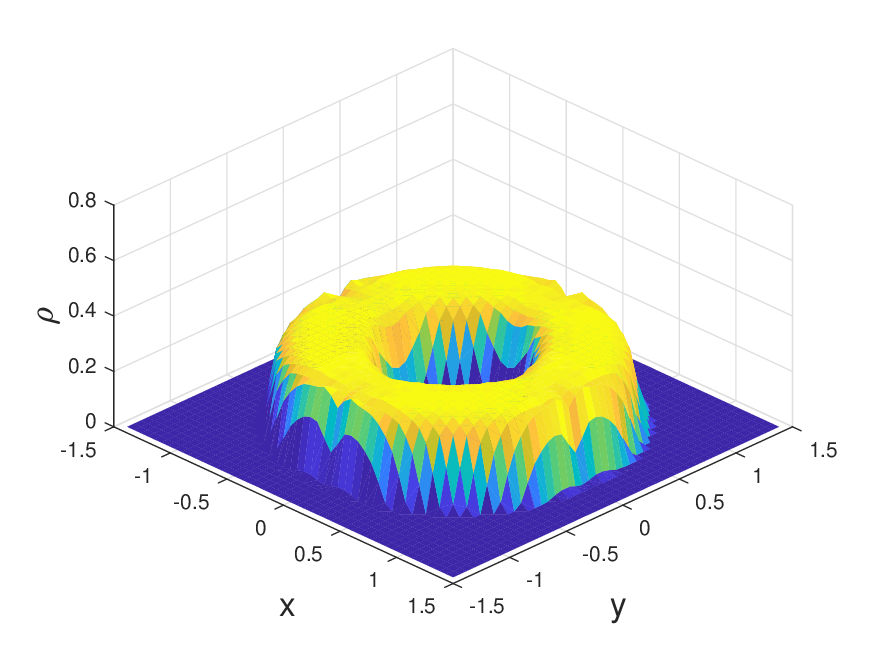}}
\put(63,105){t = 3.0}
\end{picture}
\begin{picture}(155,130)
\put(0,0){\includegraphics[width=0.33\textwidth,trim={1.9cm 1.1cm .7cm 2.2cm},clip]{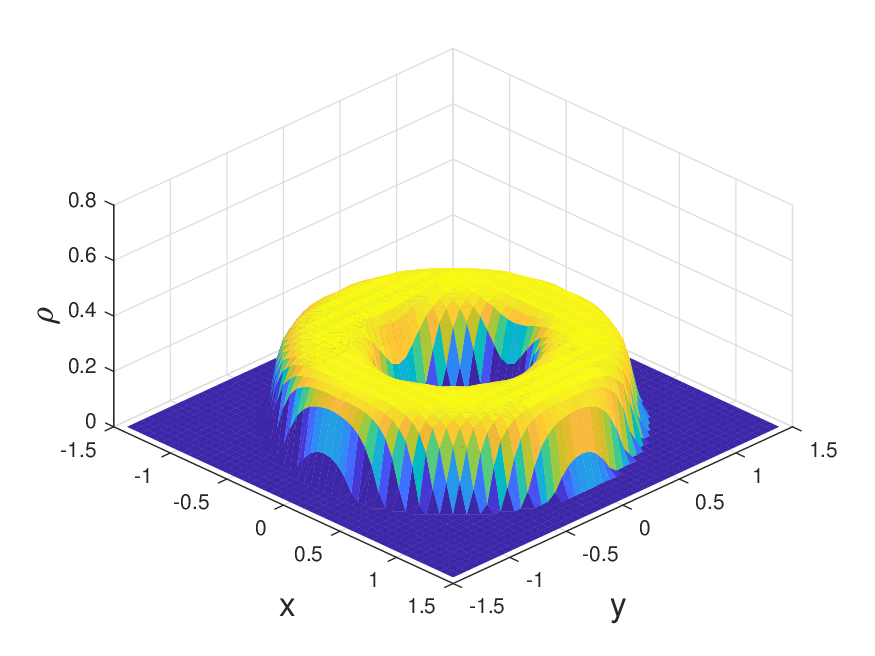}}
\put(58,105){t = 8.0}
\end{picture}
\caption{Evolution of the solution $\rho(x,y,t)$ to the two dimensional aggregation-drift equation, with $W(x) = x^2/2 - \ln(|x|)$ and $V(x) = -(1/4) \ln(|x|)$. The computational domain is [-1.5,1.5]$\times$[-1.5,1.5]. We choose $\tau = 0.2$, $ \Delta x = \Delta y = 0.06$,  $\Dt = 0.1$, $\lambda = 10$, and $\sigma=0.0244$. We observe convergence to the characteristic function on a torus centered at the origin.} \label{fig:AR4}
\end{figure}

\subsubsection{Aggregation diffusion equation}
We close by simulating several examples of aggregation diffusion equations
\begin{align} \label{aggdiffeqn}
\partial_t \rho = \grad\cdot (\rho \grad W*\rho) + \nu \Delta \rho^m , \quad W: \Rd \to \R, \quad m \geq 1 .
\end{align}
In recent years, there has been significant activity studying equations of this form, both analytically and numerically. When the interaction kernel $W$ is attractive, the competition between the nonlocal aggregation $ \grad\cdot (\rho \grad W*\rho)$ and nonlinear diffusion $\nu \Delta \rho^m$ causes solutions to blow up in certain regimes and exist globally in time in others, see for example \cite{BDP06,BCL09,CCH1,CCH2,CHMV} and the survey \cite{CCY}.  With fixed $m$, and in the presence of nonlocal interaction, the equation has a unique steady state which is radially decreasing up to a translation \cite{BL13, CHVY}.

\begin{figure}[h!]
\centering
\textbf{Evolution of aggregation-diffusion equation, smooth interaction kernel} \\ \hspace{-.2cm}
\begin{picture}(155,130)
\put(0,0){\includegraphics[width=0.345\textwidth,trim={.6cm 1.1cm 1.3cm 2.2cm},clip]{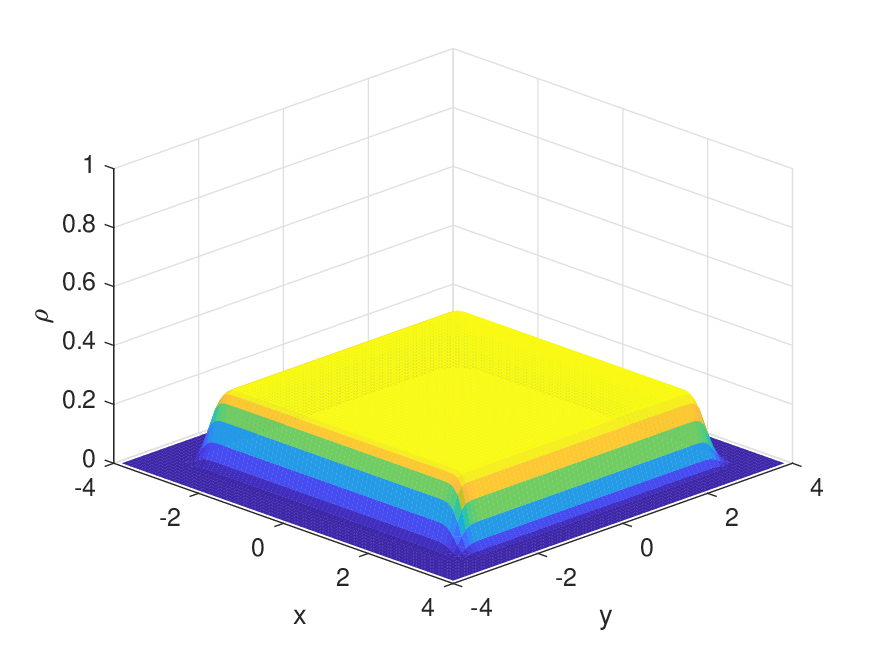}}
\put(72,105){t = 0}
\end{picture}
\begin{picture}(155,130)
\put(5,0){\includegraphics[width=0.33\textwidth,trim={1.9cm 1.1cm .7cm 2.2cm},clip]{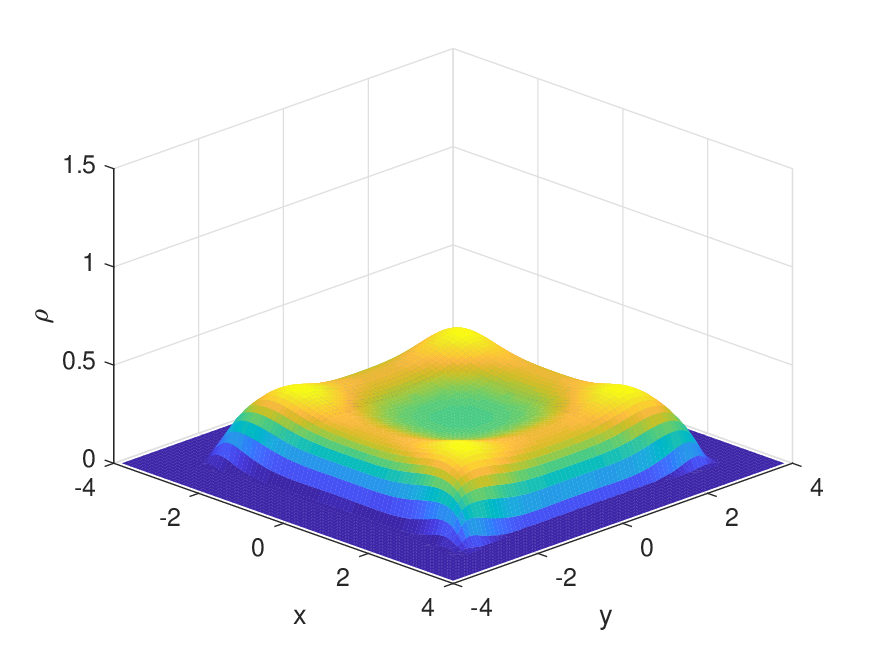}}
\put(63,105){t = 2}
\end{picture}
\begin{picture}(155,130)
\put(0,0){\includegraphics[width=0.33\textwidth,trim={1.9cm 1.1cm .7cm 2.2cm},clip]{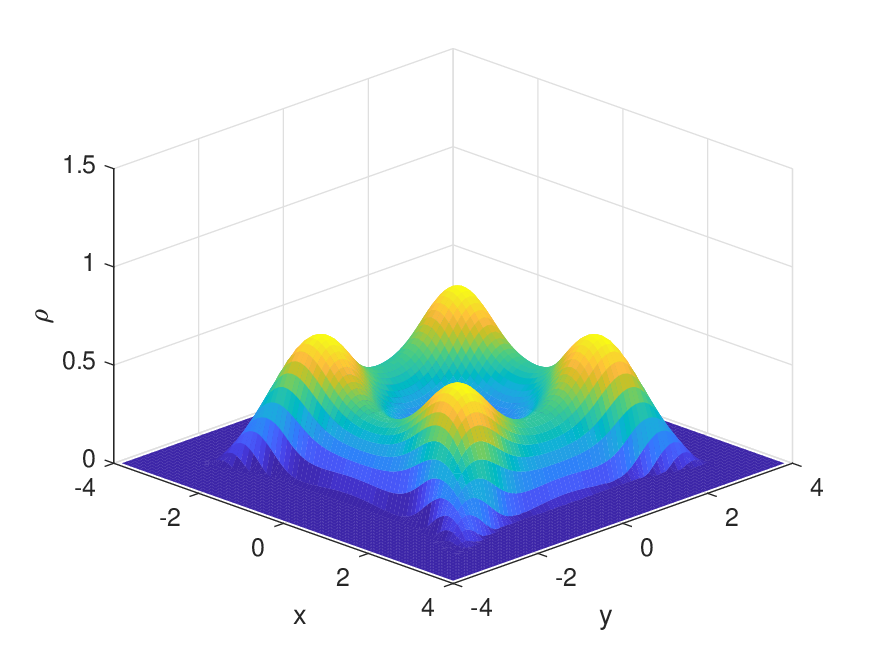}}
\put(58,105){t = 4}
\end{picture} \\  \hspace{-.3cm}
\begin{picture}(155,130)
\put(0,0){\includegraphics[width=0.345\textwidth,trim={.6cm 1.1cm 1.3cm 2.2cm},clip]{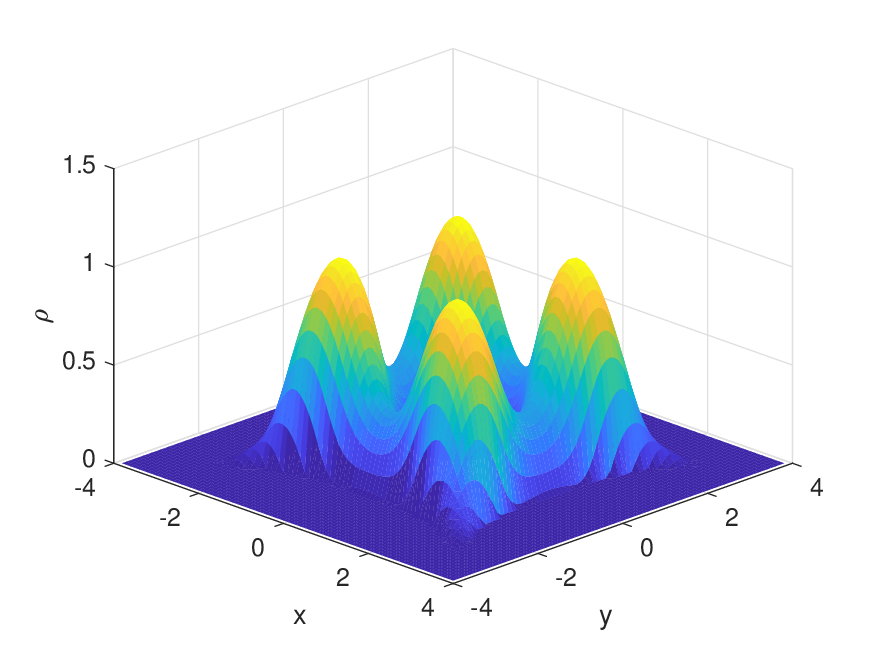}}
\put(72,105){t = 6}
\end{picture}
\begin{picture}(155,130)
\put(5,0){\includegraphics[width=0.33\textwidth,trim={1.9cm 1.1cm .7cm 2.2cm},clip]{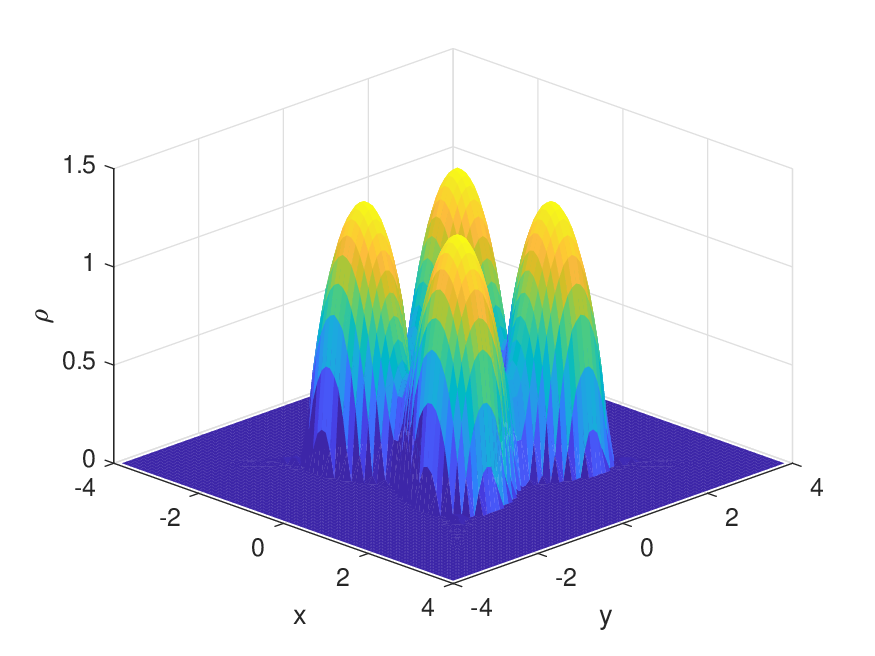}}
\put(63,105){t = 10}
\end{picture}
\begin{picture}(155,130)
\put(0,0){\includegraphics[width=0.33\textwidth,trim={1.9cm 1.1cm .7cm 2.2cm},clip]{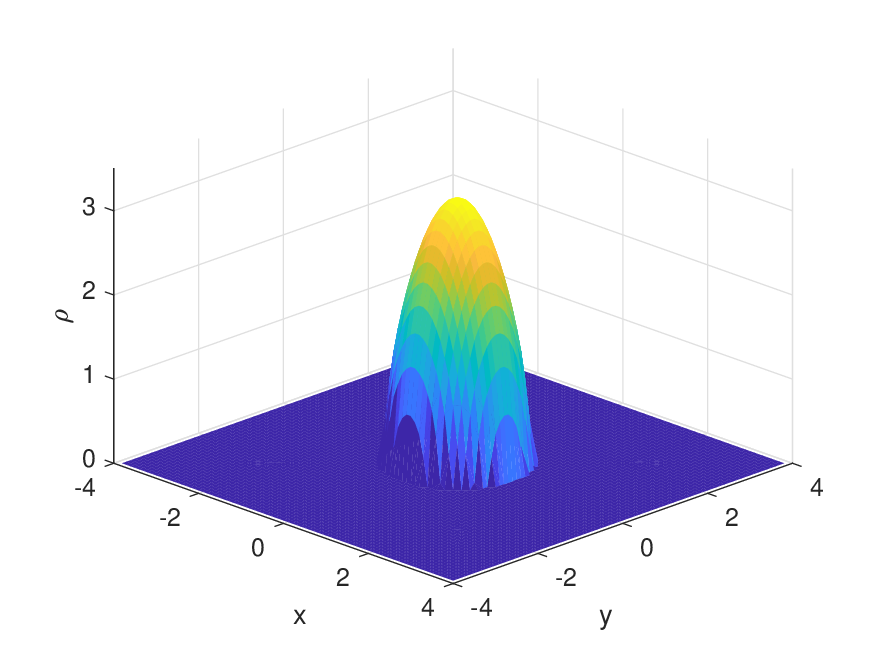}}
\put(58,105){t = 15}
\end{picture}
\caption{Evolution of solution $\rho(x,y,t)$ to the two dimensional aggregation-diffusion equation, with   $W(x) = -e^{-|x|^2}/\pi$, $\nu = 0.1,$ and $m =3$ on the domain $\Omega = [-4,4] \times [-4,4]$. We choose $\tau = 0.5$, $ \Delta x = \Delta y = 0.1$,  $\Dt = 0.1$, $\sigma=0.1144$, and $ \lambda=0.5$. The total iteration number for 40 JKO time steps is 197852. We observe convergence to the a single bump centered at the origin.} \label{fig:Agg2D}
\end{figure}
In Figure \ref{fig:Agg2D}, we simulate a solution of the aggregation diffusion equation with   $W(x) = - e^{-|x|^2}/\pi $, $\nu=0.1$, and $m =3$, and  initial data  given by a rescaled characteristic function on the square,
\[
\rho_0(x,y) = \frac{1}{4} \chi_{[-3,3]\times[-3,3]} (x,y)\,,
\]
Diffusion dominates both the short and long ranges, and the medium range aggregation leads to the formation of four bumps, which ultimately approach a single bump equilibrium. (See also \cite{CCH15}.)

\begin{figure}[h!]
\centering
\textbf{Keller-Segel equation:  blowup  vs. global existence for $m=1,2$} \\
\includegraphics[width=0.49\textwidth,trim={.3cm 0cm .7cm 1.8cm},clip]{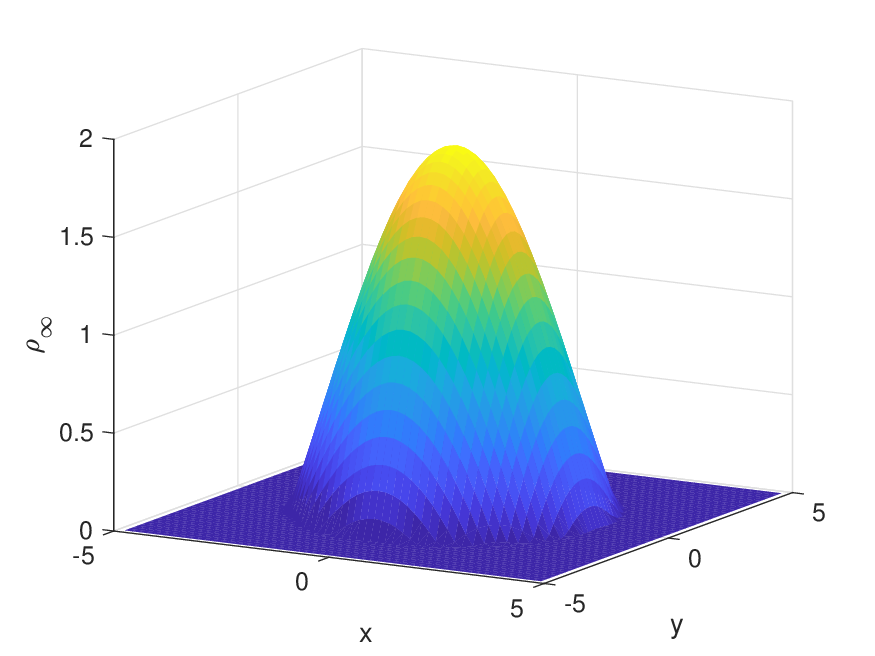}
\includegraphics[width=0.49\textwidth,trim={.3cm 0cm .7cm 1.8cm},clip]{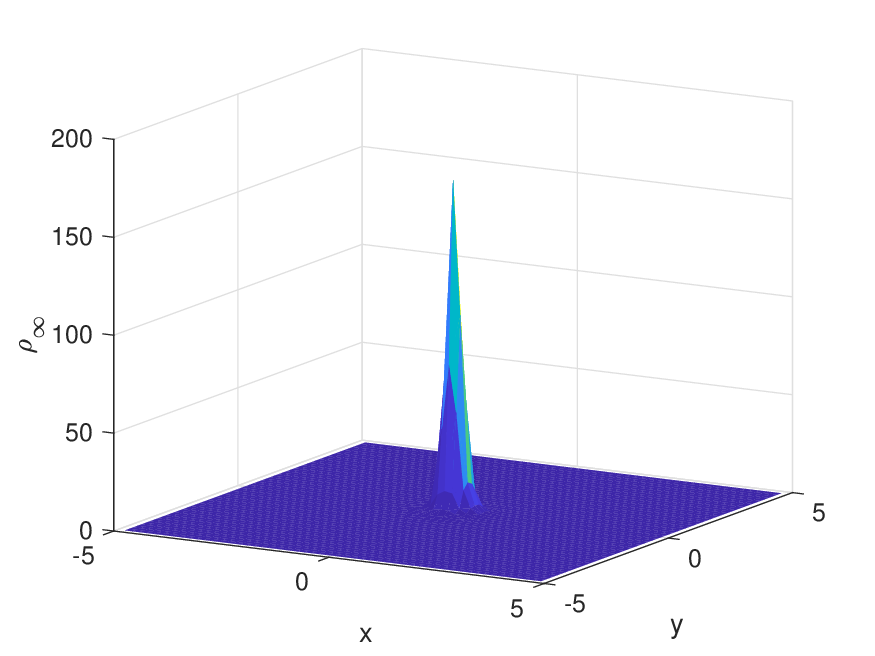}
\caption{Plot of solution $\rho(x,y,t)$ to the two dimensional Keller-Segel equation at $t=2$.  Left: When $m=2$, the solution approaches a bounded, continuous equilibrium profile. Here the computational domain is [-5,5]$\times$[-5,5]. Right: When $m=1$, the solution blows up, becoming sharply peaked. The computational domain here is [-2,2]$\times$[-2,2]. For both we choose $\tau = 0.05$, $ \Delta x = \Delta y = 0.067$, $\Dt = 0.1$, $\lambda = 0.5$, $\sigma = 0.042$.  } \label{fig:KS}
\end{figure}
In Figure \ref{fig:KS}, we simulate solutions of the Keller-Segel equation, which is an aggregation-diffusion equation (\ref{aggdiffeqn}) with a Newtonian interaction kernel, i.e. $W(x)=\frac{1}{2\pi}\text{ln}(|x|)$ in two dimensions for $\nu = 1$ and both $m=1$ and $m=2$, illustrating the role of the diffusion exponent in blowup or global existence of solutions. We choose the initial data to be given by a rescaled gaussian, obtained by multiplying equation (\ref{gaussian}) by a mass $M= 9 \pi$, with mean $\mu = 0$ and variance $\theta = 0.5$. On the left, we take $m =2$ and simulate the steady state of the Keller-Segel equation, which is a single bump. On the right, we simulate the long-time behavior of solutions for $m=1$, in which case we are in the  blow up regime. Indeed, at time $t=2$, we observe the formation of a blowup profile, with the solution becoming sharply peaked at the origin.

\begin{figure}[h!]
\centering
\textbf{Metastability of solutions to Keller-Segel equation} \\ \hspace{-.2cm}
\begin{picture}(155,130)
\put(0,0){\includegraphics[width=0.345\textwidth,trim={.6cm 1.1cm 1.3cm 2.2cm},clip]{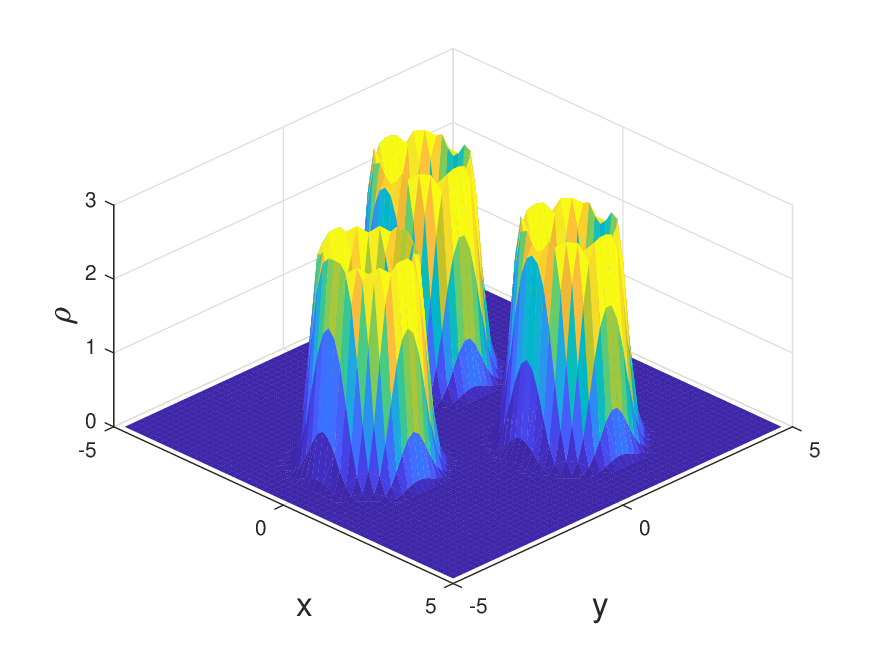}}
\put(72,105){t = 0.00}
\end{picture}
\begin{picture}(155,130)
\put(5,0){\includegraphics[width=0.33\textwidth,trim={1.9cm 1.1cm .7cm 2.2cm},clip]{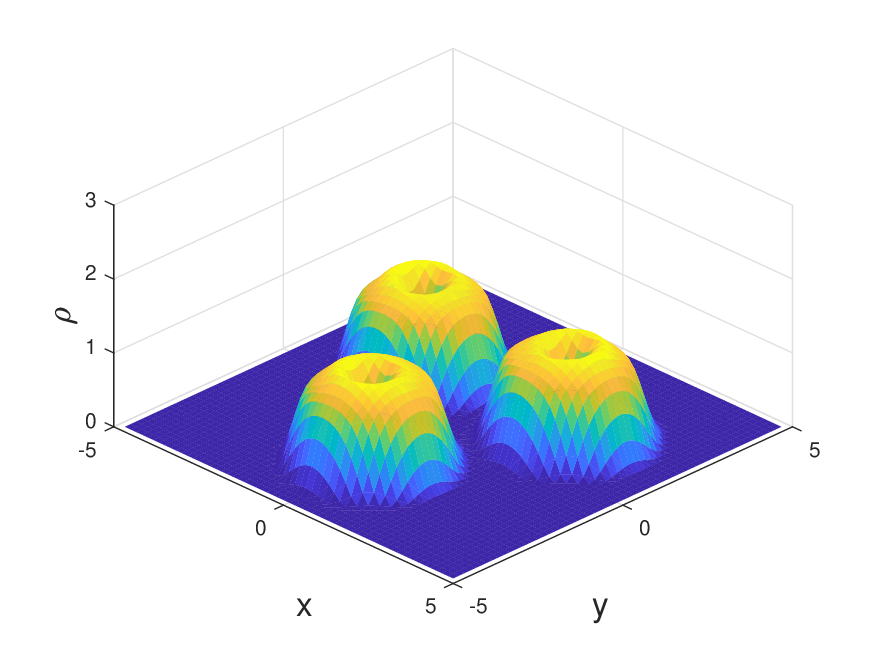}}
\put(63,105){t = 0.10}
\end{picture}
\begin{picture}(155,130)
\put(0,0){\includegraphics[width=0.33\textwidth,trim={1.9cm 1.1cm .7cm 2.2cm},clip]{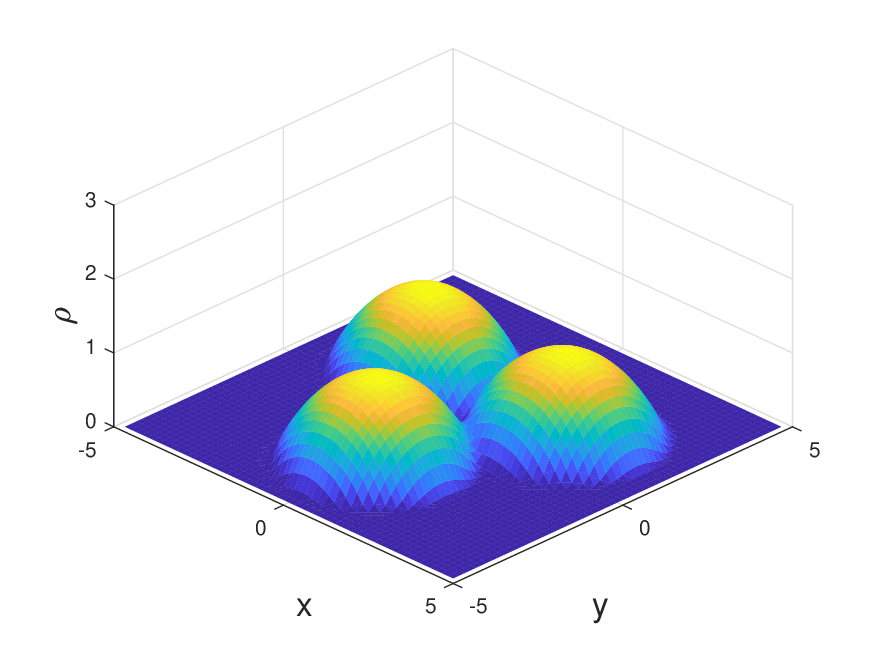}}
\put(58,105){t = 0.25}
\end{picture} \\  \hspace{-.3cm}
\begin{picture}(155,130)
\put(0,0){\includegraphics[width=0.345\textwidth,trim={.6cm 1.1cm 1.3cm 2.2cm},clip]{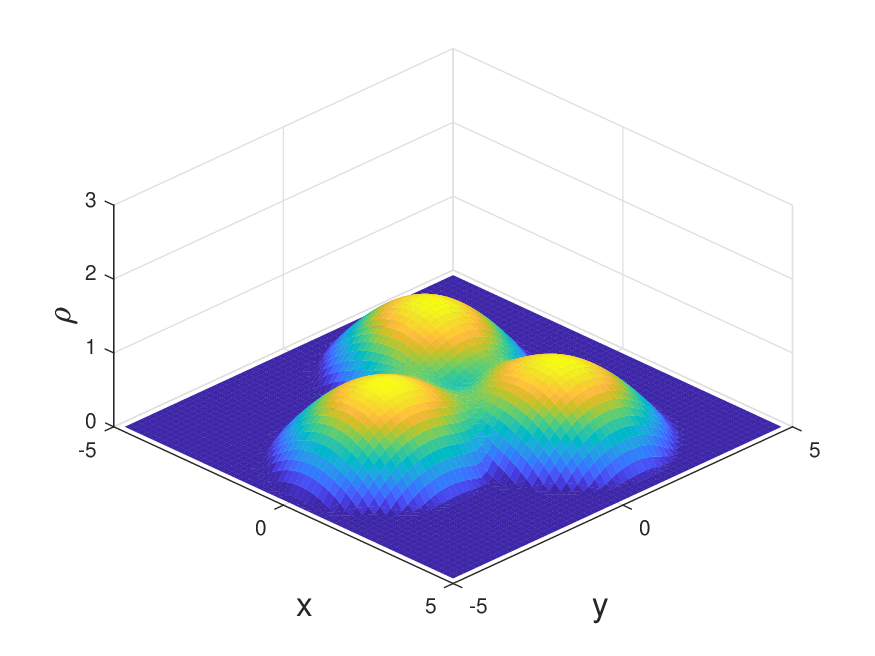}}
\put(72,105){t = 0.50}
\end{picture}
\begin{picture}(155,130)
\put(5,0){\includegraphics[width=0.33\textwidth,trim={1.9cm 1.1cm .7cm 2.2cm},clip]{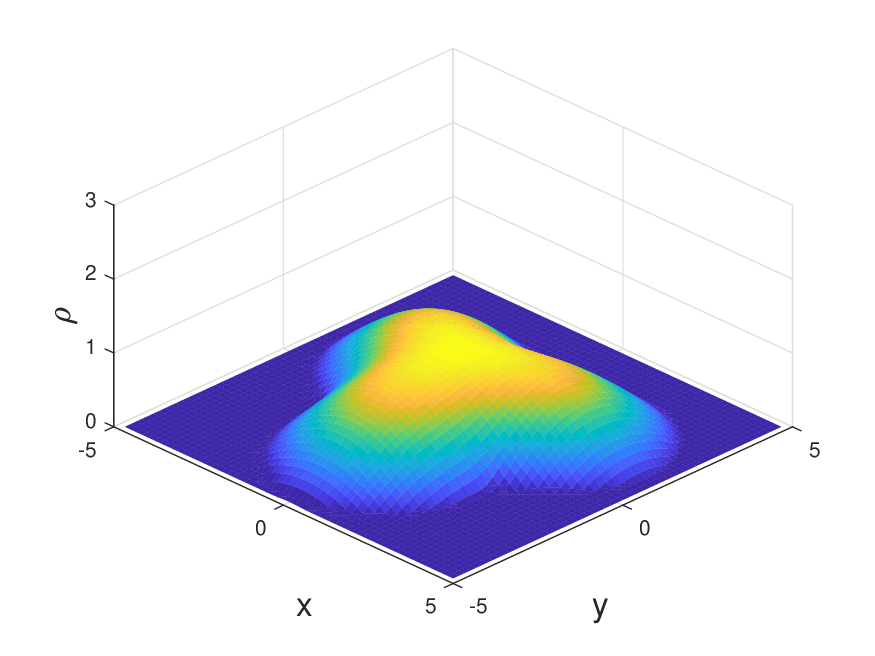}}
\put(63,105){t = 1.00}
\end{picture}
\begin{picture}(155,130)
\put(0,0){\includegraphics[width=0.33\textwidth,trim={1.9cm 1.1cm .7cm 2.2cm},clip]{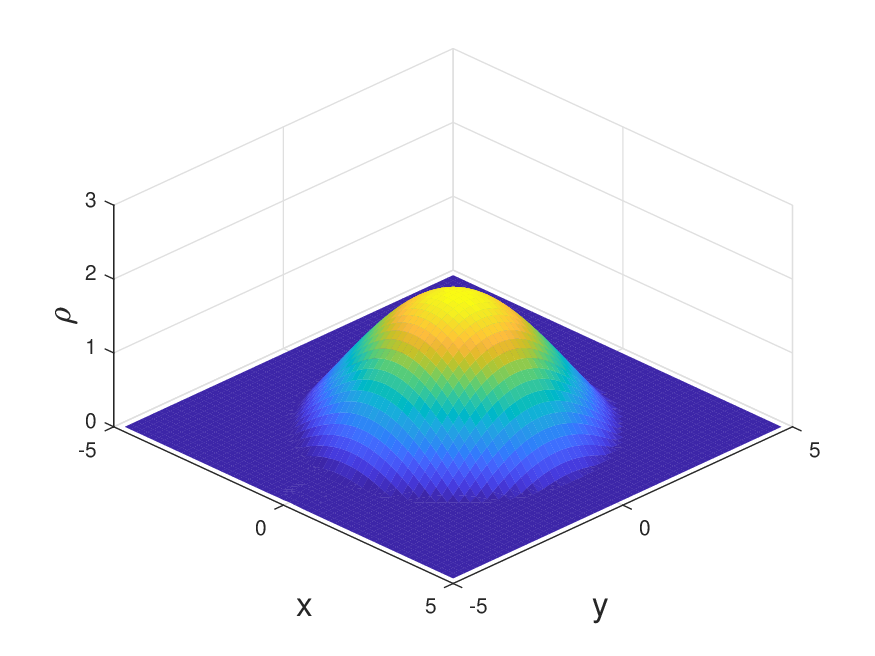}}
\put(58,105){t = 6.00}
\end{picture}
\caption{The evolution of $\rho(x,y,t)$ for the Keller-Segel equation with $U(x)=x^2$. Here $\Delta t = 0.1$, $h_x = h_y = 0.167$, $\tau = 0.05$.} \label{fig:KS3}
\end{figure}
In Figure \ref{fig:KS3}, we again simulate solutions of the Keller-Segel equation with  $m=2$, but in this case we take the initial data to be given by  three localized bumps (Gaussian rings, i.e., the radial cut of the ring is a Gaussian with a center on the circle.) We observe a two-stage evolution in which the each of the bumps converges to a localized quasi-stationary state, and then interact and merge into one single bump in the long time limit. This is a manifestation of the typical metastability phenomena, which is likely present in the majority of the diffusion dominated Keller-Segel models \cite{BFH14, CCH15,CCY}.   

\vspace{0.3cm}
\noindent{\bf Acknowledgement:} The authors would like to thank Ming Yan for fruitful discussions on primal dual methods and preliminary code for the three operator splitting algorithm. They would like to thank Wuchen Li for many helpful conversations. Finally, they would like to thank the anonymous reviewers for their useful observations and suggestions, which greatly improved this work. JAC was supported the Advanced Grant Nonlocal-CPD (Nonlocal PDEs for Complex Particle Dynamics: 	Phase Transitions, Patterns and Synchronization) of the European Research Council Executive Agency (ERC) under the European Union's Horizon 2020 research and innovation programme (grant agreement No. 883363). JAC was also partially supported by the EPSRC grant number EP/P031587/1. KC was supported by NSF DMS-1811012 and a Hellman Faculty Fellowship. LW and CW are partially supported by NSF DMS-1620135, 1903420, 1846854. 

\appendix

\section{Further details of numerical implementation} \label{numericalimplementation}
In this section, we provide explicit formulas for the matrix $\Amat$ and the vectors $u$ and $b$ introduced in section \ref{proxsplitsec}, which play a key role in the implementation of Algorithms \ref{alg:dis}, \ref{alg:nonlinear}, and \ref{alg:nonlinearJKO}. For simplicity, we consider the case of one space dimension, and the discretization takes the form (\ref{crankspace}). The construction of $\tilde{A}$ and $\tilde{b}$ are very similar except a slightly different treatment of $\rho$ at final time.  

Define $N=(N_x+1)  (N_t+1)$.  Let $\Kron$ denote   the Kronecker tensor product,  $\Imat_{N_x+1}$  the identity matrix of size $N_x+1$, and  $(\vec{x})_M$   the column vector in $\mathbb{R}^M$ with all components equal to $x$. Then we define 
\[
{u} = \big[(\vec{\rho}_{.,k})_{k=0}^{N_t} ; ~(\vec{m}_{.,k})_{k=0}^{N_t}\big] \in \RR^N, \quad \vec{\rho}_{.,k}=(\rho_{j,k})_{j=0}^{N_x}, \ \vec{m}_{.,k}=(m_{j,k})_{j=0}^{N_x}
\]
and the matrix $\Amat \in \RR^{M \times 2N}$ takes the form 
\[
\Amat =
\begin{bmatrix}
\Amat_\rho & \vrule & \Amat_m \\ \hline
\Amat_\textrm{mass} & \vrule &  0\\ 
\end{bmatrix}\,.
\]
Here $\Amat_\rho \in \RR^{N\times N}$ reads \[\Amat_\rho = \Dmat_t^{(1)} \otimes \Imat_x^{(1)} + 
\Dmat_t^{(2)} \otimes \Imat_{N_x + 1} := A_\rho^{(1)} + A_\rho^{(2)}\,, \] where $\Dmat_t^{(1)}$, $\Dmat_t^{(2)} \in \RR^{(N_t+1) \times (N_t + 1)}$, and $\Imat_x^{(1)} \in \RR^{(N_x + 1) \times (N_x + 1)} $ are
\[
\Imat_x^{(1)} = 
\begin{bmatrix} 0 &  & \\  & \Imat_{N_x-1} & \\  &  &  0 \end{bmatrix} \,, \quad 
\Dmat_t^{(1)}  =
\begin{bmatrix}
 0    &   & &    &       \\
-1    & 1    &  &            &                          \\
 &    \ddots        & \ddots &  &       \\
                          &            & -1 & 1 &     \\
     &  &  &  & 0  \\
\end{bmatrix} \,,
\quad 
\Dmat_t^{(2)} = 
\begin{bmatrix}
1 &  &  \\ &  {\bf 0} &  \\ &  &  1
\end{bmatrix}\,.
\]
Here $\Dmat_t^{(1)}$ and $\Dmat_t^{(2)}$ correspond to the temporal discretization and initial condition for $\rho$. Likewise, \[\Amat_m \in \RR^{N\times N} = \Bmat_t^{(1)} \otimes \Dmat_x^{(1)} + \Imat_{N_t + 1} \otimes \Dmat_x^{(2)} : = \Amat_m^{(1)} + \Amat_m^{(2)}\,,\] where $\Dmat_x^{(1)}$, $\Dmat_x^{(2)} \in \RR^{(N_x + 1)}$, and $\Bmat_t^{(1)} \in \RR^{(N_t + 1) \times (N_t + 1)}$:
\[
\Dmat_x^{(1)}= \frac{\Delta t}{4\Delta x}
\begin{bmatrix}
 0     &    & &    &       \\
-1    & 0     & 1 &            &                          \\
             & \ddots & \ddots & \ddots &     \\
                          &            & -1 & 0 &  1   \\
     &  &  &  & 0  \\
\end{bmatrix} \,,
\quad
\Dmat_x^{(2)} = 
\begin{bmatrix}
1 &  &  \\ &  {\bf 0} &  \\ &  &  1
\end{bmatrix}\,, \quad
\Bmat_t^{(1)} = 
\begin{bmatrix}
 0    &   & &    &       \\
1    & 1    &  &            &                          \\
 &    \ddots        & \ddots &  &       \\
                          &            & 1 & 1 &     \\
     &  &  &  & 0  \\
\end{bmatrix}\,.
\]
For mass conservation, let $\Smat_\rho = ({\bf \Delta x})_{N_x+1} ^t$, then $\Amat_\text{mass} = \Imat_{N_t + 1} \otimes \Smat_\rho$. In sum, different $\Amat_i$ can be written as
\[
\Amat_1 = \begin{bmatrix}  \Amat_\rho^{(1)}  & \Amat_m ^{(1)} \\  0 & 0 \end{bmatrix}, \quad
\Amat_2 = \begin{bmatrix}  0 & \Amat_m ^{(2)} \\  0 & 0 \end{bmatrix}, \quad 
\Amat_3 = \begin{bmatrix}  0 & 0\\  \Amat_\text{mass} & 0 \end{bmatrix}, \quad
\Amat_{4} (+ \Amat_5) = \begin{bmatrix}  \Amat_\rho^{(2)} & 0\\  0 & 0 \end{bmatrix} \,.
\]
Accordingly, $b \in \RR^{N + N_t + 1 }$ collects all the initial conditions for $\rho$ and boundary conditions for $m$. More specifically, it writes
\begin{eqnarray*}
b &=& [(\vec{\rho}_{.,0}); \vec{0}_{(N_t-1)(N_x+1)}; (\vec{\rho}_{.,N_t+1}); \vec{0}_{N_t+1} ] + [(m_{0,0}; \vec{0};m_{N_x,0}); \cdots ; (m_{0,N_t}; \vec{0};m_{N_x,N_t}); \vec{0}_{N_t + 1}]
\\ && \hspace{12cm}+ [\vec{0}_{N}; \vec{1}_{N_t + 1} ]
\\&:=& b_4 + (b_5) + b_2 + b_3
\end{eqnarray*}
and $b_1 = 0$.

\bibliography{OT_reference}
\bibliographystyle{siam}

\end{document}